\documentclass[12pt]{amsart}
\usepackage{amssymb}
\usepackage{amsmath}
\usepackage{longtable}
\newcommand{\h}{\mathfrak h}
\newcommand\m{\mathfrak m}
\newcommand\n{\mathfrak n}
\newcommand\p{\mathfrak{p}}

\newcommand{\Id}{\mathfrak{Id}}
\newcommand{\q}{\mathfrak{q}}
\newcommand{\z}{\mathfrak{z}}

\newcommand{\Hsh}{\mathfrak{H}}
\newcommand{\Zsh}{{\mathfrak{Z}}}

\newcommand{\Ash}{{\mathfrak{U}}}
\newcommand{\codim}{\operatorname{codim}}
\newcommand{\GL}{\operatorname{GL}}

\newcommand{\id}{\operatorname{id}}
\newcommand{\im}{\operatorname{im}}
\newcommand{\F}{\operatorname{F}}
\newcommand{\VA}{\operatorname{V}}
\newcommand{\HC}{\operatorname{HC}}
\newcommand{\LAnn}{\operatorname{LAnn}}
\newcommand{\RAnn}{\operatorname{RAnn}}
\newcommand\End{\operatorname{End}}
\newcommand{\ad}{\operatorname{ad}}
\newcommand{\rank}{\operatorname{rk}}
\newcommand{\Sp}{\operatorname{Sp}}
\newcommand{\gr}{\operatorname{gr}}
\newcommand\Der{\operatorname{Der}}
\newcommand\Mat{\operatorname{Mat}}
\newcommand\Spec{\operatorname{Spec}}
\newcommand\HH{\operatorname{HH}}

\newcommand\Span{\operatorname{Span}}
\newcommand\Hom{\operatorname{Hom}}
\newcommand\Ann{\operatorname{Ann}}

\newcommand\Bl{\operatorname{Bl}}

\newcommand\Fu{\operatorname{Fun}}
\newcommand\Fl{\operatorname{Fl}}
\newcommand\CC{\operatorname{C}}

\newcommand{\Halg}{{\mathbf{H}}}
\newcommand\W{\mathbf{A}}

\newcommand\cb{{\mathbf c}}
\newcommand{\tb}{{\mathbf{t}}}
\newcommand{\Zalg}{{\mathbf Z}}
\newcommand\Malg{{\mathbf M}}
\newcommand\Nalg{{\mathbf N}}
\newcommand\Ialg{{\mathbf I}}
\newcommand\Jalg{{\mathbf J}}
\newcommand\Aalg{{\mathbf U}}

\newcommand{\I}{\mathcal I}
\newcommand{\J}{\mathcal J}
\newcommand{\M}{\mathcal M}
\newcommand{\DCal}{\mathcal{D}}
\newcommand{\Dcal}{\mathcal{D}}
\newcommand{\Hrm}{\mathcal{H}}
\newcommand\A{\mathcal{A}}
\newcommand\Ncal{\mathcal{N}}

\newcommand\X{\mathcal{X}}
\newcommand\Fun{\mathcal{F}}
\newcommand\Gun{\mathcal{G}}
\newcommand{\Leaf}{\mathcal{L}}
\newcommand{\Zrm}{\mathcal{Z}}
\newcommand\Str{\mathcal{O}}

\newcommand\K{\mathbb{C}}
\newcommand\ZZ{\mathbb{Z}}
\newcommand\N{\mathbb{N}}
\newcommand\Q{\mathbb{Q}}
\newcommand\param{\mathfrak{c}}
\newcommand\Pol{S(\param)}

\newcommand\Res{\operatorname{Res}}
\newcommand\Ind{\operatorname{Ind}}
\newtheorem{Thm}{Theorem}[subsection]
\newtheorem{Prop}[Thm]{Proposition}
\newtheorem{Cor}[Thm]{Corollary}
\newtheorem{Lem}[Thm]{Lemma}
\theoremstyle{definition}

\newtheorem{defi}[Thm]{Definition}
\newtheorem{Rem}[Thm]{Remark}

\numberwithin{equation}{section}
\author{Ivan Losev}
\title{Completions of symplectic reflection algebras}
\thanks{Supported by the NSF grant DMS-0900907}
\thanks{MSC 2010: 16G99, 16S99}
\thanks{Address: Northeastern University, Dept. of Math., 360 Huntington Avenue, Boston MA02115, USA}
\thanks{E-mail: i.loseu@neu.edu}
\begin{document}
\begin{abstract}
In this paper we study the structure of completions of symplectic reflection
algebras. Our results provide a reduction to smaller algebras. We apply this reduction
to the study of two-sided ideals and Harish-Chandra bimodules.
\end{abstract}
\maketitle
\section{Introduction}
\subsection{Setting}\label{SUBSECTION_setting}
The goal of this paper is to describe the structure of completions of symplectic reflection algebras
(SRA) and to apply this description to the study of their two-sided ideals and Harish-Chandra bimodules
and to some other problems, as well.

SRA's were introduced by Etingof and Ginzburg, \cite{EG}. Let us recall their definition.

Our base field is the field $\K$ of complex numbers.
Let $V$ be a vector space  equipped with a symplectic (=non-degenerate skew-symmetric)
form $\omega$ and $\Gamma$ be a finite group of linear symplectomorphisms of $V$.
Recall that an element $\gamma\in \Gamma$ is said to be a {\it symplectic reflection}
if $\rank(\gamma-\id)=2$ (this is the smallest possible rank for a nontrivial element $\gamma\in \Sp(V)$).
For any symplectic reflection $s\in \Gamma$ let $\omega_s$ be the skew-symmetric form
on $V$ defined by
\begin{equation}\label{eq:omega_s}
\begin{split}
& \omega_s(x,y)=\omega(x,y) \text{ if }x,y\in \operatorname{im}(s-\id),\\
& \omega_s(x,y)=0\text{ if }x\text{ or }y\in \ker(s-\id).
\end{split}
\end{equation}
Let $S$ denote the
set of symplectic reflections in $\Gamma$. Note that $\Gamma$ acts on $S$ by conjugation,
let $S_1,\ldots,S_r$ be the orbits. Pick independent variables $\tb, \cb_1,\ldots,\cb_r$
and define $\cb(s)$ to be $\cb_i$ for $s\in S_i$. Below for convenience we will sometimes write $\cb_0$
instead of $\tb$. Let $\param$ denote the vector space with
basis $\cb_0,\ldots,\cb_r$. Consider the symmetric algebra $S(\param)=\K[\param^*]$
of the vector space $\param$. Then define the $\Pol$-algebra $\Halg(:=\Halg(V,\Gamma))$ as the quotient
of the smash-product $$\Pol\otimes_{\K} T(V)\#\Gamma$$
by the relations
\begin{equation}\label{eq:1}
[x,y]=\cb_0\omega(x,y)+\sum_{s\in S} \cb(s)\omega_s(x,y)s,\quad x,y\in V.
\end{equation}
Here $T(V)$ stands for the tensor algebra of $V$ and $T(V)\#\Gamma$ is the {\it smash-product} of $T(V)$ and $\K\Gamma$, i.e., is the algebra
that coincides with $T(V)\otimes \K\Gamma$ as a vector space with the product
$(a_1\otimes g_1)(a_2\otimes g_2)=a_1(g_1.a_2)\otimes g_1g_2$, where $g_1.a_2$ stands
for the image of $a_2$ under the action of $g_1$.  

Theorem 1.3 in \cite{EG} can be interpreted
in the following way.

\begin{Prop}\label{Prop:1.1} $\Halg$ is a flat $\Pol$-algebra.
\end{Prop}

Specializing
$\tb:=t,\cb_i:=c_i$, where $ i=1,\ldots,r,$ and $ t,c_i\in \K,$ we get the $\K$-algebra $\Hrm_{t,c}$. It is clear
that $\Hrm_{0,0}=SV\#\Gamma(=\K[V^*]\#\Gamma)$.
The algebra $\Hrm_{t,c}$ has a natural filtration,
where $\K\Gamma$ has degree 0, while $V$ has degree 1.
Proposition \ref{Prop:1.1} means that the associated graded algebra $\gr \Hrm_{t,c}$ coincides
with $\Hrm_{0,0}$. Also note that for any nonzero $a$ the algebras $\Hrm_{t,c}$ and
$\Hrm_{at,ac}$ are naturally isomorphic. So we get two (essentially) different cases: $t=1$
and $t=0$. One of the most crucial differences is that $\Hrm_{0,c}$ is finite over its center
$\Zrm_c$. Moreover, $\gr\Zrm_c$ is canonically identified with the invariant subalgebra $(SV)^\Gamma\subset SV$, see \cite{EG}, Theorem 3.3. On the other hand, the center
of $\Hrm_{1,c}$ is always trivial, see \cite{BG}, Proposition 7.2.

Let us describe briefly the content of the paper. It is divided into four  sections.
Section \ref{SECTION_completion} describes the structure of  completions of $\Halg$. The most straightforward
version of a completion we are interested in together with the statement of the corresponding result
is given in Subsection \ref{SUBSECTION_completions}. Our starting point here was a result
of Bezrukavnikov and Etingof, \cite{BE}, Theorem 3.2.

In Section \ref{SECTION_HC} we obtain some sort
of ``reduction theorems'' for ideals and Harish-Chandra bimodules of SRA's. Our results regarding ideals
are described in Subsection \ref{SUBSECTION_1c}.  This part is mostly inspired by the author's
work on W-algebras, \cite{Wquant},\cite{HC}.

In Section \ref{SECTION_Application} we give some miscellaneous applications of results of the first
two sections. These applications are described  in Subsection \ref{SUBSECTION_0c}.

Finally, in Section \ref{SECTION_Cherednik} we consider perhaps the most important special class of SRA, the so called
rational Cherednik algebras. We relate our work to that of Bezrukavnikov and Etingof, \cite{BE},
and strengthen some of our results. Also we show that the definition of a Harish-Chandra bimodule
for a rational Cherednik algebra given in \cite{BEG_HC} agrees with ours. Finally, we give a complete classification
of two-sided ideals in the rational Cherednik algebras of type $A$.

In the first subsections of Sections \ref{SECTION_completion},\ref{SECTION_HC},\ref{SECTION_Cherednik} their contents are described in more detail.

\subsection{Completions}\label{SUBSECTION_completions}
We are going to study  completions of the algebra $\Halg$. Let us explain what
kind of completions we are interested in.

Let $\pi$ denote
the quotient map $V^*\rightarrow V^*/\Gamma$. Pick a point $b\in V^*$. Let $I_{b}$ denote the ideal of $\K[V^*]$
generated by the maximal ideal of $\pi(b)$ in $\K[V^*]^{\Gamma}$. Then $I_b\#\Gamma$ is a  two-sided ideal
in $\Hrm_{0,0}=\K[V^*]\#\Gamma$. Let $\m_b$ denote the preimage of $I_b\# \Gamma$
under the canonical epimorphism $\Halg\twoheadrightarrow \Hrm_{0,0}$. This is a two-sided ideal
in $\Halg$. Consider the completion
\begin{equation}\label{eq:2}
\Halg^{\wedge_b}:=\varprojlim_n \Halg/\m_b^n.
\end{equation}
Let us remark that $\Halg^{\wedge_b}$ is a $\K[[\param^*]]$-algebra because $\m_b\cap S(\param)$
coincides with the augmentation ideal $\param S(\param)$.


We want to understand the structure of the algebra $\Halg^{\wedge_b}$ in terms of a similar but simpler algebra.
Define the algebra $\underline{\Halg}$ as the quotient of $\Pol\otimes_{\K} T(V)\# \Gamma_b$ by the relations
\begin{equation}\label{eq:1'}
[x,y]=\cb_0\omega(x,y)+\sum_{s\in S\cap\Gamma_b} \cb(s)\omega_s(x,y)s,\quad x,y\in V.
\end{equation}



We can define the algebra $\underline{\Halg}^{\wedge_0}$ analogously to $\Halg^{\wedge_b}$. In more detail, we can take
the maximal ideal $\underline{I}_0\subset \K[V^*]$ corresponding to the point $0\in V^*$. Then we consider
the preimage $\underline{\m}_0$ of $\underline{I}_0\# \Gamma_{b}$ in $\underline{\Halg}$ and form
the completion $\underline{\Halg}^{\wedge_0}:=\varprojlim_n \underline{\Halg}/\underline{\m}_0^n$.


Now we are ready to state one of  the main results of this paper comparing the completions $\Halg^{\wedge_b}$
and $\underline{\Halg}^{\wedge_0}$.

\begin{Thm}\label{Thm:1}
There is an isomorphism
$$\Theta^{b}: \Halg^{\wedge_b}\rightarrow \Mat_{|\Gamma/\Gamma_b|}(\underline{\Halg}^{\wedge_0})$$
of topological $\K[[\param^*]]$-algebras. Here $\Mat_{|\Gamma/\Gamma_b|}$ stands for the algebra of square matrices
of size $|\Gamma/\Gamma_b|$.
\end{Thm}

In fact, we will prove a much more precise (but also much more technical) statement, Theorem \ref{Thm:2.0I}, which will ``globalize'' $\Theta^b$ and will give an explicit formula for $\Theta^{b}$ ``modulo $\param$''. 

One special case of Theorem \ref{Thm:1} was known before. Namely, Bezrukavnikov and Etingof
proved a similar statement in \cite{BE} for the special class of SRA called {\it rational Cherednik algebras}.
This special class is described as follows.
Suppose that we have a decomposition $V=\h\oplus \h^*$ into the sum of two pairwise
dual $\Gamma$-stable subspaces. The algebra $\Hrm_{t,c}$  in this case is  known as a
rational Cherednik algebra. This case is much simpler than the general one: one can even get an explicit formula for the isomorphism $\Theta^b$ (in fact, for some special points $b$). It seems to be unlikely that one can
find an explicit formula in the general case.

\subsection{General results on two-sided ideals}\label{SUBSECTION_1c}
We will apply Theorem \ref{Thm:1} (more precisely, its enhanced version, Theorem \ref{Thm:2.0I})
to the study of ideals and Harish-Chandra bimodules of the algebras $\Halg, \Hrm_{1,c},\Zrm_c$.
We remark that both our results
and our proofs are inspired by those   for W-algebras, see \cite{Wquant},\cite{HC}. The definition of a Harish-Chandra
bimodule is a bit technical so we postpone it (together with the statement of the corresponding result)
until Subsection \ref{SUBSECTION_HC_bimod}. The main result for Harish-Chandra bimodules
is Theorem \ref{Thm:5}.

Now let us describe our results on ideals.
Consider a two-sided ideal $\J\subset \Hrm_{1,c}$. Then $\Hrm_{1,c}/\J$ has a natural filtration. It is easy to see
that the actions of $(SV)^{\Gamma}$ on $\gr (\Hrm_{1,c}/\J)$ by left and by right multiplications coincide.  Let $\VA(\Hrm_{1,c}/\J)\subset V^*/\Gamma$ denote the support of the $(SV)^\Gamma$-module $\gr (\Hrm_{1,c}/\J)$. One can show (Subsection \ref{SUBSECTION_HC_bimod}) that $\VA(\Hrm_{1,c}/\J)$ is a Poisson subvariety
in $V^*/\Gamma$.  Similarly, for two two-sided ideals $\J_1\subset\J_2\subset\Hrm_{1,c}$ we can define
$\VA(\J_1/\J_2)\subset V^*/\Gamma$ and it is also a Poisson subvariety.

The Poisson variety $V^*/\Gamma$ has finitely many symplectic leaves. The leaves can be described as follows, see, for example, \cite{BG}, Subsection 7.4. Consider the
set of all conjugacy classes of subgroups in $\Gamma$ that are stabilizers of elements of $V$ (equivalently, of $V^*$).
Then there is a bijection between this set and the set of symplectic leaves in $V^*/\Gamma$. Namely, given
a conjugacy class, pick a subgroup $\underline{\Gamma}\subset \Gamma$ in it. Then the corresponding symplectic leaf
coincides with the $\pi\left((V_0^*)^{fr}\right)$, where $\pi:V^*\rightarrow V^*/\Gamma$
denotes the quotient morphism, $V_0:=V^{\underline{\Gamma}},(V_0^*)^{fr}:=\{v\in V^*| \Gamma_v=\underline{\Gamma}\}$.

Fix a symplectic leaf $\Leaf\subset V^*/\Gamma$ and let $\underline{\Gamma}\subset \Gamma$ be a  subgroup
in the conjugacy class of subgroups corresponding to $\Leaf$.
We  consider the set $\Id_\Leaf(\Hrm_{1,c})$ consisting of all two-sided
ideals $\J\subset \Hrm_{1,c}$ such that $\VA(\Hrm_{1,c}/\J)=\overline{\Leaf}$.
We will relate $\Id_\Leaf(\Hrm_{1,c})$ to a certain set of ideals of finite codimension in
a ``smaller'' SRA to be specified in a moment.

There is a unique $\underline{\Gamma}$-stable decomposition
$V=V_0\oplus V_+$. Let $\underline{\Halg}^+$ be the quotient of $S(\param)\otimes
TV_+\#\underline{\Gamma}$ modulo the relations (\ref{eq:2}). In particular, we have the decomposition
$\underline{\Halg}=\W_{\tb}(V_0)\otimes_{\K[\tb]}\underline{\Halg}^+$. Here
$\W_{\tb}(V_0)$ is the homogeneous Weyl algebra of the symplectic vector
space $V_0$, in other words, $\W_{\tb}(V_0):=\Halg(V_0,\{1\})$.

Note that the group $\widetilde{\Xi}:=N_\Gamma(\underline{\Gamma})$ acts on $\underline{\Halg}^+$
by automorphisms, for $V_+\subset V$ is $\widetilde{\Xi}$-stable. The action
of $\underline{\Gamma}\subset \widetilde{\Xi}$ coincides with that by inner automorphisms.
So we can consider the set $\Id_0(\underline{\Hrm}^+_{1,c})$ of two-sided
ideals of $\underline{\Hrm}_{1,c}^+$  of finite codimension, and its subset
$\Id_0^{\Xi}(\underline{\Hrm}^+_{1,c})$ consisting
of all $\Xi:=\widetilde{\Xi}/\underline{\Gamma}$-stable ideals (in fact, $\widetilde{\Xi}$
acts on $\Id(\underline{\Hrm}^+_{1,c})$ via its quotient $\Xi$).

We have the following result that is completely parallel to Theorem 1.2.2  from \cite{HC}.

\begin{Thm}\label{Thm:4}
There are maps $\Id_{\Leaf}(\Hrm_{1,c})\rightarrow \Id_0^{\Xi}(\underline{\Hrm}_{1,c}^+),\J\mapsto
\J_\dagger$, and $\Id_0(\underline{\Hrm}^+_{1,c})\rightarrow \Id_{\Leaf}(\Hrm_{1,c}),$ $
\I\mapsto \I^{\ddag},$ with the following properties:
\begin{enumerate}
\item The image of  $\J\mapsto \J_\dagger$ coincides with $\Id_0^{\Xi}(\underline{\Hrm}_{1,c}^+)$.
\item $\J\subset (\J_\dagger)^\ddag$ and $\I\supset (\I^\ddag)_\dagger$ for all $\J\in
\Id_{\Leaf}(\Hrm_{1,c}), \I\in \Id_0(\underline{\Hrm}_{1,c}^+)$.
\item We have $\VA((\J_\dagger)^{\ddag}/\J)\subset \partial\Leaf:=\overline{\Leaf}\setminus \Leaf$.
\item Consider the restriction of the map $\I\mapsto \I^{\ddag}$ to the set of all primitive (=maximal)
ideals in $\Id_0(\underline{\Hrm}_{1,c}^+)$. The image of this restriction is the set of all primitive
ideals in $\Id_{\Leaf}(\Hrm_{1,c})$ and each fiber  is a single $\Xi$-orbit.
\end{enumerate}
\end{Thm}

Recall that a two-sided ideal $I$ in  an associative unital algebra $A$ is called {\it primitive}
if it is the annihilator of a simple module.

In \cite{Ginzburg_irr} Ginzburg proved that any primitive ideals of $\Hrm_{1,c}$ is contained in
some set $\Id_{\Leaf}(\Hrm_{1,c})$. This is an analog of the Joseph irreducibility
theorem, \cite{Joseph}, from the representation theory of universal enveloping algebras.
We will rederive Ginzburg's result in the present paper.
So Theorem \ref{Thm:4} allows one to reduce the study (in particular, the classification)
of  primitive ideals in  symplectic reflection algebras
to the study of the annihilators of  irreducible
finite dimensional representations.

We also have an analog of Theorem \ref{Thm:4} for $t=0$.
Recall that $\Zrm_{c}$ denotes the
center of $\Hrm_{0,c}$. It was shown by Etingof and Ginzburg in \cite{EG} that $\Zrm_{c}$ has a natural
Poisson bracket $\{\cdot,\cdot\}$ whose definition will be recalled in Subsection \ref{SUBSECTION_spherical}.
Recall that $\Zrm_{c}$ is a filtered
algebra, let $\F_i\Zrm_{c}$ denote the corresponding filtration. Then
$\{\F_i\Zrm_{c},\F_j\Zrm_{c}\}\subset \F_{i+j-2}\Zrm_{c}$ and  the induced
Poisson bracket on $\Zrm_{0}=(SV)^\Gamma$ coincides with the restriction of the bracket
from $SV$.

Again fix a symplectic leaf $\Leaf\subset V^*/\Gamma$. Consider the set $\Id_{\Leaf}(\Zrm_{c})$ consisting
of all {\it Poisson} ideals $\J\subset \Zrm_{c}$ whose associated variety (= the variety
of zeros of $\gr\J$) coincides  with $\overline{\Leaf}$. Recall that an ideal $J$
in a Poisson algebra $A$ is called Poisson if $\{A,J\}\subset J$.

 As above, we can form the subgroups $\underline{\Gamma},\widetilde{\Xi}\subset \Gamma$,
 and the algebras $\underline{\Hrm}_{0,c},
\underline{\Zrm}_{c},\underline{\Zrm}_c^+$. Note that
\begin{equation}\underline{\Hrm}_{0,c}=SV_0\otimes \underline{\Hrm}^+_{0,c},\end{equation}
\begin{equation}\label{eq:1.22}\underline{\Zrm}_{c}=SV_0\otimes \underline{\Zrm}^+_{c},\end{equation}
 (\ref{eq:1.22}) is an equality
of Poisson algebras, the Poisson structure on $SV_0$ is induced from the symplectic
form on $V_0$. Let $\Id^{\Xi}_0(\underline{\Zrm}^+_{c})$ denote the set of all
$\Xi$-stable Poisson ideals of finite codimension in $\underline{\Zrm}^+_c$.


The following theorem is, in a sense, a quasiclassical analogue of Theorem \ref{Thm:4}.

\begin{Thm}\label{Thm:4'}
There are maps $\Id_{\Leaf}(\Zrm_{c})\rightarrow \Id_0^{\Xi}(\underline{\Zrm}_{c}^+),\J\mapsto
\J_\dagger$, $\Id_0(\underline{\Zrm}^+_{c})\rightarrow \Id_{\Leaf}(\Zrm_{c}),
\I\mapsto \I^{\ddag},$ with the following properties:
\begin{enumerate}
\item The map $\J\mapsto \J_\dagger$ is surjective.
\item $\J\subset (\J_\dagger)^\ddag$ and $\I\supset (\I^\ddag)_\dagger$ for all $\J\in
\Id_{\Leaf}(\Zrm_{c}), \I\in \Id_0(\underline{\Zrm}_{c}^+)$.
\item We have $\VA((\J_\dagger)^{\ddag}/\J)\subset \partial\Leaf$.
\item Consider the restriction of the map $\I\mapsto \I^\ddag$ to the set of all prime (=maximal)
ideals in $\Id_0(\underline{\Zrm}_{c}^+)$. The image of this restriction is the set of all prime
ideals in $\Id_{\Leaf}(\Zrm_{c})$ and each fiber  is a single $\Xi$-orbit.
\end{enumerate}
\end{Thm}

Let us make a remark regarding a geometric interpretation of part (4). Maximal ideals in $\underline{\Zrm}^+_c$
that are Poisson are in one-to-one correspondence with zero-dimensional symplectic leaves of the variety
$\underline{C}^+_c:=\Spec(\underline{\Zrm}^+_c)$. On the other hand, the set of prime ideals in $\Id_{\Leaf}(\Zrm_{c})$ can be described as follows. Consider the variety $\Spec(R_\hbar(\Zrm_{c}))$. Here and below $R_\hbar(A)$
denotes the Rees algebra of a filtered algebra $A$. In more detail, if $\F_iA$ is an increasing filtration
on $A$, then, by definition, $R_\hbar(A):=\bigoplus_i \hbar^i \F_i A\subset A[\hbar^{-1},\hbar]$.

The inclusion $\K[\hbar]\hookrightarrow R_\hbar(\Zrm_c)$ gives rise to the dominant morphism  $\Spec(R_\hbar(\Zrm_c))\rightarrow\K$. The preimage of $\K^\times$ is naturally identified with
$\Spec(\Zrm_c)\times \K^\times$, while the preimage of 0 is $\Spec(\Zrm_0)=V^*/\Gamma$. The set  of prime ideals
in $\Id_{\Leaf}(\Zrm_c)$ is identified with the set of all symplectic leaves $Y$ such that
the closure of $Y\times \K^\times$ in $\Spec(R_\hbar(\Zrm_{c}))$ intersects $V^*/\Gamma$ exactly in
$\overline{\Leaf}$. As Martino checked in \cite{Martino}, for any symplectic leaf in $\Spec(\Zrm_{c})$
its ideal lies in $\Id_{\Leaf}(\Zrm_{c})$ for some $\Leaf$ (our techniques allows
to give an alternative proof of this result).

\subsection{Applications}\label{SUBSECTION_0c}
Our first result concerns the structure of $\Hrm_{0,c}$  as an algebra over its center
$\Zrm_c$. More precisely, since $\Hrm_{0,c}$ is finite over $\Zrm_{c}$, we can consider $\Hrm_{0,c}$ as the algebra of global sections of an appropriate coherent sheaf of  algebras over $\Spec(\Zrm_c)$. Let us describe the fibers of the restriction of this sheaf to a symplectic leaf.

\begin{Thm}\label{Thm:3}
Let $\hat{\Leaf}$ be a symplectic leaf of $\Spec(\Zrm_c)$ and $\q\subset \Zrm_c$ the prime ideal defining the closure of
$\hat{\Leaf}$. Let $\Leaf$ be a unique symplectic leaf  of $V^*/\Gamma$ such that $\q\in \Id_{\Leaf}(\Zrm_c)$.
Pick a point $y\in \hat{\Leaf}$, let $\n_y$ be its maximal ideal in $\Zrm_c$. Finally, let $\underline{\n}$ be a maximal
ideal of $\underline{\Zrm}^+_c$ containing $\q_\dagger$ (by Theorem \ref{Thm:4'}
all such ideals are $\Xi$-conjugate). Then $\Hrm_{0,c}/\Hrm_{0,c}\n_y$ is isomorphic to
$\Mat_{|\Gamma/\underline{\Gamma}|}(\underline{\Hrm}^+_{0,c}/\underline{\Hrm}^+_{0,c}\underline{\n})$.
\end{Thm}


While this paper was in preparation, a proof in the case of rational Cherednik algebras
appeared, \cite{Bellamy}. The proof
was based on the Bezrukavnikov-Etingof theorem on isomorphisms of  completions. Also note that Theorem
\ref{Thm:3} is similar to the Kac-Weisfeiler conjecture on modular Lie algebras proved
by Premet, \cite{Premet}, and to the De Concini-Kac-Procesi conjecture on quantized universal enveloping algebras
at roots of unity proved by Kremnizer, \cite{Kremnitzer} (his proof is rather sketchy
and imposes some restrictions on  the central character).

Together Theorems \ref{Thm:4'},\ref{Thm:3} allow one to reduce the study of representations of
the algebras $\Hrm_{0,c}/\Hrm_{0,c}\n_y$ to the case when $y$ is a (zero dimensional) symplectic leaf.

Our second application is the following theorem which was communicated to us
(essentially with a the proof) by P. Etingof.

\begin{Thm}\label{Thm:6}
For general $c$ the algebra $\Hrm_{1,c}$ is simple.
\end{Thm}

The meaning of a ``general $c$'' will be made precise in Subsection \ref{SUBSECTION_simplicity}.


\subsection{Conventions and notation} In this subsection we gather some conventions and
notation we use.
Each of Sections \ref{SECTION_completion},\ref{SECTION_HC} contains its own list of conventions and notation,
Subsections \ref{SUBSECTION_completion_not}, \ref{SUBSECTION_HC_not}.

{\bf Sheaves.} Let $X$ be an algebraic variety and $\mathcal{S}$ be a sheaf (of abelian groups) on $X$.
For an open subset $U\subset X$ we denote by $\mathcal{S}(U)$ the group of  sections
of $\mathcal{S}$ on $U$. For a morphism $Y\rightarrow X$ we denote by $\mathcal{S}|_Y$ the sheaf-theoretic
pull-back of $\mathcal{S}$ to $Y$.

In this paper we mostly consider quasi-coherent and pro-coherent sheaves. By a pro-coherent sheaf we mean
a sheaf $\mathcal{S}$  of vector spaces (algebras, modules) on $X$ admitting a decreasing  filtration $\F_i\mathcal{S}$
(by subspaces, two-sided ideals, submodules) such that
$\mathcal{S}$ is complete with respect to this filtration, i.e.,  $\mathcal{S}=\varprojlim_{i}
\mathcal{S}/\F_i \mathcal{S}$, and $\F_i \mathcal{S}/\F_{i+1}\mathcal{S}$ is a coherent  $\Str_X$-module.

The following table contains the list of some standard notation used in the paper.

\begin{longtable}{p{2.5cm} p{13.5cm}}
$\widehat{\otimes}$&completed tensor product of complete topological vector spaces/modules.\\
$(V)$& the two-sided ideal in an associative algebra generated by a subset $V$.\\
$\Der(A)$& the Lie algebra of derivations of an algebra $A$.
\\
$G_x$& the stabilizer of $x$ in a group $G$.\\
$\gr \A$& the associated graded vector space of a filtered
vector space $\A$.\\
$R_\hbar(\A)$&$:=\bigoplus_{i\in
\mathbb{Z}}\hbar^i \F_i\A$ ,the Rees vector space of a filtered
vector space $\A$.\\
$R\Gamma$& the group algebra of a group $\Gamma$ with coefficients in a ring $R$.\\
$\VA(\M)$& the associated variety of $\M$.\\
\end{longtable}

{\bf Acknowledgements.} This paper would never have appeared without  P. Etingof's help
that is gratefully acknowledged. Apart from  inspiration and many fruitful ideas,
Etingof contributed several results to the present paper, including Theorem \ref{Thm:6}.
Also I would like to thank G. Bellamy, R. Bezrukavnikov, A. Braverman, I. Gordon and A. Premet for
stimulating discussions and for useful remarks on the previous versions of this paper.
Finally, I would like to thank the referee for numerous remarks that allowed me to improve the exposition.

\section{Completion theorem}\label{SECTION_completion}
\subsection{Structure of this section}
In this section we will study  completions of the algebra $\Halg$. The main result
of this section is Theorem \ref{Thm:2.0I}, which is an enhanced (in particular, ``globalized'') version
of Theorem \ref{Thm:1}.

In Subsection \ref{SUBSECTION_completion_not} we explain conventions and some notation used in the present
section. In Subsection \ref{SUBSECTION_centralizer} we gather various results on the centralizer construction
introduced in \cite{BE}. This construction provides a convenient language for our considerations.

The next two subsections, \ref{SUBSECTION_sheafifiedI},\ref{SUBSECTION_completion_I} are devoted to our principal
result, Theorem \ref{Thm:2.0I}. The theorem deals with certain sheaves on  symplectic leaves.
The sheaves of interest are ``localized completions'' of the algebras $\Halg, \Mat_{|\Gamma/\Gamma_b|}(\underline{\Halg})$.
These sheaves are defined in Subsection \ref{SUBSECTION_sheafifiedI}.
Theorem \ref{Thm:2.0I} claims that there is an isomorphism with some special properties
between (certain twists of) these sheaves.

The proof of Theorem \ref{Thm:2.0I} is rather indirect. As it happens, it is more convenient to work
with another version of sheafified completions. This version is introduced in Subsection \ref{SUBSECTION_sheafified_algebras}. For instance, on the $\Halg$ side we have a pro-coherent sheaf $\Hsh^\wedge$
on a symplectic leaf, whose fiber at $b$ is $\Halg^{\wedge_b}$. The sheaves under consideration  come
equipped with  flat connections. In Subsection \ref{SUBSECTION_completion_II} we state
an isomorphism theorem for the sheaves with flat connections of interest, Theorem \ref{Thm:2.0}.

In Subsection \ref{SUBSECTION_spherical} we mostly recall some basic and standard facts about the spherical subalgebras
and the centers of the SRA's. Subsections \ref{SUBSECTION_compatibility},\ref{SUBSECTION_flatness} are quite technical.
Their goal is to make sure that in our setting we have nice properties of completions that are standard
in the commutative situation.

The proof of Theorem \ref{Thm:2.0} occupies two subsections, \ref{SUBSECTION_derivations},\ref{SUBSECTION_ProofII}.
In the former we study derivations of algebras of sections of $\Hsh^\wedge$. The latter subsection
completes the proof.

Theorem \ref{Thm:2.0I}
is deduced from Theorem \ref{Thm:2.0} in the last Subsection \ref{SUBSECTION ProofI} of the section. The idea, roughly speaking, is that to get Theorem \ref{Thm:2.0I} from Theorem \ref{Thm:2.0} we just need to pass
to flat sections of the sheaves considered in the latter.

\subsection{Notation and conventions}\label{SUBSECTION_completion_not}
Let us introduce some notation to be used throughout the section.

{\bf Varieties and morphisms.}
Set $\X:=V^*/\Gamma$. This is a Poisson variety with finitely many symplectic leaves. Let $\pi:V^*\rightarrow \X$
denote the quotient morphism.
Let $\Leaf$ be a symplectic leaf of $\X$. Pick a point $b\in V^*$ with $\pi(b)\in \Leaf$.
Set $\underline{\Gamma}:=\Gamma_b, V_0:=V^{\underline{\Gamma}}$. Then $\overline{\Leaf}=\pi(V_0^*)$.
Set  $ \underline{\Leaf}:= V_0^{*}\cap \pi^{-1}(\Leaf), \underline{\X}:=V^*/\underline{\Gamma}$, then we can consider $\underline{\Leaf}$ as a subvariety in $\underline{\X}$. We remark that $\underline{\Leaf}$ is open in $V_0^*$.

Let $\underline{\pi}$ denote the quotient morphism $V^*\rightarrow \underline{\X}$
and $\pi':\underline{\X}\rightarrow \X$ be the natural morphism so that
$\pi=\pi'\circ \underline{\pi}$. Let $\underline{\X}^0$ be the open subset of $\underline{\X}$
consisting of all points where $\pi'$ is \'{e}tale. In other words,
$\underline{\X}^0=\underline{\pi}(\{v\in V| \Gamma_v\subset \underline{\Gamma}\})$. Further,
set $\X^0=\pi'(\underline{\X}^0)$.

It is clear that $\Leaf$ is a closed subvariety in $\X^0$. Let $\X^{\wedge}_{\Leaf}$ denote the formal neighborhood of $\Leaf$ in $\X^0$. Similarly,   $\underline{\X}^{\wedge}_{\underline{\Leaf}}$ denotes the formal neighborhood
of $\underline{\Leaf}$ in $\underline{\X}^0$.

{\bf Groups.}
We set $\widetilde{\Xi}:=N_\Gamma(\underline{\Gamma}), \Xi:=\widetilde{\Xi}/\underline{\Gamma}$.
We remark that $\Xi$ acts naturally on $\underline{\X}$. Moreover, the restriction
of $\pi'$ to $\underline{\Leaf}$ is a covering $\underline{\Leaf}\rightarrow \Leaf$ with fundamental
group $\Xi$.

{\bf Algebras.}

Recall the algebras $\Halg,\underline{\Halg}, \W_{\tb},\underline{\Halg}^+$ defined in the Introduction.

Let $\param^{(i)}$ (resp., $\param_{(i)}$), $i=0,\ldots,r+1,$ denote the subspace in $\param$
spanned by $\cb_0,\ldots,\cb_{i-1}$ (resp., by $\cb_{i},\ldots,\cb_r)$.
We write $\Halg_{(l)},l=0,\ldots,r+1,$ for the quotient
$\Halg/(\param^{(l)})$ (so $\Halg_{(0)}=\Halg, \Halg_{(r+1)}=SV\#\Gamma$).
We remark that $\Halg_{(l)}$ is a $S(\param_{(l)})$-algebra. Let  $\rho_l, l=0,\ldots,r,$ denote the natural projection $ \Halg_{(l)}\rightarrow \Halg_{(l+1)}$, set $\hat{\rho}=\rho_r\circ\ldots\circ\rho_0:\Halg\rightarrow \Halg_{(r+1)}$. The analogous projections for $\underline{\Halg}_{(l)}$ are denoted by $\underline{\rho}_l,\underline{\hat{\rho}}$.

{\bf Centers.}
Let $A$ be an associative unital algebra. We denote the center of $A$ by
$\z(A)$. Now suppose $A_\tb$ is an associative $\K[\tb]$-algebra.
By $\z^\tb(A_\tb)$ we denote the inverse image of $\z(A_\tb/(\tb))$ in $A_\tb$
under the natural projection.
In other words, $\z^\tb(A_\tb)$ consists of all elements $a\in A_\tb$ such that
$[a,b]\in  (\tb)$ for all $b\in A_\tb$.

\subsection{The centralizer construction}
\label{SUBSECTION_centralizer}
In this subsection we will recall the centralizer construction of Bezrukavnikov and Etingof, \cite{BE}.

Let $G\supset H$ be finite groups and $A$ be an algebra containing $\K H$. Set
\begin{equation*} \Fu_H(G,A):=\{f:G\rightarrow A| f(hg)=hf(g), \forall g\in G,h\in H\}.\end{equation*} Clearly, $\Fu_H(G,A)$ is a free right $A$-module of rank $|G/H|$.  By definition, the {\it centralizer algebra} $Z(G,H,A)$ is $\operatorname{End}_A(\Fu_H(G,A))$. In particular, $Z(G,H,A)$ is isomorphic to the matrix
algebra $\operatorname{Mat}_{|G/H|}(A)$. An isomorphism depends  on a choice of
representatives in the cosets from $H\backslash G$, and there is no canonical isomorphism.
We have a monomorphism $\K G\hookrightarrow Z(G,H,A)$ given by $(g.f)(g_1)=f(g_1g), f\in \Fu_H(G,A), g,g_1\in G$.
Also we have an embedding of $A^H$  into $Z(G,H,A)$: $A^H$ acts on $\Fu_H(G,A)$ by $(a.f)(g_1)=af(g_1)$. Under a choice of identification
$Z(G,H,A)\cong \operatorname{Mat}_{|G/H|}(A)$ as explained above, $A^H$ is identified with the subalgebra consisting
of all matrices of the form $\operatorname{diag}(a,a,\ldots,a)$, where $a\in A^H$. In particular, the center
$Z$ of $A$ is contained in $A^H$. So we get
the embedding of $Z$ into $Z(G,H,A)$ that identifies $Z$ with the center of $Z(G,H,A)$.

Now let $M$ be an $A$-bimodule. We want to construct a $Z(G,H,A)$-bimodule $Z(G,H,M)$ from $M$.
Set $$Z(G,H,M):=\Hom_{H\times H^{op}}(G\times G, M)=\Fu_H(G,A)\otimes_A M\otimes_A \Hom_A(\Fu_H(G,A),A).$$
The notation $H^{op}$ in the subscript indicates that the second copy of $H$ acts on the second copy of $G$
from the right.
If we choose representatives in the right cosets $H\backslash G$ and identify
$Z(G,H,A)$ with $\Mat_{|G/H|}(A)$, the $Z(G,H,A)$-bimodule $Z(G,H,M)$ gets identified
with $\Mat_{|G/H|}(M)$.

Let us explain how to recover $A$ from $Z(G,H,A)$ and $M$ from $Z(G,H,M)$.
Pick a left coset $x\in H\backslash G$. Define  $e(x)\in \End(\Fu_H(G,A))$ in the following way:
\begin{equation}\label{eq:1.1}
[e(x).f](g)=\begin{cases} f(g), &\text{ if }g\in x,\\
0,&\text{ else.}
\end{cases}
\end{equation}

Clearly, $e(x)\in Z(G,H,A)$ and $e(x)^2=e(x)$. Moreover, \begin{equation}\label{eq:0.0.new}\sum_{x\in H\backslash G}e(x)=1,\quad
g e(x) g^{-1}=e(xg).\end{equation} Further, we see that
\begin{equation}\label{eq:0.0.11}
a e(x)=e(x) a, \forall a\in A^H.
\end{equation}

We have a natural isomorphism $A\cong Z(H,H,A)=e(H)Z(G,H,A)e(H)$. On the other hand, $Z(G,H,A)e(H)$ is naturally
identified with $\Fu_H(G,A)$ as a $Z(G,H,A)$-$A$-bimodule. Namely, to $\varphi\in Z(G,H,A)e(H)$ we assign the map
$g\mapsto e(H)g\varphi\in \Fu_H(G,A)$. Therefore the bimodule $Z(G,H,A)e(H)$ gives rise to a Morita
equivalence between $A$ and $Z(G,H,A)$.
Thus the assignment $N\mapsto e(H)N e(H)$ is a quasi-inverse equivalence
to $M\mapsto Z(G,H,M)$.

\begin{Lem}\label{Lem:0.0.12}
Let $B$ be a unital associative algebra and $\iota:Z(G,H,\K H)\hookrightarrow B$
be an algebra monomorphism (mapping $1$ to $1$).
Then $B$ is naturally identified with $Z\left(G,H, \iota(e(H))B\iota(e(H))\right)$.
\end{Lem}
\begin{proof}
Set $e:=\iota(e(H))$ for brevity.
Let us identify $Be$ with $\Fu_H(G,eBe)$. To $\varphi\in Be$ we assign the map $g\mapsto eg\varphi$. An inverse
map is given by $f\mapsto \frac{1}{|H|}\sum_{g\in G}g^{-1}f(g)$. The claim that these maps are mutually inverse
follows from (\ref{eq:0.0.new}).

This defines an action of $B$ on $\Fu_H(G, eBe)$ by right $eBe$-module
endomorphisms and hence gives rise to a homomorphism $B\rightarrow Z(G,H, eBe)$. To check that this homomorphism
is an isomorphism we need to show that the bimodule $Be$ produces a Morita equivalence between
$B$ and $eBe$, equivalently, that $1\in BeB$. This stems from the inclusion $1\in Z(G,H,\K H)e(H) Z(G,H,\K H)$.
\end{proof}

Let us discuss the compatibility of the centralizer construction with certain group actions.
Suppose that $\widetilde{H}$
is a subgroup of $G$ containing $H$ and that $A$ is equipped with an action of $\widetilde{H}$ by automorphisms,
whose restriction to $H$ coincides with the adjoint action of  $H\subset A$. Further, suppose that $H$
is normal in $\widetilde{H}$. Then $\widetilde{H}$ acts on $\Fu_H(G,A)$ by \begin{equation}\label{eq:1.2}(\widetilde{h}.f)(g)=\widetilde{h}.f(\widetilde{h}^{-1}g \widetilde{h}),\end{equation} where on the right hand side $\widetilde{h}.$ means the action of $\widetilde{h}$ on $A$.
 So we have also an $\widetilde{H}$-action on $Z(G,H,A)=\End_A(\Fu_H(G,A))$ by
\begin{equation}\label{eq:1.2'}
(\widetilde{h}.a)f=\widetilde{h}.(a(\widetilde{h}^{-1}.f)).
\end{equation}
Note that the restriction of the last $\widetilde{H}$-action to $H$ coincides with the adjoint $H$-action
on $Z(G,H,A)$. Also note that the element $e(H)$ is $\widetilde{H}$-invariant and the induced
$\widetilde{H}$-action on $A=e(H)Z(G,H,A)e(H)$ coincides with the initial one.


Also we can consider the action of $\widetilde{H}$ on $\Fu_H(G,A)$ given by $\widetilde{h}.f(g):=\widetilde{h}.f(\widetilde{h}^{-1}g)$ and  the induced action on $Z(G,H,A)$. Note that $H$ acts trivially. So we get the action of $\widetilde{H}/H$ on
$Z(G,H,A)$ by automorphisms.

Suppose now that $M$ is an $A$-bimodule. We say that $M$ is an $\widetilde{H}$-{\it equivariant} $A$-bimodule, if there is
an action of $\widetilde{H}$ on $M$ such that:
\begin{itemize}
\item The restriction of this action to $H$ coincides with the adjoint action of $H$ on $M$.
\item The natural map $A\otimes M\otimes A\rightarrow M$ is $\widetilde{H}$-equivariant.
\end{itemize}
Then we can define natural $\widetilde{H}$- and $\widetilde{H}/H$-actions on $Z(G,H,M)$
similarly to the above.

Now let $D$ be a $\K H$-linear derivation of $A$. Define the operator $D$ on $\Fu_H(G,A)$ by
$(D.f)(g)=D.(f(g))$ and the derivation $D$ of $\End(\Fu_H(G,A))$ by $D.\varphi=[D,\varphi]$. Clearly, $D.(fa)=(D.f)a+f(D.a)$ for $a\in A, f\in \Fu_H(G,A)$. Therefore $D$ preserves
$Z(G,H,A)\subset \End(\Fu_H(G,A))$. Note that the derivation
$D$ of $\End_A(\Fu_H(G,A))$ is $\K G$-linear. We remark that under an identification $Z(G,H,A)\cong \Mat_{|G/H|}(A)$
the derivation $D$ of this algebra corresponds to the entry-wise derivation, whose components
coincide with $D\in \Der(A)$.

Finally, let us provide an alternative description of $Z(G,H,A)$ in a special case.

Suppose that $A:=A_0\# H$ for some  algebra $A_0$ acted on by $H$.
Equip the space $\Fu(G,A_0)=\K[G]\otimes A_0$ with the algebra structure (with respect to the pointwise
multiplication of functions) and with the diagonal $H$-action, where the action
on $\K[G]$ is induced by the left $G$-action. The invariant subalgebra
$(\K[G]\otimes A_0)^H$ is equipped with a  $G$-action induced from the action of $G$ on itself
by right translations.

\begin{Lem}\label{Lem:new_iso}
There is a natural isomorphism $(\K[G]\otimes A_0)^H\# G\xrightarrow{\sim} Z(G,H, A_0\# H)$.
\end{Lem}
\begin{proof}
Let us define  maps $\vartheta:G, (\K[G]\otimes A_0)^H\rightarrow Z(G,H,A_0\# H)$
 as follows:
\begin{equation}\label{eq:def_vartheta}
\begin{split}&[\vartheta(g)f](g'):=f(g'g),\\
&[\vartheta(F)f](g'):=F(g')f(g'),\\
& g,g'\in G, f\in \Fu_H(G,A_0\# H), F\in (\K[G]\otimes A_0)^H=\Fu_H(G,A_0).
\end{split}
\end{equation}
It is easy to see that $\vartheta$ is well-defined and extends to a homomorphism
$(\K[G]\otimes A_0)^H\# G\rightarrow Z(G,H,A_0\# H)$.

Let us show that this homomorphism is an isomorphism. Choose representatives $g_1,\ldots,g_k$
of the left $H$-cosets in $G$. This choice identifies $Z(G,H,A_0\# H)$ with $\Mat_{|G/H|}(A_0\# H)$.
It is easy to see that $\vartheta$ also gives rise to an identification of
$(\K[G]\otimes A_0)^H\# G$ with $\Mat_{|G/H|}(A_0\# H)$.
\end{proof}

\begin{Rem}
All constructions of this subsection deal with algebras and their bimodules. However, they
work  for sheaves of algebras and sheaves of bimodules without any noticeable modifications.
\end{Rem}

\subsection{Sheafified versions of algebras, I}\label{SUBSECTION_sheafifiedI}
The goal of this section is to introduce certain sheaves $\Halg_{(l)}^{\wedge_\Leaf}|_{\Leaf}$
and $\underline{\Halg}_{(l)}^{\wedge_{\underline{\Leaf}}}|_{\underline{\Leaf}}$ of topological algebras on $\Leaf$
and on $\underline{\Leaf}$, respectively.

First of all, let us define the sheaf of algebras $\Halg|_{\X}$ on $\X=V^*/\Gamma$.  The procedure we are
going to use is very similar to the microlocalization procedure, see, for example, \cite{Ginzburg_char}, Section 1,
or \cite{Ginzburg_irr}, the proof of Theorem 2.1.
It is enough to specify the algebras of sections on principal open subsets.

Namely, pick a principal open subset $U\subset \X$ corresponding to
$f\in \K[\X]$. Set $S:=\{f^k\}_{k\geqslant 0}$. Consider the algebra $\Halg/(\param)^k$. We have a natural
epimorphism $\hat{\rho}:\Halg/(\param)^k\rightarrow \Halg_{(r+1)}$.  Since $\ad a:\Halg/(\param)^k\rightarrow \Halg/(\param)^k$ is locally nilpotent for any $a\in \hat{\rho}^{-1}(S)$, we see that $\hat{\rho}^{-1}(S)$ is an Ore subset of $\Halg/(\param)^k$.
Consider the localization $\Halg/(\param)^k(U)$ of $\Halg/(\param)^k$ with respect to $\hat{\rho}^{-1}(S)$.
The algebras $\Halg/(\param)^k(U)$ glue to a sheaf $\Halg/(\param)^k|_\X$ in the Zariski
topology.
Then set $\Halg|_{\X}:=\varprojlim_{k}\Halg/(\param)^k|_{\X}$.

There is a natural $\K^\times$-action on $\Halg$: $t.\gamma=\gamma, t.v=tv, t.c=t^2c, \gamma\in \Gamma,
v\in V, c\in \param^*, t\in \K^\times$. This action naturally extends to an action on $\Halg|_\X$. Also we note that the filtration $(\param)^k$ on $\Halg|_\X$ is complete and the quotients of this filtration
are coherent $\Str_\X$-modules.

Similarly to the analogous properties of the microlocalization, see \cite{Ginzburg_char}, Section 1,  we have
the following lemma.

\begin{Lem}\label{Lem:2.12.1}
\begin{enumerate}
\item $\Halg|_\X(\X)$ is the $\param$-adic completion  $\Halg^{\wedge_\param}$ of
$\Halg$ in the $(\param)$-adic topology.
\item $\Halg(U)$ is a flat (left or right) $\Halg$-module for any principal open subset
$U\subset \X$.
\item The functor $(\Halg|_\X)\otimes_{\Halg}\bullet$ from the category of
finitely generated left $\Halg$-modules to the category of left $\Halg|_\X$-modules is exact.
\end{enumerate}
\end{Lem}


Consider the ideal $J$ of $(SV)^\Gamma$ consisting of all functions vanishing on $\overline{\Leaf}$.
This is a Poisson ideal in $(SV)^\Gamma$.
Set $\p_{(r+1)}:=\Halg_{(r+1)}J, \p:=\hat{\rho}^{-1}(\p_{(r+1)})$.
Clearly, $\p$ is a two-sided ideal in $\Halg$.  Localizing $\p$ over $\X$, we get a sheaf $\p|_\X$
of two-sided ideals in $\Halg|_\X$. Formally, $\p|_\X:=\Halg|_\X\otimes_{\Halg}\p$.

Now consider the restriction $\Halg|_{\X^0}$ of the sheaf $\Halg|_\X$ to $\X^0$.
We have the sheaf of ideals $\p|_{\X^0}\subset \Halg|_{\X^0}$. We remark that $\p|_{\X_0}=\Halg|_{\X^0}\otimes_{\Halg}\p=\Halg|_{\X^0}\p$ and similarly $\p|_{\X^0}=\p\Halg|_{\X^0}$.
Introduce the completion $\Halg^{\wedge_\Leaf}:=\varprojlim_{k} \Halg|_{\X^0}/(\p|_{\X^0})^k=\varprojlim_{k}(\Halg/\p^k)|_{\X^0}$.
This is a  sheaf of algebras on $\X^\wedge_{\Leaf}$ complete in the $\p^{\wedge_\Leaf}$-adic topology,
where $\p^{\wedge_\Leaf}:=\varprojlim_{k}(\p/\p^k)|_{\X^0}$. Since
$\p$ is $\K^\times$-stable, the group $\K^\times$ still acts on $\Halg^{\wedge_\Leaf}$.
Finally, let $\Halg^{\wedge_\Leaf}|_{\Leaf},\p^{\wedge_\Leaf}|_{\Leaf} $ denote the (completed) sheaf-theoretic inverse images of the topological sheaves
$\Halg^{\wedge_\Leaf},\p^{\wedge_\Leaf}$ under the embedding  $\Leaf\hookrightarrow \X^{\wedge}_{\Leaf}$.
Alternatively, one can consider the pull-backs $\Halg|_{\Leaf},\p|_{\Leaf}$ of $\Halg|_{\X},\p|_{\X}$
to $\Leaf$. Then  $\Halg^{\wedge_\Leaf}|_{\Leaf}=\varprojlim_{k} \Halg|_{\Leaf}/(\p|_{\Leaf})^k$.

Again, $\Halg^{\wedge_\Leaf}|_{\Leaf}$ is a $\K^\times$-equivariant sheaf of algebras on $\Leaf$ complete in  the
$\p^{\wedge_\Leaf}|_{\Leaf}$-adic topology. We also remark that the quotients
for the $\p^{\wedge_\Leaf}|_{\Leaf}$-adic filtration are coherent $\Str_{\Leaf}$-modules.

We have the following standard property of completions. More subtle properties based on the
Artin-Rees lemma will be established later in Subsection \ref{SUBSECTION_flatness}.

\begin{Lem}\label{Lem:2.12.11}
The functor $\bullet^{\wedge_\Leaf}|_\Leaf: \M\mapsto (\varprojlim_{k}\M/\p^k\M)|_{\Leaf}$
from the category of finitely generated left $\Halg|_{\X^0}$-modules
to the category of $\Halg^{\wedge_\Leaf}|_{\Leaf}$-modules is right exact.
\end{Lem}

Similarly, we can define the sheaves $\underline{\Halg}|_{\underline{\X}^0}, \underline{\p}|_{\underline{\X}^0}$
on $\underline{\X}^0$, $\underline{\Halg}^{\wedge_{\underline{\Leaf}}},
\underline{\p}^{\wedge_{\underline{\Leaf}}}$ on $\underline{\X}^\wedge_{\underline{\Leaf}}$
and $\underline{\Halg}^{\wedge_{\underline{\Leaf}}}|_{\underline{\Leaf}}$ on $\underline{\Leaf}$.
The latter is a $\K^\times$-equivariant sheaf of algebras. Also it has an action of $\widetilde{\Xi}$
induced from the action on $\underline{\Halg}$, as  described in the introduction.

In fact, thanks to the decomposition,
\begin{equation}\label{eq:2.12.1}\underline{\Halg}=\W_\tb\otimes_{\K[\tb]}\underline{\Halg}^+,\end{equation} the structure of $\underline{\Halg}^{\wedge_{\underline{\Leaf}}}|_{\underline{\Leaf}}$ is pretty simple.
Namely, we can sheafify the Weyl algebra $\W_{\tb}$ over $\underline{\Leaf}$ similarly to the above. As a sheaf of algebras
$\W_{\tb}|_{\underline{\Leaf}}$ is $\Str_{\underline{\Leaf}}[[\tb]]$ equipped with the Moyal-Weyl $*$-product.
Let us recall that the latter is defined as follows: $f*g:=\mu(\exp(\frac{\tb}{2}\omega|_{V_0})f\otimes g)$ for sections $f,g$ of $\Str_{\underline{\Leaf}}$. Here $\mu$ is the multiplication map $\Str_{\underline{\Leaf}}^{\otimes 2}\rightarrow \Str_{\underline{\Leaf}}$. The form $\omega|_{V_0}\in \bigwedge^2 V_0^*$  can be considered as
the contraction map $\Str_{\underline{\Leaf}}^{\otimes 2}\rightarrow \Str_{\underline{\Leaf}}^{\otimes 2}$.
The decomposition (\ref{eq:2.12.1}) implies \begin{equation}\label{eq:2.12.2}
\underline{\Halg}^{\wedge_{\underline{\Leaf}}}|_{\underline{\Leaf}}=
\W_{\tb}|_{\underline{\Leaf}}\widehat{\otimes}_{\K[[\tb]]}(\underline{\Halg}^+)^{\wedge_0}\end{equation}
Here and below $\widehat{\otimes}$ stands for the completed tensor product of topological algebras/sheaves.
We equip $\W_{\tb}|_{\underline{\Leaf}}$ with the $\tb$-adic topology and $(\underline{\Halg}^+)^{\wedge_0}$ with
the $\underline{\m}^+_0$-adic topology.

Below we write $\CC(\bullet)$ instead of $Z(\Gamma,\underline{\Gamma},\bullet)$.

Consider the sheaf $\CC(\underline{\Halg}^{\wedge_{\underline{\Leaf}}}|_{\underline{\Leaf}})=\W_\tb|_{\underline{\Leaf}}\widehat{\otimes}_{\K[[\tb]]}
\CC(\underline{\Halg}^{+\wedge_0})$.
As explained in Subsection \ref{SUBSECTION_centralizer}, $\Xi$ acts on this sheaf by automorphisms.
The action clearly preserves the tensor factors and the corresponding action on $\W_\tb|_{\underline{\Leaf}}$
is induced from the $\Xi$-action on $\underline{\Leaf}$.
Abusing the notation, we write $\CC(\underline{\Halg}^{\wedge_{\underline{\Leaf}}}|_{\underline{\Leaf}})^\Xi$
instead of $\left[\pi'_*\left(\CC(\underline{\Halg}^{\wedge_{\underline{\Leaf}}}|_{\underline{\Leaf}})\right)\right]^\Xi$.
This is a sheaf on $\Leaf$.

\begin{Rem}\label{Rem:2.12.12}
We can define the sheaves $\Halg_{(l)}|_{\X}, \Halg_{(l)}^{\wedge_\Leaf}, \Halg_{(l)}^{\wedge_\Leaf}|_{\Leaf}$ etc.
similarly to $\Halg|_{\X}, \Halg^{\wedge_\Leaf}, \Halg^{\wedge_\Leaf}|_{\Leaf}$. However, Lemmas \ref{Lem:2.12.1},\ref{Lem:2.12.11} imply that $\Halg_{(l)}|_{\X}=\Halg|_{\X}/(\param^{(l)}), \Halg_{(l)}^{\Leaf}=\Halg^{\wedge_\Leaf}/(\param^{(l)})$.
\end{Rem}

Our goal is to establish a relationship between the sheaves $\Halg^{\wedge_\Leaf}|_{\Leaf}$ and
$\CC(\underline{\Halg}^{\wedge_{\underline{\Leaf}}}|_{\underline{\Leaf}})^\Xi$. We will see below
that these sheaves  become isomorphic sheaves of topological algebras if we twist one of them by a 1-cocycle. The precise statement will be given below in Subsection \ref{SUBSECTION_completion_I}.
As we will explain there, this result is an enhanced version of Theorem \ref{Thm:1}.

\subsection{Isomorphism of completions theorem, I}\label{SUBSECTION_completion_I}
First of all, we will describe an isomorphism
$$\theta_0: \Halg_{(r+1)}^{\wedge_{\Leaf}}|_{\Leaf}\rightarrow \CC(\underline{\Halg}_{(r+1)}^{\wedge_{\underline{\Leaf}}}|_{\underline{\Leaf}})^\Xi.$$
Recall that the algebras $\Halg_{(r+1)},\underline{\Halg}_{(r+1)}$ are nothing else but $SV\#\Gamma$ and $SV\#\underline{\Gamma}$
and $\CC(\bullet)$ stands for $Z(\Gamma,\underline{\Gamma},\bullet)$.
The construction of $\theta_0$ will be given after some preliminary considerations.

The proof of the following lemma is straightforward.

\begin{Lem}\label{Lem:2.1.11}
There is  a unique homomorphism $\theta_0: \Halg_{(r+1)}\rightarrow \CC(\underline{\Halg}_{(r+1)})$ such that
\begin{equation}\label{eq:2.1.11}
\begin{split}
&[\theta_0(\gamma)f](\gamma')=f(\gamma'\gamma),\\
&[\theta_0(v)f](\gamma')=\gamma'(v)f(\gamma'),\\
&\gamma,\gamma'\in \Gamma, v\in V, f\in \Fu_{\underline{\Gamma}}(\Gamma,\underline{\Halg}_{(r+1)}),
\end{split}
\end{equation}
\end{Lem}

The sheaf $\Str_{\underline{\X}^0}$ is the center of $\underline{\Halg}_{(r+1)}|_{\underline{\X}^0}$
and so also of $\CC(\underline{\Halg}_{(r+1)}|_{\underline{\X}^0})$.
Thanks to Lemma \ref{Lem:2.1.11}, we have a natural map \begin{equation}\label{eq:2.1.12}\pi'^*(\Halg_{(r+1)}|_{\X^0})=
\Str_{\underline{\X}^0}\otimes_{(SV)^\Gamma}\Halg_{(r+1)}\rightarrow \CC(\underline{\Halg}_{(r+1)}|_{\underline{\X}^0}).
\end{equation}

\begin{Lem}\label{Lem:2.1.12}
The homomorphism (\ref{eq:2.1.12}) is an isomorphism.
\end{Lem}
\begin{proof}
It is enough to check that (\ref{eq:2.1.12})
is an isomorphism fiberwise. Pick $b\in \X^0$ and $b'\in \pi^{-1}(b)$ such that $\Gamma_{b'}\subset \underline{\Gamma}$.

Let us describe the fiber of $\Halg_{(r+1)}$ at $b$. This fiber is naturally identified with
$A_{b}\# \Gamma$, where $A_b:=\K[V^*]/\K[V^*]\n_b$, $\n_b$ standing for the maximal
ideal of $b$ in $\K[V^*]^\Gamma$. Thanks to the slice theorem, $A_b$ is naturally identified with $\sum_{v\in \pi^{-1}(b)} A^0_v$,
where $A^0_v:=\K[V^*]/\K[V^*]\n^0_v$, with $\n^0_v$ being the maximal ideal of $v$ in $\K[V^*]^{\Gamma_v}$.
Set $A^0:=A^0_{b'}$. Let us identify $A_b$ with $\Fu_{\Gamma_{b'}}(\Gamma,A^0)$. Namely, to
$a\in A^0_v$ we assign a map $f:\Gamma\rightarrow A^0$ that maps $g$ to 0 if $g v\neq b'$
and to $g.a$ if $gv=b'$. It is easy to see that the assignment $a\mapsto f_a$ is a $\Gamma$-equivariant
isomorphism $A_b\xrightarrow{\sim} \Fu_{\Gamma_{b'}}(\Gamma, A^0)$. So $A_b\#\Gamma =\Fu_{\Gamma_{b'}}(\Gamma,A^0)\#\Gamma$. Further, we have a natural identification $$\Fu_{\Gamma_{b'}}(\Gamma,A^0)\cong \Fu_{\underline{\Gamma}}(\Gamma, \Fu_{\Gamma_{b'}}(\underline{\Gamma}, A^0))$$ induced by restricting a function $\Gamma\rightarrow A_0$ to the left $\underline{\Gamma}$-cosets.

Similarly, we see that the fiber of $\CC(\underline{\Halg}_{(r+1)})$ is identified with
$\CC(\Fu_{\Gamma_{b'}}(\underline{\Gamma}, A^0)\#\underline{\Gamma})$.  Set
$A_0:=\Fu_{\Gamma_{b'}}(\underline{\Gamma}, A^0)$ so that the fibers of interest
are identified with $\Fu_{\underline{\Gamma}}(\Gamma,A_0)\#\Gamma$ and $\CC(A_0\#\underline{\Gamma})$.
It is not difficult to see that under our identifications the map between the fibers
induced by (\ref{eq:2.1.12}) coincides with $\vartheta$ from Lemma \ref{Lem:new_iso}.
\end{proof}

So we get an isomorphism $\theta_0: \pi'^*(\Halg_{(r+1)}^{\wedge_\Leaf})\rightarrow \CC(\underline{\Halg}_{(r+1)}^{\wedge_{\underline{\Leaf}}})$ of coherent sheaves of $\Str_{\underline{\X}^\wedge_{\underline{\Leaf}}}$-algebras induced by (\ref{eq:2.1.12}).
But, since $\pi'$ is etale
on $\underline{\Leaf}$ and it induces a covering $\underline{\Leaf}\rightarrow\Leaf$ with fundamental
group $\Xi$, we see that $\pi': \underline{\X}^\wedge_{\underline{\Leaf}}\rightarrow \X^\wedge_{\Leaf}$
is just the quotient morphism for the (free) action of $\Xi$.
So by restricting $\theta_0$ to $\Xi$-invariants, we get an isomorphism
$\theta_0: \Halg_{(r+1)}^{\wedge_{\Leaf}}\rightarrow \CC(\underline{\Halg}_{(r+1)}^{\wedge_{\underline{\Leaf}}})^\Xi$.
Finally, restricting everything to $\Leaf$, we get the required isomorphism
\begin{equation}\label{eq:0.2.40}\theta_0:\Halg_{(r+1)}^{\wedge_{\Leaf}}|_{\Leaf}\rightarrow \CC(\underline{\Halg}_{(r+1)}^{\wedge_{\underline{\Leaf}}}|_{\underline{\Leaf}})^\Xi.\end{equation}

In fact, one can show that the sheaves $\Halg^{\wedge_{\Leaf}}|_{\Leaf}, \CC(\underline{\Halg}^{\wedge_{\underline{\Leaf}}}|_{\underline{\Leaf}})^\Xi$ cannot
be isomorphic for the open symplectic leaf $\Leaf\subset \X$.
To get an isomorphism we need to ``twist'' one of the sheaves with a cocycle.

Namely, let us choose an open $\K^\times$-stable affine covering  $\Leaf=\bigcup_i W_i$.
For all $i,j$ choose $\Xi\times \Gamma$-invariant elements
$X^{ij}\in \param\z^\tb\left(\CC(\underline{\Halg}^{\wedge_{\underline{\Leaf}}}|_{\underline{\Leaf}}(W_{ij})\right)$
having degree 2 with respect to $\K^\times$. Here $W_{ij}:=W_i\cap W_j$. We may (and will) assume that $X^{ij}=-X^{ji}$. Suppose that the elements $\exp(\frac{1}{\tb}X^{ij})$ form a 1-cocycle, i.e. \begin{equation}\label{eq:2.12.3}\exp(\frac{1}{\tb}X^{ij})\exp(\frac{1}{\tb}X^{jk})\exp(\frac{1}{\tb}X^{ki})=1.
\end{equation} A problem with (\ref{eq:2.12.3}) is that the  elements $\exp(\frac{1}{\tb}X^{ij})$ do not make sense, because, at least if we are working algebraically, the series defining $\exp(\frac{1}{\tb}X^{ij})$ diverges even if we extend the sheaf of algebras under consideration.
However, the expressions like $$\ln(\exp(\frac{1}{\tb}X^{ij})\exp(\frac{1}{\tb}X^{jk})\exp(\frac{1}{\tb}X^{ki}))$$
still make sense (thanks to the Campbell-Hausdorff formula), the resulting expression lies in  $\frac{1}{\tb}\param\z^\tb(\CC(\underline{\Halg}^{\wedge_{\underline{\Leaf}}}|_{\underline{\Leaf}})(W_{ijk})$
(where $W_{ijk}:=W_i\cap W_j\cap W_k$). And so when we write (\ref{eq:2.12.3}), we mean
that $$\ln\left(\exp(\frac{1}{\tb}X^{ij})\exp(\frac{1}{\tb}X^{jk})\exp(\frac{1}{\tb}X^{ki})\right)=0.$$

Given sections $X^{ij}$ as above we can form the twist $\CC(\underline{\Halg}^{\wedge_{\underline{\Leaf}}}|_{\underline{\Leaf}})^{tw}$ of $\CC(\underline{\Halg}^{\wedge_{\underline{\Leaf}}}|_{\underline{\Leaf}})$ by the cocycle
$\exp(\frac{1}{\tb}X^{ij})$. Namely, $\CC(\underline{\Halg}^{\wedge_{\underline{\Leaf}}}|_{\underline{\Leaf}})^{tw}(W_i):=
\CC(\underline{\Halg}^{\wedge_{\underline{\Leaf}}}|_{\underline{\Leaf}})(W_i)$ but the transition
function from $\CC(\underline{\Halg}^{\wedge_{\underline{\Leaf}}}|_{\underline{\Leaf}})(W_j)$
to $\CC(\underline{\Halg}^{\wedge_{\underline{\Leaf}}}|_{\underline{\Leaf}})(W_i)$ on $W_{ij}$
is $u\mapsto \exp(\frac{1}{\tb}\ad X^{ij})u(=\exp(\frac{1}{\tb}X^{ij})u\exp(-\frac{1}{\tb}X^{ij}))$.
We note that $\exp(\frac{1}{\tb}\ad X^{ij})u$ converges although $\exp(\frac{1}{\tb}X^{ij})$ does not.
We also remark that $\CC(\underline{\Halg}^{\wedge_{\underline{\Leaf}}}|_{\underline{\Leaf}})^{tw}/(\param)$
still identifies naturally with $\CC(\underline{\Halg}_{(r+1)}^{\wedge_{\underline{\Leaf}}}|_{\underline{\Leaf}})$.
This is because $\frac{1}{\tb}\ad X^{ij}$ is zero modulo $(\param)$.

\begin{Thm}\label{Thm:2.0I}
For a suitable 1-cocycle $X^{ij}$ as above
there is a $\K^\times$-equivariant isomorphism $\theta: \Halg^{\wedge_{\Leaf}}|_{\Leaf}\rightarrow \CC(\underline{\Halg}^{\wedge_{\underline{\Leaf}}}|_{\underline{\Leaf}}^{tw})^\Xi$ of sheaves of
$\K[[\param^*]]\Gamma$-algebras on $\Leaf$ making the following diagram commutative
and such that $\theta(\p^{\wedge_\Leaf}|_{\Leaf})$ coincides with the twist of
$\CC(\underline{\p}^{\wedge_{\underline{\Leaf}}}|_{\underline{\Leaf}})^\Xi$.

\begin{picture}(80,30)
\put(2,2){$\Halg_{(r+1)}^{\wedge_{\Leaf}}|_{\Leaf}$}
\put(2,22){$\Halg^{\wedge_{\Leaf}}|_{\Leaf}$}
\put(52,2){$\CC(\underline{\Halg}_{(r+1)}^{\wedge_{\underline{\Leaf}}}|^{tw}_{\underline{\Leaf}})^\Xi$}
\put(52,22){$\CC(\underline{\Halg}^{\wedge_{\underline{\Leaf}}}|_{\underline{\Leaf}})^\Xi$}
\put(5,20){\vector(0,-1){12}}
\put(58,20){\vector(0,-1){12}}
\put(17,4){\vector(1,0){34}}
\put(15,23){\vector(1,0){36}}
\put(28,5){\tiny $\theta_0$}
\put(28,24){\tiny $\theta$}
\put(6,13){\tiny $\hat{\rho}$}
\put(59,13){\tiny $\CC(\underline{\hat{\rho}})$}
\end{picture}
\end{Thm}

Let us see why this theorem implies Theorem \ref{Thm:1}.  For $b\in \Leaf$ the ideal $\m_b$ of $\Halg$ gives rise
to the sheaf of ideals $\m_b^{\wedge_\Leaf}|_{\Leaf}\subset\Halg^{\wedge_\Leaf}|_{\Leaf}$. The completion of $\Halg^{\wedge_\Leaf}|_{\Leaf}$ at $b$ (i.e., the completion with respect to
$(\m_b)^{\wedge_\Leaf}|_{\Leaf}$-adic topology)
is nothing else but $\Halg^{\wedge_b}$. Now let a point $\underline{b}$ lie in $\pi'^{-1}(b)\cap\underline{\Leaf}$.
The completion of $\CC(\underline{\Halg}^{\wedge_{\underline{\Leaf}}}|_{\underline{\Leaf}})^\Xi$ at $b$
is the same as the completion of $\CC(\underline{\Halg}^{\wedge_{\underline{\Leaf}}}|_{\underline{\Leaf}})$
at $\underline{b}$. The decomposition (\ref{eq:2.12.2}) implies that this completion is
naturally isomorphic to $\CC(\underline{\Halg}^{\wedge_0})$. Theorem \ref{Thm:1} follows.

The proof of Theorem \ref{Thm:2.0I} will be given  in Subsection \ref{SUBSECTION ProofI}.
It is rather indirect. We derive Theorem \ref{Thm:2.0I} from Theorem \ref{Thm:2.0} below. The latter
is another enhanced version of Theorem \ref{Thm:1}.

\subsection{Sheafified versions of algebras, II}
\label{SUBSECTION_sheafified_algebras}
Set $\Hsh:=\Str_{\underline{\Leaf}}\otimes_\K \Halg$. This is a quasi-coherent sheaf of $\Str_{\underline{\Leaf}}$-algebras.  Define a global counterpart
$\m$ of the ideals $\m_b$ as follows.  Let $I$ denote the sheaf of
 ideals in $\Str_{\underline{\Leaf}}\otimes_\K \K[V^*]$ vanishing in all points $(b,\gamma b),b\in \underline{\Leaf},
\gamma\in \Gamma$. Let $\m$ denote the inverse image  of $I\#\Gamma$ in $\Hsh$. Finally, we can define the completion $\Hsh^{\wedge}:=\varprojlim_{n}\Hsh/\m^n$. We can view $\Hsh^{\wedge}$ as a pro-coherent sheaf on $\underline{\Leaf}$ whose fiber at $b\in \underline{\Leaf}$ is $\Halg^{\wedge_b}$. This claim follows from the exact sequence $\m_b\otimes_{\Str_{\underline{\Leaf}}}\Hsh\rightarrow \Hsh\rightarrow \Halg\rightarrow 0$ and the right exactness of the completion functor, compare with Lemma \ref{Lem:2.12.11}. If $U\subset \underline{\Leaf}$ is an open affine subvariety then the space of sections $\Hsh^{\wedge}(U)$ is the completion of $\K[U]\otimes_\K \Halg$ in the $\m(U)$-adic topology. Define the sheaves $\Hsh_{(l)},\m_{(l)}, \Hsh^\wedge_{(l)}$ in a similar way.

The $\Xi$-action on $\underline{\Leaf}$  gives rise to an action of $\Xi$ on $\Hsh$ by $\xi.r\otimes h=(\xi.r)\otimes h, \xi\in \Xi, r\in \Str_{\underline{\Leaf}}, h\in \Halg$. This action extends to an action of $\Xi$ on $\Hsh^\wedge$.

Also we need a $\K^\times$-action on $\Hsh$. Namely, recall the $\K^\times$-action on $\Halg$. Next, we have a $\K^\times$-action on $V^*$: $t.\alpha=t^{-1}\alpha,\alpha\in V^*,$ that gives rise to the
$\K^\times$-action on $\Str_{\underline{\Leaf}}$. We consider the diagonal $\K^\times$-action on $\Hsh$.
We note that the ideal $I\subset\Str_{\underline{\Leaf}}\otimes_\K \K[V^*]$ is stable with respect to
this action, so $\m_{(l)}$ is also stable.

Now we introduce a sheafified version of $\underline{\Halg}^{\wedge_0}$. Namely, set $\underline{\Hsh}:=\Str_{\underline{\Leaf}}\otimes_\K \underline{\Halg}$.
We have the $S(\param)$-linear action of $\widetilde{\Xi}$ on $\underline{\Halg}$ given by
$\widetilde{\xi}.\underline{\gamma}=
\widetilde{\xi}\underline{\gamma}\widetilde{\xi}^{-1},\widetilde{\xi}.v=\widetilde{\xi} v$.
So we get an action of $\widetilde{\Xi}$ on $\underline{\Hsh}$, $\widetilde{\xi}.(r\otimes \underline{h})=(\widetilde{\xi}.r)\otimes (\widetilde{\xi}.\underline{h})$. Also analogously to
the above, we get a $\K^\times$-action on $\underline{\Hsh}$.

Define the completion $\underline{\Hsh}^\wedge$ of $\underline{\Hsh}$: consider the ideal $\underline{I}$ of the zero section $\underline{\Leaf}\rightarrow \underline{\Leaf}\times V^*$, construct the ideal $\underline{\m}\subset
\underline{\Hsh}$ from $\underline{I}$ by analogy with $\m\subset \Hsh$, and set $\underline{\Hsh}^\wedge:=\varprojlim_m \underline{\Hsh}/\underline{\m}^m$. We have a natural isomorphism $\underline{\Hsh}^\wedge=\Str_{\underline{\Leaf}}\widehat{\otimes}\underline{\Halg}^{\wedge_0}$.  The actions of $\widetilde{\Xi},\K^\times$ on $\underline{\Hsh}$ extend to $\underline{\Hsh}^\wedge$. As explained in Subsection \ref{SUBSECTION_centralizer}, we get an action of $\Xi$ on $\CC(\underline{\Hsh}^\wedge)$.

The sheaves of algebras we introduced have flat connections.
We have a linear map $\alpha\mapsto L_{\alpha,base}$ from $V_0^{*}$ to the space of derivations
of $\Str_{\underline{\Leaf}}$, mapping $\alpha$ to the Lie derivative $-\partial_\alpha$ corresponding to $\alpha$.
We can extend $L_{\alpha,base}$ to the derivation $L_{\alpha}$ of $\Hsh$ by $L_{\alpha}=L_{\alpha,base}\otimes 1$. Then we can extend $L_{\alpha,base}$ to the completion $\Hsh^\wedge$ uniquely.  Similarly,
we get derivations $\underline{L}_{\alpha}$ of $\underline{\Hsh},\underline{\Hsh}^\wedge,\CC(\underline{\Hsh}^\wedge)$.

\begin{Rem}\label{Rem:2.13.1}
 The algebras $\Hsh_{(l)},\Hsh^\wedge_{(l)}$ have the flat connections
$\alpha\mapsto L_{\alpha,(l)}$ defined analogously to $L_{\alpha}$. We have $\Hsh^\wedge_{(l)}=\Hsh^\wedge/(\param^{(l)})$, compare with Remark \ref{Rem:2.12.12}.
A similar remark applies to $\underline{\Hsh}_{(l)}$, etc.
\end{Rem}

\subsection{Isomorphism of completions theorem, II}\label{SUBSECTION_completion_II}
We are going to establish an isomorphism between $\Hsh^\wedge$ and  a certain twist of
$\CC(\underline{\Hsh}^\wedge)$. Again, we start with $\Hsh^\wedge_{(r+1)}, \CC(\underline{\Hsh}^\wedge_{(r+1)})$.
We will see that these two sheaves are isomorphic.
Their isomorphism  is  a baby version of the  isomorphism constructed in Theorem 3.2 from \cite{BE}.

Let $\beta$ denote the tautological
section of $V_0^{*}\otimes \Str_{\underline{\Leaf}}$ given by $b\mapsto b$. Note that
$\beta$ is $\K^\times$-invariant.

\begin{Prop}\label{Prop:2.1}
A  map $\Theta_0$ defined by
\begin{equation}\label{eq:2.1}
\begin{split}
&[\Theta_0(\gamma)f](\gamma')=f(\gamma'\gamma),\\
&[\Theta_0(v)f](\gamma')=(\gamma'(v)+\langle \gamma'(v),\beta\rangle)f(\gamma'),\\
&\gamma,\gamma'\in \Gamma, v\in V, f\in \Fu_{\underline{\Gamma}}(\Gamma,\underline{\Hsh}^\wedge_{(r+1)}),
\end{split}
\end{equation}
extends to a unique  homomorphism $\Hsh_{(r+1)}\rightarrow \CC(\underline{\Hsh}_{(r+1)})$
of sheaves of $\Str_{\underline{\Leaf}}$-algebras. This homomorphism
is $\Xi\times\K^\times$- equivariant. Moreover, its natural extension to the completions
$\Hsh^\wedge_{(r+1)}\rightarrow \CC(\underline{\Hsh}^\wedge_{(r+1)})$ is an isomorphism.
\end{Prop}
\begin{proof}
The claim about the extension to a homomorphism follows from the relation $\Theta_0(\gamma)\Theta_0(v)=\Theta_0(\gamma.v)\Theta_0(\gamma)$. The equivariance is checked
directly. To prove that the extension to the completions is an isomorphism it is enough
to show that it is an isomorphism fiberwise. This can be proved analogously to
Lemma \ref{Lem:2.1.12}.  
\end{proof}

In the sequel we will need a relation between $\underline{L}_{\alpha,(r+1)}$ and
$\Theta_0\circ L_{\alpha,(r+1)}\circ \Theta_0^{-1}$. For $\alpha\in V_0$ let
$\check{\alpha}=\omega(\alpha,\cdot)$. This is an element of $V_0^*$. By the definition of the Poisson
bracket on $\Str_{\underline{\Leaf}}$ we have $\{\check{\alpha},f\}=-\underline{L}_{\alpha,base}f$ for a
section $f$ of $\Str_{\underline{\Leaf}}$.  Set $\underline{\hat{L}}_{\alpha,(r+1)}:=\underline{L}_{\alpha,(r+1)}+ \{\check{\alpha},\cdot\}$. Extend $\underline{\hat{L}}_{\alpha,(r+1)}$ to $\underline{\Hsh}_{(r+1)},
\underline{\Hsh}^\wedge_{(r+1)},\CC(\underline{\Hsh}^\wedge_{(r+1)})$ in a natural way.

\begin{Lem}\label{Lem:2.40}
We have $\Theta_0\circ L_{\alpha,(r+1)}\circ \Theta_0^{-1}=\underline{\hat{L}}_{\alpha,(r+1)}$ as linear operators on $\CC(\underline{\Hsh}^\wedge_{(r+1)})$.
\end{Lem}
\begin{proof}
Note that both $\Theta_0\circ L_{\alpha,(r+1)}\circ \Theta_0^{-1}$ and $\underline{\hat{L}}_{\alpha,(r+1)}$ are  derivations
of $\CC(\underline{\Hsh}^\wedge_{(r+1)})$ that coincide on $\Str_{\underline{\Leaf}}$.
So it is enough to verify that the two derivations
coincide on some topological generators of the sheaf $\CC(\underline{\Hsh}^\wedge_{(r+1)})$ of
$\Str_{\underline{\Leaf}}$-algebras, for example, on
$\Theta_0(\gamma)$ and $\Theta_0(v)$. This follows directly from (\ref{eq:2.1}).
\end{proof}

Now we are ready to state the second version of our isomorphism of completions result, Theorem \ref{Thm:2.0}.
This theorem says that the flat (=equipped with flat connections) sheaves of algebras $\Hsh^\wedge, \CC(\underline{\Hsh}^\wedge)$ are $\Xi\times\Gamma\times
\K^\times$-equivariantly isomorphic up to a twist and, moreover, modulo $(\param)$ the isomorphism coincides
with $\Theta_0$. Pick an open  covering $U_i$ of $\underline{\Leaf}$ consisting
of $\K^\times\times \Xi$-stable affine subsets.

\begin{Thm}\label{Thm:2.0}
There are isomorphisms $\Theta^{i}: \Hsh^\wedge(U_i)\rightarrow \CC(\underline{\Hsh}^\wedge)(U_i)$
and elements $\overline{X}^{ij}\in \param\z^\tb\left(\CC(\underline{\Hsh}^\wedge(U_{ij}))\right)$, where $U_{ij}:=U_i\cap U_j$, with the following properties:
\begin{enumerate}
\item Modulo $\param$ the isomorphism $\Theta^i$ coincides with $\Theta_0$ for all $i$.
\item $\overline{X}^{ij}$ is $\Xi\times\Gamma$-invariant and has degree $2$ with respect to $\K^\times$.
\item $\Theta^i\circ (\Theta^j)^{-1}=\exp(\frac{1}{\tb}\ad \overline{X}^{ij})$.
\item $\Theta^i\circ L_\alpha\circ (\Theta^i)^{-1}= \underline{\hat{L}}_\alpha$, where $\underline{\hat{L}}_\alpha:=\underline{L}_\alpha+\frac{1}{\tb}\ad\check{\alpha}$.
\item $\overline{X}^{ij}=-\overline{X}^{ji}$.
\item $\exp(\frac{1}{\tb}\overline{X}^{ij})\exp(\frac{1}{\tb}\overline{X}^{jk})\exp(\frac{1}{\tb}\overline{X}^{ki})=1$
(see the discussion before Theorem \ref{Thm:2.0I}).
\item $\underline{\hat{L}}_\alpha \overline{X}^{ij}=0$.
\end{enumerate}
\end{Thm}

%
%

\subsection{Spherical subalgebras and the centers}\label{SUBSECTION_spherical}
Let $e=\frac{1}{|\Gamma|}\sum_{\gamma\in \Gamma}\gamma\in \K\Gamma$
be the trivial idempotent. By definition, the {\it spherical
subalgebras} $\Aalg,\Aalg_{(l)}$ in $\Halg_{(l)}$ are $e\Halg e, e\Halg_{(l)} e$. Clearly, $\Aalg_{(l)}=\Aalg/(\param^{(l)})$. Similarly, we can define the spherical subalgebra $\underline{\Aalg}_{(l)}\subset \underline{\Halg}_{(l)}$ (using the trivial
idempotent $\underline{e}$ for $\underline{\Gamma}$).

Let $\Zalg_{(l)},l=1,\ldots,r+1,$ (resp., $\underline{\Zalg}_{(l)}$) denote the
center of $\Halg_{(l)}$ (resp., $\underline{\Halg}_{(l)}$). In particular,
$\Zalg_{(r+1)}=(SV)^\Gamma\subset SV\#\Gamma=\Halg_{(r+1)}$. Set $\Zalg(=\Zalg_{(0)}):=\z^\tb(\Halg)$.
Define the algebra $\underline{\Zalg}_{(0)}$ similarly.

Also we can introduce the sheafified versions $\Ash_{(l)}:=e\Hsh_{(l)}e=\Str_{\underline{\Leaf}}\otimes \Aalg_{(l)}, \Zsh_{(l)}:=\Str_{\underline{\Leaf}}\otimes \Zalg_{(l)}, \underline{\Ash}_{(l)},\underline{\Zsh}_{(l)}$
of the algebras under consideration.

\begin{Prop}\label{Prop:1.31}
\begin{enumerate}
\item For $l>0$ the map $z\mapsto ez$ induces  isomorphism  $\Zalg_{(l)}\rightarrow\Aalg_{(l)},
\Zsh_{(l)}$ $\rightarrow \Ash_{(l)}$ (the so called Satake isomorphisms).
Similarly, the map $z\mapsto \underline{e}z$ induces  isomorphisms $\underline{\Zalg}_{(l)}\rightarrow
\underline{\Aalg}_{(l)}, \underline{\Zsh}_{(l)}\rightarrow \underline{\Ash}_{(l)}$.
\item For $l=0,1,\ldots,r$ one has $\rho_l(\Zalg_{(l)})=\Zalg_{(l+1)}, \rho_l(\Zsh_{(l)})=\Zsh_{(l+1)}, \underline{\rho}_l(\underline{\Zalg}_{(l)})=\underline{\Zalg}_{(l+1)},
    \underline{\rho}_l(\underline{\Zsh}_{(l)})=\underline{\Zsh}_{(l+1)}$.
\end{enumerate}
\end{Prop}
\begin{proof}
 Assertion (1) was essentially verified by Etingof and
Ginzburg in \cite{EG}, Theorem 3.1. Now assertion (1) easily implies assertion (2).
\end{proof}

\begin{Cor}\label{Lem:1.313}
Let $I$ be an ideal in $\Zalg_{(l)}, l>0$. Then $\Halg_{(l)}I\cap \Zalg_{(l)}=I$.
An analogous claim holds for $\Zsh_{(l)}\subset \Hsh_{(l)}$ etc.
\end{Cor}
\begin{proof}
Clearly, $I\subset \Halg_{(l)}I\cap \Zalg_{(l)}$.
Thanks to the Satake isomorphism, it is enough to show that $e I\supset e (\Halg_{(l)}I\cap \Zalg_{(l)})$.
But $e (\Halg_{(l)}I\cap \Zalg_{(l)})=e (\Halg_{(l)}I\cap \Zalg_{(l)})e\subset e\Halg_{(l)}Ie=
e\Halg_{(l)}e eI=eI$.
\end{proof}


\begin{Prop}\label{Prop:1.311}
\begin{enumerate}
\item The algebra $\Halg_{(l)}$ is finite over $\Zalg_{(l)}$ for $l>0$.
\item The center of $\Halg$ coincides with $S(\param)$.
\end{enumerate}
The similar claims hold for the algebras $\underline{\Halg}, \underline{\Halg}_{(l)}$.
\end{Prop}
\begin{proof}
The first assertion is essentially due to Etingof and Ginzburg, \cite{EG}, Theorem 3.3. The second one
follows from results of Brown and Gordon, \cite{BG}, Proposition 3.2.
\end{proof}

We finish the subsection by recalling the Poisson structures on $\Zalg_{(l)},\Aalg_{(l)}$
that are essentially due to Etingof and Ginzburg.
We start with $\Zalg_{(1)}$. Choose an arbitrary
embedding $\iota:\Zalg_{(1)}\rightarrow \Halg$ such that $\rho_0\circ \iota=\id$. Then it is easy to
see that $\rho_0([\iota(a),\iota(b)])=0$ for any $a,b\in \Zalg_{(1)}$ and that
$\frac{1}{\tb}[\iota(a),\iota(b)]\in \z^{\tb}(\Halg)$. So an element $\{a,b\}:=\rho_0(\frac{1}{\tb}[\iota(a),\iota(b)])$ is well-defined. It is straightforward
to check that $\{\cdot,\cdot\}$ does not depend on the choice of $\iota$ and that $\{\cdot,\cdot\}$ is a
$S(\param_{(1)})$-bilinear Poisson bracket on $\Zalg_{(1)}$. In particular, we have the induced
brackets on $\Zalg_{(l)}$.

In the sequel we will need an embedding $\Zalg_{(1)}\rightarrow \Halg$ with some special properties.

\begin{Lem}\label{Lem:1.312}
We have a graded (=$\K^\times$-equivariant)
$S(\param_{(1)})$-linear section $\iota:\Halg_{(1)}\rightarrow \Halg$ of $\rho_0$.
\end{Lem}
\begin{proof}
The $S(\param)$-algebra $\Halg$ is a free graded (=graded flat) $S(\param)$-module.
So $\Halg\cong \Halg_{(1)}[\tb]$ as a graded $S(\param)$-module. In particular, we
have a graded $S(\param_{(1)})$-linear embedding $\Halg_{(1)}\hookrightarrow \Halg$.
\end{proof}

This lemma allows us to define (non-canonical) maps $\Zalg_{(l)}\times \Halg_{(l)}\rightarrow \Halg_{(l)}$
induced from $(a,b)\mapsto \rho_0(\frac{1}{\tb}[\iota(a),\iota(b)]), a\in \Zalg_{(1)},b\in \Halg_{(1)}$.

Lemma \ref{Lem:1.312} shows that the induced bracket on $\Zalg_{(r+1)}=(SV)^\Gamma$ coincides with the one coming
from the symplectic form on $V$. Indeed, the bracket on $(SV)^\Gamma$ is obtained using a section
$SV^\Gamma\rightarrow \Halg/(\param_{(1)})=\W_{\tb}\#\Gamma$ of $\W_{\tb}\#\Gamma\rightarrow SV\#\Gamma$,
where $\W_{\tb}$ stands for the Weyl algebra of $V$.
With this interpretation of the bracket our claim is easy.

We can equip $\Aalg_{(1)}$ with a Poisson bracket using a section $\Aalg_{(1)}\rightarrow \Aalg$. It is easy
to see that the isomorphism in Proposition \ref{Prop:1.31} preserves the Poisson brackets.

Similarly, we can equip $\underline{\Zalg}_{(l)},\underline{\Aalg}_{(l)}, l>0,$ with Poisson structures.
Also the sheaves $\Zsh_{(l)},\Ash_{(l)},$ $\underline{\Zsh}_{(l)},\underline{\Ash}_{(l)}$
become sheaves of Poisson algebras. We also can extend the bracket $\{\cdot,\cdot\}:\Zalg_{(l)}\otimes
\Halg_{(l)}\rightarrow \Halg_{(l)}$ to $\Zsh_{(l)}\otimes \Hsh_{(l)}\rightarrow \Hsh_{(l)}$ by
$\Str_{\Leaf}$-bilinearity.

\subsection{Compatibility of filtrations}\label{SUBSECTION_compatibility}
Recall that two decreasing filtrations $\F_m V, \F'_m V$ of a vector space are called {\it compatible} if
for any $m$ there is $n$ with $\F_n V\subset \F'_m V$ and $\F_n' V\subset \F_m V$.

In this subsection
we will, roughly speaking, prove that all filtrations on $\Hsh,\Zsh$ (resp., on $\Halg,\Zalg$) related to the $\m$-adic filtration on $\Hsh$ (resp., the $\p$-adic filtration on $\Halg$) are compatible.

\begin{Prop}\label{Cor:1.31}
The following descending filtrations  are compatible:
\begin{enumerate}
\item $\F_m\Hsh_{(l)}:=\Hsh_{(l)} (\Zsh_{(l)}\cap \m_{(l)})^m, \F_m'\Hsh_{(l)}:=\m_{(l)}^m$ (for any
$l=0,1,\ldots,r+1$).
We remark that $\F_m\Hsh_{(l)}$ is actually a two-sided ideal
in $\Hsh$.
\item $\F_m\Zsh_{(l)}:=(\Zsh_{(l)}\cap\m_{(l)})^m, \F'_m\Zsh_{(l)}:=\Zsh_{(l)}\cap \m_{(l)}^m ,l\geqslant 1$.
\end{enumerate}
The similar claims hold for $\Halg_{(l)},\Zalg_{(l)}$ and the ideals related to $\m_b$.
\end{Prop}

In the proof we will need a technical lemma.
\begin{Lem}\label{Prop:1.32}
Let $I$ be a two-sided ideal in $\Zalg$ containing $\tb\Halg$. Then    $\Halg I^m=(\Halg I)^m$
for all $m$. The similar claim holds for $\Zsh\subset \Hsh$.
\end{Lem}
\begin{proof}
This follows from $[\Halg,I]\subset [\Halg,\Zalg]\subset \tb\Halg\subset I$.
%
\end{proof}
\begin{proof}[Proof of Proposition \ref{Cor:1.31}]
We remark that the filtrations $\m^n$ and $(\Hsh (\m\cap \Zsh))^n$ on $\Hsh$
are compatible. Indeed, it is enough to show that $\m^n\subset \Hsh(\m\cap \Zsh)$ for some
$n$. Since the r.h.s. contains $\param$, it is enough to prove the analogous inclusion modulo $(\param)$.
There it is enough to prove the inclusion fiberwise. This reduces to checking that $I_b^n\subset SV (I_b)^\Gamma$
for sufficiently large $n$. But this is clear.

It is enough to prove (1) for $l=0$. Thanks to the previous paragraph, (1) follows from Lemma \ref{Prop:1.32}.

%

%

(2) follows from (1) and Corollary \ref{Lem:1.313}.
%
\end{proof}

We can sheafify $\Zalg_{(l)}$ on $\X$ analogously to Subsection \ref{SUBSECTION_sheafifiedI}.
The proof of the following proposition is completely analogous to that of Proposition \ref{Cor:1.31}.
\begin{Prop}\label{Cor:1.31I}
The following descending filtrations  are compatible:
\begin{enumerate}
\item $\F_m\Halg_{(l)}|_{\X^0}:=\Halg_{(l)} (\Zalg_{(l)}\cap \p_{(l)})^m|_{\X^0}, \F_m'\Halg_{(l)}|_{\X^0}:=\p_{(l)}^m|_{\X^0}$ (for any
$l=0,1,\ldots,r+1$).
\item $\F_m\Zalg_{(l)}|_{\X^0}:=(\Zalg_{(l)}\cap\p_{(l)})^m|_{\X^0}, \F'_m\Zalg_{(l)}|_{\X^0}:=\Zalg_{(l)}\cap \p_{(l)}^m|_{\X^0} ,l\geqslant 1$.
\end{enumerate}
\end{Prop}

Define the completion  $\Zsh_{(l)}^\wedge$ of $\Zsh_{(l)}$  with respect to any of the equivalent
filtration of Proposition \ref{Cor:1.31}. Alternatively, $\Zsh_{(l)}^\wedge$ is the closure of
$\Zsh_{(l)}$ in $\Hsh_{(l)}$. Also we can define the completion $\Zalg^{\wedge_b}$ of $\Zalg$.

Similarly, we can define the sheaf $\Zalg_{(l)}^{\wedge_{\Leaf}}|_{\Leaf}
\subset \Halg_{(l)}^{\wedge_\Leaf}|_{\Leaf}$. By the construction and the properties of completions,
\cite{Eisenbud}, Chapter 7, we have the following lemma.

\begin{Lem}\label{Lem:2.12.41}
 For $l>0$ we have the following statements:
\begin{enumerate} \item The sheaf $\Zsh_{(l)}^\wedge$
is  flat over  $\Zsh_{(l)}$. Also $\Zalg_{(l)}^{\wedge_b}$ is  flat over
$\Zalg_{(l)}$.
\item $\Zsh_{(l)}^\wedge/(\cb_{l})\cong \Zsh_{(l+1)}^\wedge$ and $\Zalg^{\wedge_b}_{(l)}/(\cb_l)\cong \Zalg^{\wedge_b}_{(l+1)}$.
\item $\Zalg_{(l)}^{\wedge_{\Leaf}}|_{\Leaf}$
is  flat over $\Zalg_{(l)}$.
\item $\Zalg_{(l)}^{\wedge_{\Leaf}}|_{\Leaf}/(\cb_l)\cong \Zalg_{(l+1)}^{\wedge_{\Leaf}}|_{\Leaf}$.
\end{enumerate}
\end{Lem}

Also we remark that the analogs of Propositions \ref{Cor:1.31},\ref{Cor:1.31I} hold also
for $\underline{\Hsh},\underline{\Zsh},\underline{\Halg},\underline{\Zalg}$. There the corresponding
statements are even easier, thanks to the decompositions $\underline{\m}=\Str_{\underline{\Leaf}}\otimes \underline{\m}_0, \underline{\p}=\W_{\tb}\otimes_{\K[\tb]}\underline{\m}^+_0$.

\begin{Cor}\label{Cor:1.24}
\begin{enumerate}
\item
$\rho_l(\Zsh^\wedge_{(l)})=\Zsh^{\wedge}_{(l+1)}$ and $\rho_l(\Zalg^{\wedge_b}_{(l)})=\Zalg^{\wedge_b}_{(l+1)}$.
\item For $l>0$ we have $\Zsh_{(l)}^\wedge=\z(\Hsh^\wedge_{(l)}), \Zalg^{\wedge_b}_{(l)}=\z(\Halg^{\wedge_b}_{(l)})$.
\end{enumerate}
\end{Cor}
\begin{proof}
Assertion (1) follows from assertions (1) and (2) of Lemma \ref{Lem:2.12.41}.

Let us proceed to assertion 2.
We have the inclusion $\Zsh^\wedge_{(l)}\subset \z(\Hsh^\wedge_{(l)})$.
It is an inclusion of procoherent $\Str_{\underline{\Leaf}}$-sheaves.
Therefore it is enough to check the equality fiberwise.
By assertion 1,   $\rho_l(\Zalg^{\wedge_b}_{(l)})=\Zalg^{\wedge_b}_{(l-1)}$.
It follows easily from Proposition \ref{Prop:1.31} that the algebra $\Zalg_{(l)}$ is flat over
$S(\param_{(l)})$.  By assertion (1) of Lemma \ref{Lem:2.12.41},
the algebra $\Zalg^{\wedge_b}_{(l)}$ is flat over $\K[[\param_{(l)}^*]]$.

Suppose that we know that  $\Zalg^{\wedge_b}_{(l+1)}=\z(\Halg^{\wedge_b}_{(l+1)})$. Then we have
$\z(\Halg^{\wedge_b}_{(l)})=\Zalg^{\wedge_b}_{(l)}+ \cb_l \z(\Halg^{\wedge_b}_{(l)})$. Since $\z(\Halg^{\wedge_b}_{(l)})$ is closed in $\Halg^{\wedge_b}_{(l)}$, the last equality easily implies $\Zalg^{\wedge_b}_{(l)}=\z(\Halg^{\wedge_b}_{(l)})$.

This reduces the proof to the case $l=r+1$.
Recall the isomorphism
$\Theta_0^{b}: \Halg^{\wedge_b}_{(r+1)}\rightarrow \CC(\underline{\Halg}^{\wedge_0}_{(r+1)})$.
Under this isomorphism $\Zalg^{\wedge_b}_{(r+1)}$ goes to $(SV^{\wedge_0})^{\underline{\Gamma}}$.
Now our claim is clear.
\end{proof}

An analogous argument, of course, applies to $\Zalg_{(l)}|_\X\subset \Halg_{(l)}|_\X$, $\Zalg_{(l)}|_{\Leaf}\subset
\Halg_{(l)}|_{\Leaf}$, $\Zalg_{(l)}^{\wedge_\Leaf}|_{\Leaf}\subset \Halg_{(l)}^{\wedge_\Leaf}|_{\Leaf}$. We will need the following corollary of the analog of assertion (2).

\begin{Cor}\label{Cor:1.24.1}
$\Zalg|_{\X^0}$ (resp., $\Zalg|_{\Leaf}$) is dense in $\z^{\tb}(\Halg^{\wedge_\Leaf})$
(resp., $\z^{\tb}(\Halg^{\wedge_\Leaf}|_{\Leaf})$).
\end{Cor}

\subsection{Blow-ups and flatness}\label{SUBSECTION_flatness}
We are going to prove that $\Hsh_{(l)}^\wedge$ is flat over $\Hsh_{(l)}$, while $\Halg^{\wedge_\Leaf}|_{\Leaf}$
is flat over $\Halg|_{\Leaf}$ (=the sheaf-theoretic pull-back of $\Halg|_{\X^0}$ to $\Leaf$).

A standard technique to prove such statements is to consider the blow-up algebras.
For an associative algebra $A$ and a two-sided ideal $J\subset A$ one can form the blow-up
algebra $\Bl_J(A)=\bigoplus_{i=0}^\infty J^i$. This algebra is $\mathbb{Z}_{\geqslant 0}$-graded.  To ensure nice properties of the completion $A^\wedge:=\varprojlim_i A/J^i$ we need  the blow-up algebra
$\Bl_J(A)$ to be Noetherian.

More generally, given a sequence $J_i, i=1,2,\ldots,$ of two-sided ideals in
$A$ with $J_iJ_j\subset J_{i+j}$ one can consider the completion $A^\wedge:=\varprojlim_i A/J_i$
and the blow-up algebra $\Bl_{(J_i)}(A)=\bigoplus_{i=0}^\infty J_i$, where $J_0=A$.

The following lemma is proved completely analogously to the corresponding statements
for commutative algebras, see, for example,
\cite{Eisenbud}, Chapter 7.

\begin{Lem}\label{Lem:1.23.1}
Let $A,J_i$ be as above. Suppose $\Bl_{(J_i)}(A)$ is Noetherian. Then
\begin{enumerate}
\item The algebra $\bigoplus_{i=0}^\infty J_i/J_{i+1}$ is Noetherian.
\item The algebra $A^\wedge$ is Noetherian.
\item The algebra $A^\wedge$ is a flat (left or right) $A$-module.
\item The completion functor $M\mapsto M^\wedge:=\varprojlim_i M/J_i M$ from the category
of finitely generated left $A$-modules to the category of left $A^\wedge$-modules is exact.
Moreover, $M^\wedge$ is canonically isomorphic to $A^\wedge\otimes_A M$.
\end{enumerate}
\end{Lem}

Of course, Lemma \ref{Lem:1.23.1} can be generalized to sheaves in a straightforward way.

In \cite{HC}, Lemma 2.4.2, we have proved the following statement.

\begin{Lem}\label{Lem:1.23}
Let $A$ be a $\K[\tb]$-algebra and $I$ be a two-sided ideal of $A$ containing $\tb$.
Suppose that $A$ is complete and separated with respect to the $\tb$-adic topology. Further, suppose that
the algebra $A/(\tb)$ is commutative  and  Noetherian. Finally, suppose that
$[J,J]\subset \tb J$. Then the algebra $\Bl_J(A)$ is Noetherian.
\end{Lem}

Although we can prove that $\Bl_{\m}(\Hsh)$ is Noetherian, it is easier to do  this for a related  sheaf,
which is also enough for our purposes.

Set $\m'_{i}:=\Hsh (\Zsh\cap\m)^{2i}$. By Proposition \ref{Cor:1.31},
the $\m$-adic topology on $\Hsh$ coincides with the one defined by the sequence $\m'_{i}$. Remark that $\m'_i\m'_j\subset \m'_{i+j}$ because $[\Hsh,\Zsh\cap\m]\subset \tb\Hsh\subset \Zsh\cap\m$.

Let $\Hsh^{\wedge_\tb}$ denote
the $\tb$-adic completion of $\Hsh$. Abusing the notation, we write $\m'_i$
for the completion of $\m'_i\subset\Hsh$ in $\Hsh^{\wedge_\tb}$.

\begin{Prop}\label{Prop:1.33}
The sheaf of algebras $\bigoplus_{i\geqslant 0}\m'_i$ is Noetherian.
\end{Prop}
\begin{proof}
The proof is analogous to that of Lemma A2 in \cite{ES_appendix}. For reader's convenience
we will provide the proof here.

Fix an open affine subset $U\subset \underline{\Leaf}$. We need to prove that $\bigoplus_{i\geqslant 0}\m'_i(U)$
is a Noetherian algebra. But this algebra equals $\Hsh(U)\Bl_I(A)$, where $A:=\Zsh(U), I:=(\Zsh(U)\cap\m(U))^2$. The algebra $\Hsh(U)$ is finite over $A\subset \Bl_I(A)$.
So $\bigoplus_{i\geqslant 0}\m'_i(U)$ is finite over $\Bl_I(A)$ and it is enough to show that
$\Bl_I(A)$ is Noetherian.

According to Lemma \ref{Lem:1.23}, to do this we need to check
that the algebra $A/\tb A$ is commutative and finitely generated and  that $[I,I]\subset \tb I$.

The fact that the algebra $A/\tb A$ is commutative has been already proved when we discussed the Poisson
bracket on $\Zalg_{(1)}$ (or $\Zsh_{(1)}$). To show that the algebra $A/\tb A$ is finitely generated consider the epimorphism
$A/\tb A\twoheadrightarrow \Zsh_{(1)}(U)$. Its kernel is naturally identified with $\tb \Hsh(U)/\tb \Zsh(U)$.
In particular, its square is zero so we can view
the kernel as a $\Zsh_{(1)}(U)$-module. This module is finitely generated.
The algebra $\Zsh_{(1)}(U)$ is finitely generated as well. So we see that $A/\tb A$ is finitely generated.

Let us check that $[I,I]\subset \tb I$. This boils down to checking that $\{[\Zsh_{(1)}(U)\cap \m_{(1)}(U)]^2,[\Zsh_{(1)}(U)\cap \m_{(1)}(U)]^2\}\subset [\Zsh_{(1)}(U)\cap \m_{(1)}(U)]^2$. But this is clear.
\end{proof}

We have the following corollary of Lemma \ref{Lem:1.23.1} and Proposition \ref{Prop:1.33}.

\begin{Cor}\label{Cor:1.22}
\begin{enumerate}
\item The sheaf $\Hsh_{(l)}^\wedge$ and the algebra $\Halg^{\wedge_b}$ are Noetherian.
\item $\Hsh_{(l)}^\wedge$ (resp., $\Halg_{(l)}^{\wedge_b}$) is a flat (left or right) $\Hsh_{(l)}$-module
(resp., $\Halg_{(l)}$-module).
\end{enumerate}
\end{Cor}

A similar construction can be applied to $\p|_{\X^0}\subset \Halg|_{\X^0}$.
Namely, set $\p_{i}:=\Halg (\Zalg\cap\p)^{2i}$. By Proposition \ref{Cor:1.31I},
the $\p|_{\X^0}$-adic topology on $\Halg|_{\X^0}$ coincides with the one defined by the sequence $\p_{i}|_{\X^0}$.

\begin{Prop}\label{Prop:1.33I}
The sheaf of algebras $\bigoplus_{i\geqslant 0}\p_i|_{\X^0}$ is Noetherian.
\end{Prop}
\begin{proof}
The proof basically repeats that of Proposition \ref{Prop:1.33}. The most essential
difference is that the algebra $\Zalg_{(1)}(U)$ is not finitely generated. However,
this algebra is still Noetherian. Indeed, $\Zalg_{(1)}(U)$ is complete in the $(\param_{(1)})$-adic
topology and its quotient by $(\param_{(1)})$, $\Zalg_{(r+1)}(U)$, is finitely generated.
\end{proof}

Again, this proposition implies the following corollary.

\begin{Cor}\label{Cor:1.22I}
\begin{enumerate}
\item $\Halg_{(l)}^{\wedge_{\Leaf}}|_{\Leaf}$  is a flat  sheaf of (left or right) $\Halg_{(l)}$-modules.
\item $\Halg_{(l)}^{\wedge_{\Leaf}}|_{\Leaf}/(\cb_l)\cong \Halg_{(l+1)}^{\wedge_{\Leaf}}|_{\Leaf}$.
\end{enumerate}
\end{Cor}

We will need some further corollaries in Subsection \ref{SUBSECTION_fun_low_dag}.

Finally, we remark that all results in this subsection hold also for $\underline{\Hsh}^\wedge, \underline{\Halg}^{\wedge_b}$ etc.
They are even simpler, compare with the remark at the end of the previous subsection.

\subsection{Derivations}\label{SUBSECTION_derivations}
Throughout this subsection $U$ is an open $\K^\times\times\Xi$-stable affine subvariety of $\underline{\Leaf}$.
We set $R:=\K[U]$. In this subsection we are going to study derivations of the  algebra $\Hsh^{\wedge}(U)$.
We will show that all $R$-linear derivations of this algebra are ``almost inner'' in the sense of (A)
of Proposition \ref{Prop:2.3}.
Also we will establish the existence of an {\it ``Euler'' derivation} of $\Hsh^\wedge(U)$.
This derivation is crucial in the ``local'' construction of an isomorphism $\Theta$ from Theorem \ref{Thm:2.0}.

First of all, one can consider the space $\Der_{R[[\param^*]]}(\Hsh^{\wedge}(U))$ of
$R[[\param^*]]$-linear derivations. Further, we say that a $R$-linear derivation
$d$ of $\Hsh^{\wedge}$ is {\it weakly Euler} if  $d|_{\param}=2\id$.
Clearly,  weakly Euler derivations form an affine space, whose associated vector space is
$\Der_{R[[\param^*]]}(\Hsh^{\wedge}(U))$.

This affine space is non-empty.
Indeed, let $E_{(l)}$ denote the derivation of $\Hsh_{(l)}$ induced by the $\K^\times$-action.
Set $E_{f,(l)}:=E_{(l)}-\sum_{i=1}^{\dim V_0}\alpha^i L_{\alpha_i,(l)}$, where $\alpha_1,\ldots,\alpha_{\dim V_0}$
is a basis of $V_0$, and $\alpha^i$ are the dual basis vectors in $V_0^*$. Then $E_{f,(l)}$ is $\Str_{\underline{\Leaf}}$-linear,
$E_{f,(l)} v=v, v\in V, E_{f,(l)} \cb_i=2\cb_i$.
Extend $E_{f,(l)}$ to $\Hsh^\wedge_{(l)}$. Clearly, $E_{f,(l)}$ is a weakly Euler
derivation of $\Hsh^\wedge_{(l)}$. Similarly, we can define the derivation $\underline{E}_{f,(l)}$
of $\underline{\Hsh}_{(l)}, \underline{\Hsh}^\wedge_{(l)}, \CC(\underline{\Hsh}^\wedge_{(l)})$.

Now let us define Euler derivations.
We say that a derivation $d$ of $\Hsh^\wedge(U)$ is {\it Euler} if it is weakly
Euler, $\K^\times\times\Gamma\times \Xi$-equivariant, and the induced derivation of $\Hsh^\wedge_{(r+1)}(U)$ coincides with $\Theta_0^{-1}\circ\underline{E}_{f,(r+1)}\circ \Theta_0$.


\begin{Prop}\label{Prop:2.3}
The following statements hold:
\begin{itemize}
\item[(A)] All elements of $\Der_{R[[\param^*]]}(\Hsh^{\wedge}(U))$ are of the form $\frac{1}{\cb_0}\ad a$,
where $a\in \Zsh^{\wedge}(U)$. Recall that, by definition, $\cb_0=\tb$.
\item[(B)] There is an Euler derivation of $\Hsh^{\wedge}(U)$.
\end{itemize}
\end{Prop}
\begin{proof}
First, we reduce (A),(B) to the similar claims for the algebras $\Hsh^\wedge_{(l)}(U),l=1,\ldots,r+1$.
Recall that $\Zsh^{\wedge}_{(l)}(U)$ is a Poisson algebra for any $l>0$ and coincides with
the center of $\Hsh^\wedge_{(l)}(U)$ (the latter follows from Corollary \ref{Cor:1.24}).
We say that an $R$-linear derivation of $\Hsh^{\wedge}_{(l)}(U)$ is:
\begin{itemize}
\item Poisson, if it is $R[[\param_{(l)}^*]]$-linear and annihilates the Poisson bracket on $\Zsh^{\wedge}_{(l)}(U)$
(we say that a derivation $D$ of an algebra $A$ annihilates a bracket $\{\cdot,\cdot\}$ if
$D\{a,b\}-\{Da,b\}-\{a,Db\}=0$ for all $a,b\in A$).
\item weakly Euler, if it  acts on $\param_{(l)}$ by $2\cdot\id$ and
also multiplies the Poisson bracket  $\{\cdot,\cdot\}$ on $\Zsh_{(l)}^{\wedge}(U)$ by $-2$.
\end{itemize}

The notion of an Euler derivation of $\Hsh_{(l)}^\wedge$  is introduced exactly as above.

Any $R[[\param^*]]$-linear (resp., weakly Euler, Euler) derivation of $\Hsh^\wedge(U)$ induces a Poisson (resp.,
weakly Euler, Euler) derivation of $\Hsh^{\wedge}_{(l)}(U)$. 

Consider the following  statements, where $l=1,\ldots,r+1$,
\begin{itemize}
\item[(A$_l$)] Any Poisson derivation $d_0$ of $\Hsh^{\wedge}_{(l)}(U)$ has the form
$d_0(c)=\{a_1,c\}+[a_2,c], a_1\in \Zsh^{\wedge}_{(l)}(U), a_2\in \Hsh^{\wedge}_{(l)}(U)$.
Moreover, any such derivation lifts
to $\Hsh^{\wedge}_{(l-1)}(U)$. Here $\{\cdot,\cdot\}$ is defined as the continuous extension
of $\{\cdot,\cdot\}:\Zsh_{(l)}\otimes \Hsh_{(l)}\rightarrow \Hsh_{(l)}$ see the concluding
remarks in Subsection \ref{SUBSECTION_spherical}.
\item[(B$_l$)] There is an Euler derivation of $\Hsh^{\wedge}_{(l)}(U)$.
\item[(C$_l$)] Any  derivation $d_0$ of $\Hsh^{\wedge}_{(l)}(U)$ vanishing
on $\Zsh^\wedge_{(l)}(U)$ is inner.
\end{itemize}

(A$_1$)\&(C$_1$)$\Rightarrow$(A): Let $d\in \Der_{R[[\param^*]]}(\Hsh^{\wedge}(U))$ and let $d_0$ be the corresponding
derivation of $\Hsh^\wedge_{(1)}(U)$. Choose $a_1$ as in  (A$_1$) so that the restriction
of $d_0$ to $\Zsh_{(1)}^\wedge(U)$ coincides with $\{a_1,\cdot\}$. Lift $a_1$ to an element
$a'\in \Hsh^\wedge(U)$. Replacing $d$ with $d:=d-\frac{1}{\cb_0}\ad(a')$ we get that
$d_0$ vanishes on $\Zsh_{(1)}^\wedge(U)$. So by (C$_1$), $d_0$ is inner, i.e., there is $a_2$ such that
$d_0=\ad(a_2)$. Lift $a_2$ to an element $a''$ of $\Hsh^\wedge(U)$. Replacing $d$ with $d-\ad(a'')$ we get
$d_0=0$. In other words, $\im d\subset \cb_0 \Hsh^\wedge(U)$. Since $\Hsh^\wedge(U)$ is a flat $\K[[\cb_0]]$-module
(Corollary \ref{Cor:1.22}),
we see that $d=\cb_0 d_1$ for some $d_1\in \Der_{R[[\param^*]]}(\Hsh^\wedge(U))$, which completes the proof of the
lifting property. We can apply the same argument
to $d_1$ and replace $\cb_0 d_1$ with $\cb_0^2 d_2$ and so on.  Since $\Hsh^{\wedge}(U)$ is complete in
the $\cb_0$-adic topology, we see that $d$ has the required form.

The proofs of the implications (C$_l$)$\Rightarrow$(C$_{l-1}$) and (A$_l$)\&(C$_l$)$\Rightarrow$(A$_{l-1}$) for $l>1$ are very similar and so we omit them (one needs to consider $\{a',\cdot\}$ instead of $\frac{1}{\cb_0}\ad(a')$).

Let us now prove  (A$_{r+1}$) and (C$_{r+1}$). Both stem from the following result.
\begin{Lem}\label{Lem:2.21}
The following statements hold:
\begin{enumerate}
\item The restriction of derivations from $\Hsh^{\wedge}_{(r+1)}(U)$ to $\Zsh^{\wedge}_{(r+1)}(U)$ gives
rise to an isomorphism $\HH^1_{R}(\Hsh^{\wedge}_{(r+1)}(U))\rightarrow \HH^1_{R}(\Zsh^\wedge_{(r+1)}(U))=\Der_R(\Zsh^\wedge_{(r+1)}(U))$, where $\HH$ denote
the Hochschild cohomology.
\item Any Poisson derivation of $\Zsh^\wedge_{(r+1)}(U)$ is inner.
\end{enumerate}
\end{Lem}
\begin{proof}[Proof of Lemma \ref{Lem:2.21}]
Recall the isomorphism $\Theta_0^U:\Hsh^{\wedge}_{(r+1)}(U)\rightarrow \CC(
\underline{\Hsh}^\wedge_{(r+1)}(U))$ from Proposition \ref{Prop:2.1}. This isomorphism
restricts to an isomorphism $\Zsh^\wedge_{(r+1)}(U)\rightarrow \underline{\Zsh}^\wedge_{(r+1)}(U)$. As we explained
in Subsection \ref{SUBSECTION_centralizer}, there is a natural map \begin{equation}\label{eq:2.21}\Der_{R}(\underline{\Hsh}^\wedge_{(r+1)}(U))
\rightarrow \Der_R(\CC(\underline{\Hsh}^\wedge_{(r+1)}(U))).\end{equation}   The algebra
$\CC(\underline{\Hsh}^\wedge_{(r+1)}(U))$ is isomorphic to a matrix algebra
over $\underline{\Hsh}^\wedge_{(r+1)}(U)$. It follows that the map (\ref{eq:2.21}) yields an isomorphism
of the 1st Hochshild cohomology groups. Since this map intertwines the restriction maps
$\Der_R(\underline{\Hsh}^\wedge_{(r+1)}(U))\rightarrow \Der_R(\underline{\Zsh}^\wedge_{(r+1)}(U)),
\Der_R(\CC(\underline{\Hsh}^\wedge_{(r+1)}(U)))\rightarrow \Der_R(\underline{\Zsh}^\wedge_{(r+1)}(U))$,
we see that it is enough to prove the analogs of (1),(2) for $\underline{\Hsh}^\wedge_{(r+1)}(U),\underline{\Zsh}^\wedge_{(r+1)}(U)$.

Let us prove (1). Recall the fiberwise Euler derivation $\underline{E}_{f,(r+1)}$ of $\underline{\Hsh}^\wedge_{(r+1)}(U)$. Note that any element of $\Der_R(\underline{\Hsh}^\wedge_{(r+1)}(U))$ is uniquely determined by its restrictions to $\K\underline{\Gamma}$ and $V$. These subspaces of $\underline{\Hsh}^\wedge_{(r+1)}(U)$ have degrees 0,1 with respect to $\underline{E}_{f,(r+1)}$.
It follows that any element $d\in \Der_R(\underline{\Hsh}^\wedge_{(r+1)}(U))$
can be uniquely written as the (automatically converging) sum $\sum_{i=-1}^\infty d_i$, where $[\underline{E}_{f,(r+1)},d_i]=i \cdot d_i$. Note that each $d_i$ is a derivation of $\Hsh_{(r+1)}(U)=R\otimes SV\# \underline{\Gamma}$. It is easy to see that $\Der_R(R\otimes SV\#\underline{\Gamma})=R\otimes \Der(SV\#\underline{\Gamma})$, $\Der_R(R\otimes (SV)^{\underline{\Gamma}})=R\otimes \Der((SV)^{\underline{\Gamma}})$.

Analogously to the proof of Theorem 9.1 in \cite{Etingof}, one can show that the natural map
$(\Der(SV))^{\underline{\Gamma}}\rightarrow \Der(SV\#{\underline{\Gamma}})$ gives rise to an isomorphism
$(\Der(SV))^{\underline{\Gamma}}\rightarrow \HH^1(SV\#{\underline{\Gamma}})$. Since
the morphism $V\rightarrow V/\underline{\Gamma}$ is \'{e}tale in codimension 1, we see that
the restriction map $(\Der(SV))^{\underline{\Gamma}}\rightarrow \Der((SV)^{\underline{\Gamma}})$
is an isomorphism. (1) follows.

The proof of (2) is similar. Namely, we reduce the proof to checking that any Poisson
$R$-linear  derivation of $\underline{\Zsh}_{(r+1)}(U)=R\otimes (SV)^{\underline{\Gamma}}$ is Hamiltonian.
We have $\operatorname{PDer}(\underline{\Zsh}_{(r+1)}(U))=R\otimes \operatorname{PDer}((SV)^{\underline{\Gamma}})$,
where $\operatorname{PDer}$ denotes the space of $R$-linear Poisson derivations. The isomorphism
$(\Der(SV))^{\underline{\Gamma}}\rightarrow \Der((SV)^{\underline{\Gamma}})$ restricts to an isomorphism
$[\operatorname{PDer}(SV)]^{\underline{\Gamma}}\rightarrow \operatorname{PDer}((SV)^{\underline{\Gamma}})$.
However, any Poisson derivation of $SV$ is Hamiltonian.
\end{proof}

(B$_1$)$\Rightarrow$(B): Let $E_1$ be an Euler derivation of $\Hsh^{\wedge}_{(1)}(U)$.  Then
$E_{f,(1)}-E_1$ is a Poisson derivation of $\Hsh^\wedge_{(1)}(U)$.  By (A$_1$), being a Poisson derivation of $\Hsh^\wedge_{(1)}(U)$,  $E_{f,(1)}-E_1$ can be lifted to an element  $d\in\Der_{R[[\param^*]]\Gamma}(\Hsh^{\wedge}(U))$.  Replace $d$
with its $\K^\times\times\Gamma \times \Xi$-invariant component (this is possible
because the action of $\K^\times$ is pro-algebraic). Then $E_f+d$ is an Euler derivation of $\Hsh^{\wedge}(U)$.

The proof of (B$_l$)$\Rightarrow$(B$_{l-1}$) for $l>1$ is completely analogous.

The claim (B$_{r+1}$) is vacuous, for  $\Theta_0^{-1}\circ\underline{E}_{f,(r+1)}\circ \Theta_0$ is an Euler derivation of $\underline{\Hsh}^{\wedge}_{(r+1)}(U)$.
\end{proof}

The following proposition explains why the existence of an Euler derivation is important.

\begin{Prop}\label{Prop:2.4}
Let $E$ be an Euler derivation of $\Hsh^{\wedge}(U)$.
Set $\Hsh^i(U):=\{a\in \Hsh^{\wedge}(U): Ea=ia\}, i\in \ZZ, \Hsh^i_{(r+1)}(U):=\{a\in \Hsh_{(r+1)}^{\wedge}(U): \Theta_0^{-1}\circ \underline{E}_{f,(r+1)}\circ \Theta_0 (a)=ia\}$.
Then the following claims hold:
\begin{enumerate}
\item $\Hsh^i(U)=\{0\}$ for $i<0$.
\item $\hat{\rho}: \Hsh^\wedge(U)\rightarrow \Hsh^\wedge_{(r+1)}(U)$ gives  isomorphisms  $\Hsh^i(U)\rightarrow \Hsh^i_{(r+1)}(U), i=0,1$.
\item The map $\prod_{i=0}^\infty \Hsh^i(U)\rightarrow \Hsh^{\wedge}(U),\prod_{i=0}^\infty v^{(i)}\mapsto \sum_{i=0}^\infty v^{(i)},$ is well-defined (meaning that the r.h.s. converges) and is a bijection.
\end{enumerate}
\end{Prop}
\begin{proof}
Let us define the action of $E$ on $\Hsh^\wedge_{(r+1)}(U)[[\param^*]]$ as follows: $E$ acts on $\Hsh^{\wedge}_{(r+1)}(U)$ by $\Theta_0^{-1}\circ \underline{E}_{f,(r+1)}\circ\Theta_0$ and by 2 on $\param$. We remark that the analogs of
claims (1),(3) hold for $\underline{\Hsh}^{\wedge}_{(r+1)}(U)$. Recall that $\Hsh^\wedge(U)$ is $\K[[\param^*]]$-flat
and we have the natural isomorphism $\Hsh^{\wedge}(U)/(\param)\cong \Hsh^{\wedge}_{(r+1)}(U)$ intertwining the
action of $E$. Now (1) for $\Hsh^{\wedge}_{(r+1)}(U)$ implies that there is a $E$-equivariant $\K[[\param^*]]$-linear isomorphism $\Hsh^\wedge(U)\cong \Hsh^\wedge_{(r+1)}(U)[[\param^*]]$.

All three claims follow easily from here.
\end{proof}

In the next subsection we will need the following result.

\begin{Lem}\label{Lem:auxil}
Let $U\subset \underline{\Leaf}$ be an open affine subset, and $X\in \z^{\tb}(\CC(\underline{\Hsh}^\wedge(U)))$.
Suppose that the image of the derivation $\frac{1}{\tb}\ad X$ lies in $\cb \z^{\tb}(\CC(\underline{\Hsh}^\wedge(U)))$.
Then $X\in \cb \z^{\tb}\left(\CC(\underline{\Hsh}^\wedge(U))\right)+\K[U]$.
\end{Lem}
\begin{proof}
We will prove by the  induction on $i=0,\ldots,r+1$   that $$X\in (\cb_0,\ldots,\cb_{i-1})\z^{\tb}(\CC(\underline{\Hsh}^\wedge(U)))+(\cb_i,\ldots,\cb_r)\CC(\underline{\Hsh}^\wedge(U)).$$

Let $X_i$ denote the image of $X$ in $\CC(\underline{\Hsh}^\wedge(U))/(\cb_i,\ldots,\cb_r)$.
Since $\K[U]$ coincides with the Poisson center of $\CC(\underline{\Hsh}_{(r+1)}^{\wedge}(U))$, we see that
$X_0\in \K[U]$ and so we have the base of induction.

Now suppose that we have proved our claim for $i$. Let us prove it for $i+1$. The natural
homomorphism $\z^{\tb}\left(\CC(\underline{\Hsh}^\wedge(U))\right)\rightarrow \z^{\tb}(\CC(\underline{\Hsh}^\wedge(U))/(\cb_{i+1},\ldots,\cb_r))$ is surjective (compare with Corollary \ref{Cor:1.24}). This reduces the claim for $i+1$ to checking
that $X_{i+1}\in (\cb_0,\ldots,\cb_{i})\z^{\tb}\left(\CC(\underline{\Hsh}^\wedge(U))/(\cb_{i+1},\ldots,\cb_r)\right)$.
Let $X_i'$ denote some lifting of $X_i$ to the algebra $\z^{\tb}\left(\CC(\underline{\Hsh}^\wedge(U))/(\cb_{i+1},\ldots,\cb_r)\right)$.
Then $X_i'-X_{i+1}=\cb_i Y$ for some $Y\in \CC(\underline{\Hsh}^\wedge(U))/(\cb_{i+1},\ldots,\cb_r)$.
Of course, $\cb_i Y\in \z^{\tb}(\CC(\underline{\Hsh}^\wedge(U))/(\cb_{i+1},\ldots,\cb_r))$.
Since $\CC(\underline{\Hsh}^\wedge(U))/(\cb_{i+1},\ldots,\cb_r)$ is flat over $\K[\cb_0,\ldots,\cb_{i}]$,
we see that $Y\in \z^{\tb}(\CC(\underline{\Hsh}^\wedge(U))/(\cb_{i+1},\ldots,\cb_r))$, and we are done.
\end{proof}

\subsection{Proof of Theorem \ref{Thm:2.0}}\label{SUBSECTION_ProofII}
In this subsection we will prove Theorem \ref{Thm:2.0}. The first step is to show that
$\Theta$ exists locally. As before, let
 $U$ be a $\K^\times\times \Xi$-stable open affine subvariety of $\underline{\Leaf}$
 and  $R:=\K[U]$.

\begin{Prop}\label{Prop:2.0.11}
There is an $R[[\param^*]]$-linear $\K^\times\times\Gamma\times\Xi$-equivariant isomorphism $\Theta^U:\Hsh^\wedge(U)\rightarrow \CC(\underline{\Hsh}^\wedge(U))$ lifting $\Theta_0^U:\Hsh^\wedge_{(r+1)}(U)\rightarrow
\CC(\underline{\Hsh}^\wedge_{(r+1)}(U))$.
\end{Prop}
\begin{proof}
Recall the idempotent $e(\underline{\Gamma})\in \CC(\underline{\Hsh}^\wedge_{(r+1)}(U))$. In the notation of
Proposition \ref{Prop:2.4}, we have
$\CC(\K\underline{\Gamma})\subset \Hsh^0_{(r+1)}(U)$. So we can use $\hat{\rho}$ to get an embedding $\CC(\K\underline{\Gamma})\hookrightarrow\Hsh^0(U)$. It is clear that $\hat{\rho}: \Hsh^0(U)\rightarrow \Hsh^0_{(r+1)}(U)$ is an isomorphism of algebras. So we may consider $\CC(\K\underline{\Gamma})$
as a subalgebra in $\Hsh^0(U)\subset \Hsh^\wedge(U)$.
Thanks to Lemma \ref{Lem:0.0.12}, the algebra
$\Hsh^{\wedge}(U)$ is naturally identified with
$\CC\left(e(\underline{\Gamma})\Hsh^{\wedge}(U)e(\underline{\Gamma})\right)$.
So we need to establish an isomorphism $e(\underline{\Gamma})\Hsh^{\wedge}(U)e(\underline{\Gamma})\rightarrow\underline{\Hsh}^\wedge(U)$.
Consider the subspace $e(\underline{\Gamma}) \Hsh^{1}(U)e(\underline{\Gamma})\subset e(\underline{\Gamma}) \Hsh^{\wedge}(U)e(\underline{\Gamma})$.
It is identified with $R\otimes (V\otimes \K\underline{\Gamma})$ by means of $\hat{\rho}$. Let $\iota:V\hookrightarrow e(\underline{\Gamma})\Hsh^{1}(U)e(\underline{\Gamma})$ be the corresponding embedding.

Let us check that
\begin{equation}\label{eq:2.5}
[\iota(u),\iota(v)]=\cb_0\omega(u,v)+\sum_{s\in S\cap\underline{\Gamma}}\cb(s)\omega_s(u,v)s.
\end{equation}
Consider $v\in V\hookrightarrow \Hsh(U)$. We can write the decomposition $v=\sum_{i=0}^\infty v^{(i)}$ with
$v^{(i)}\in \Hsh^i(U)$. The formulas (\ref{eq:2.1}) imply that $e(\underline{\Gamma})\Theta_0(v)=\Theta_0(v)e(\underline{\Gamma})=v+\langle\beta,v\rangle$ for all $v\in V$.
Therefore
\begin{equation}\label{eq:2.4}
e(\underline{\Gamma})ve(\underline{\Gamma})\equiv e(\underline{\Gamma})v\equiv ve(\underline{\Gamma})\equiv \langle\beta,v\rangle+\iota(v) \mod \prod_{i\geqslant 2}\Hsh^i(U).
\end{equation}
Recall that $[u,v]=\cb_0\omega(u,v)+\sum_{s\in S}\cb(s)\omega_s(u,v)s$. It follows that
\begin{equation}\label{eq:2.6}
e(\underline{\Gamma})[u,v]e(\underline{\Gamma})=\cb_0\omega(u,v)+\sum_{s\in S\cap\underline{\Gamma}}\cb(s)\omega_s(u,v)s.
\end{equation}
On the other hand, from (\ref{eq:2.4}) we deduce that
\begin{equation}\label{eq:2.7}
[\iota(u),\iota(v)]\equiv e(\underline{\Gamma})[u,v]e(\underline{\Gamma})\equiv \cb_0\omega(u,v)+\sum_{s\in S\cap\underline{\Gamma}}\cb_s\omega_s(u,v)s\mod \prod_{i\geqslant 3}\Hsh^i(U).
\end{equation}
Since $[\iota(u),\iota(v)]\in \Hsh^2(U)$, (\ref{eq:2.5}) is proved.

(\ref{eq:2.5}) implies that  there is a unique $R[[\param^*]]\underline{\Gamma}$-linear isomorphism $\Theta^U:e(\underline{\Gamma})\Hsh^\wedge(U)
e(\underline{\Gamma})\rightarrow \underline{\Hsh}^\wedge(U)$ coinciding with $\iota$ on $V$. The natural
extension of this isomorphism
$$\Theta^U:\Hsh^\wedge(U)=\CC(e(\underline{\Gamma})\Hsh^\wedge(U) e(\underline{\Gamma}))\rightarrow
\CC(\underline{\Hsh}^\wedge(U))$$
is the isomorphism we need in the proposition. This isomorphism is $\K^\times\times\Gamma\times\Xi$-equivariant
because the Euler derivation $E$ is  $\K^\times\times\Gamma\times \Xi$-equivariant, by definition.
\end{proof}

To finish the proof of Theorem \ref{Thm:2.0} (glue the isomorphisms $\Theta^U$
together in an appropriate way) we need two technical lemmas.

\begin{Lem}\label{Lem:2.0.1}
There are
\begin{itemize}\item[(i)] Elements  $X^{ij}\in \param \z^\tb(\CC(\underline{\Hsh}^\wedge(U_{ij})))$ satisfying (2),(3),(5) of Theorem \ref{Thm:2.0},
\item[(ii)] and a $\Gamma\times \Xi$-equivariant map $V_0^*\rightarrow \param\z^\tb(\CC(\underline{\Hsh}^\wedge))(U_{i}), \alpha\mapsto Y^i_\alpha,$
     of degree 2 with respect to $\K^\times$
\end{itemize}
such that for all $\alpha,\beta\in V_0^*$ the following hold
\begin{align}\label{eq:5.1.3}
&\Theta^{U_i}\circ L_\alpha\circ (\Theta^{U_i})^{-1}=\underline{\hat{L}}_\alpha+\frac{1}{\tb}\ad(Y^i_\alpha).\\
\label{eq:5.1.0}
&  c^{ijk}:=\tb\ln\left(\exp(\frac{1}{\tb}X^{ij})\exp(\frac{1}{\tb}X^{jk})\exp(\frac{1}{\tb}X^{ki})\right)\in \param\K[U_{ijk}][[\param^*]],\\\label{eq:5.1.5}
&c^{ij}(\alpha):=\check{\alpha}+Y^{i}_\alpha- \exp(\frac{1}{\tb}X^{ij})(\check{\alpha}+Y^{j}_\alpha) \exp(-\frac{1}{\tb}X^{ij})-\exp(\frac{1}{\tb}X^{ij})[\underline{L}_\alpha \exp(-\frac{1}{\tb}X^{ij})]\\\nonumber&\in \param\K[U_{ij}][[\param^*]],\\\label{eq:5.1.4}
&c^i(\alpha,\beta):=\underline{L}_\alpha(\check{\beta}+Y^i_\beta)-\underline{L}_\beta(\check{\alpha}+Y^i_\alpha)+
\frac{1}{\tb}[\check{\alpha}+Y_\alpha^i,\check{\beta}+Y_\beta^i]-\omega(\alpha,\beta)\in \param\K[U_i][[\param^*]].
\end{align}
\end{Lem}
\begin{proof}
The automorphism $\Theta^{U_i}\circ (\Theta^{U_j})^{-1}$ of $\CC(\underline{\Hsh}^\wedge(U_{ij}))$ is the identity modulo $(\param)$. Therefore $d:=\ln(\Theta^{U_i}\circ (\Theta^{U_j})^{-1})$ converges and is a $\K^\times\times \Gamma\times \Xi$-equivariant $\K[U_{ij}][[\param^*]]$-linear derivation of $\CC(\underline{\Hsh}^\wedge(U_{ij}))$. Clearly, $d$ is zero modulo $(\param)$, and $\Theta^{U_i}\circ (\Theta^{U_j})^{-1}=\exp(d)$. By Proposition \ref{Prop:2.3}, $d=\frac{1}{\tb}\ad(X^{ij})$ for some $X^{ij}\in \z^{\tb}(\CC(\underline{\Hsh}^{\wedge}))(U_{ij})$.
By Lemma \ref{Lem:auxil}, since $d$ is trivial modulo $(\param)$, we have  $X^{ij}\in \K[U_{ij}][[\param^*]]+\param\z^{\tb}(\CC(\underline{\Hsh}^{\wedge}))(U_{ij})$.
Since $\K[U_{ij}][[\param^*]]\subset \z(\CC(\underline{\Hsh}^{\wedge}))(U_{ij})$,
we may assume that actually  $X^{ij}\in \param\z^{\tb}\left(\CC(\underline{\Hsh}^{\wedge})(U_{ij})\right)$.
Further, after replacing $X^{ij}$ with its appropriate isotypic component for
the action of $\K^\times\times\Gamma\times \Xi $, we may assume that $X^{ij}$ satisfies (2) of Theorem \ref{Thm:2.0}.
Also we may assume that $X^{ij}=-X^{ji}$.
Since  \begin{equation}\label{eq:5.1.1}\exp(\frac{1}{\tb}\ad X^{ij})=\Theta^{U_i}\circ (\Theta^{U_j})^{-1},\end{equation} we see that $\exp(\frac{1}{\tb}\ad X^{ki})\exp(\frac{1}{\tb}\ad X^{jk})\exp(\frac{1}{\tb}\ad X^{ij})=\id$. Therefore the left hand side of (\ref{eq:5.1.0}) lies in the center of $\CC(\underline{\Hsh}^\wedge(U_{ijk}))$.
Proposition \ref{Prop:1.311} implies that  the center of $\CC(\underline{\Hsh}^\wedge(U_{ij}))$
coincides with $\K[U_{ij}][[\param^*]]$ hence (\ref{eq:5.1.0}).

Similarly, we see that there is a map $\alpha\mapsto Y^i_\alpha$ satisfying (ii) and (\ref{eq:5.1.3}).
To prove (\ref{eq:5.1.4}) we notice that $[L_\alpha,L_\beta]=0$ and repeat the argument in the previous
paragraph. (\ref{eq:5.1.5}) follows in a similar way from (\ref{eq:5.1.1}) and (\ref{eq:5.1.3}) together.
\end{proof}


It is pretty easy to see that $c^{ij}(\alpha)=-c^{ji}(\alpha)$ and $c^{ijk}=-c^{jik}=-c^{ikj}$.
So the triple $c=(c^{i}(\bullet,\bullet), c^{ij}(\bullet), c^{ijk})$ is a 2-cochain in the \v{C}ech-De Rham complex of $\underline{\Leaf}$.

\begin{Lem}\label{Lem:2.0.2}
The 2-cochain $c$ is a coboundary.
\end{Lem}
\begin{proof}
Below we write $\hat{Y}^i_\alpha$ for $\check{\alpha}+Y^i_\alpha$.

The variety $\underline{\Leaf}$ is the complement  to the union of certain subspaces of even codimension
in $V_0^*$.
It follows that $\operatorname{H}^2_{DR}(\underline{\Leaf})=0$. Therefore it is enough to check that
$c$ is a 2-cocycle. This reduces to checking the following four equalities:
\begin{align}\label{eq:5.2.1}
&\underline{L}_\alpha c^i(\beta,\gamma)+ \underline{L}_\beta c^i(\gamma,\alpha)+ \underline{L}_\gamma c^i(\alpha,\beta)=0,\\\label{eq:5.2.2}
&c^i(\alpha,\beta)-c^j(\alpha,\beta)= \underline{L}_\alpha c^{ij}(\beta)-\underline{L}_\beta c^{ij}(\alpha),\\\label{eq:5.2.3}
&c^{ij}(\alpha)+c^{jk}(\alpha)+c^{ki}(\alpha)=\underline{L}_\alpha c^{ijk},\\\label{eq:5.2.4}
&c^{ijk}-c^{ijl}+c^{ikl}-c^{jkl}=0.
\end{align}

(\ref{eq:5.2.1}) reduces to the Jacobi identity for $Y_\alpha^i, Y_\beta^i, Y_\gamma^i$.

Below we write $A^{ij}$ instead of $\exp(\frac{1}{\tb}X^{ij})$. Again, although $A^{ij}$ is not well-defined
(because it diverges) all expressions involving $A^{ij}$ below will be well-defined. More precisely, we will
need expressions of the form $A^{ij}x A^{ji}, A^{ji} D(A^{ij})$ for some derivation $D$.
The former just equals $\exp(\frac{1}{\tb}\ad X^{ij})x$, while the latter is expressed as some convergent
series in $X^{ij}, DX^{ij}$ and their brackets,  with $DX^{ij}$ appearing only ones, thanks to the
Campbell-Haussdorff formula.

Let us prove (\ref{eq:5.2.2}). Rewrite the l.h.s. as
\begin{align*} &c^i(\alpha,\beta)-A^{ij} c^j(\alpha,\beta) A^{ji}=\\= &\underline{L}_\alpha \hat{Y}^i_\beta - \underline{L}_\beta \hat{Y}^i_\alpha +\frac{1}{\tb}[\hat{Y}^i_\alpha,\hat{Y}^i_\beta] - A^{ij}\underline{L}_\alpha \hat{Y}^j_\beta A^{ji}+ A^{ij}\underline{L}_\beta \hat{Y}^j_\alpha A^{ji}-\frac{1}{\tb}[A^{ij} \hat{Y}_\alpha^j A^{ji}, A^{ij} \hat{Y}_\beta^j A^{ji}].\end{align*}
From (\ref{eq:5.1.5}) it follows that $$[\hat{Y}^i_\alpha, \hat{Y}^i_\beta]=[A^{ij} \hat{Y}_\alpha^j A^{ji}+ A^{ij} \underline{L}_\alpha A^{ji}, A^{ij} \hat{Y}_\beta^j A^{ji}+ A^{ij}\underline{L}_\beta A^{ji}].$$
So the l.h.s. of (\ref{eq:5.2.2}) equals
\begin{align*}
&\underline{L}_\alpha \hat{Y}^i_\beta - \underline{L}_\beta \hat{Y}^i_\alpha- A^{ij}\underline{L}_\alpha \hat{Y}^j_\beta A^{ji}
+ A^{ij}\underline{L}_\beta \hat{Y}^j_\alpha A^{ji} +\\+ &\frac{1}{\tb}[A^{ij} \hat{Y}_\alpha^j A^{ji},A^{ij}\underline{L}_\beta A^{ji}]
+\frac{1}{\tb}[A^{ij} \underline{L}_\alpha A^{ji}, A^{ij} \hat{Y}_\beta^j A^{ji}]+\frac{1}{\tb}[A^{ij} \underline{L}_\alpha A^{ji}, A^{ij}\underline{L}_\beta A^{ji}].
\end{align*}
It is pretty straightforward to see that the last expression coincides with the r.h.s. of (\ref{eq:5.2.2})
(we remark that $A^{ij}\underline{L}_\alpha A^{ji}=-(\underline{L}_\alpha A^{ij}) A^{ji}$ because
$A^{ij}A^{ji}=1$).

Let us check (\ref{eq:5.2.3}). We can rewrite the l.h.s. as
\begin{align*}
&c^{ij}(\alpha)+A^{ij}c^{jk}(\alpha)A^{ji}+A^{ij}A^{jk}c^{ki}(\alpha)A^{kj}A^{ji}=\\
&-A^{ij}\underline{L}_\alpha A^{ji}- A^{ij}A^{jk}(\underline{L}_\alpha A^{kj})A^{ji}-
A^{ij}A^{jk}A^{ki}(\underline{L}_\alpha A^{ik})A^{kj}A^{ji}.
\end{align*}
On the other hand, the r.h.s. of (\ref{eq:5.2.3}) equals
\begin{align*}
&[\underline{L}_\alpha(A^{ij}A^{jk}A^{ki})]A^{ik}A^{kj}A^{ji}=\\
&=(\underline{L}_\alpha A^{ij})A^{ji}+ A^{ij}(\underline{L}_\alpha A^{jk})A^{kj}A^{ji}+ A^{ij}A^{jk}(\underline{L}_\alpha A^{ki})A^{ik}A^{kj}A^{ji}.
\end{align*}
Again, it is easy to see that these two expressions are the same.

Finally, let us prove (\ref{eq:5.2.4}).
\begin{align*}
&\ln(A^{ij}A^{jk}A^{ki})+\ln(A^{ik}A^{kl}A^{li})-\ln(A^{ij}A^{jl}A^{li})-\ln(A^{jk}A^{kl}A^{lj})=\\
&=\ln(A^{ij}A^{jk}A^{kl}A^{li})-\ln(A^{jk}A^{kl}A^{lj})-\ln(A^{jl}A^{li}A^{ij})=\\
&=\ln(A^{ij}A^{jk}A^{kl}A^{li})-\ln(A^{jk}A^{kl}A^{li}A^{ij})=0.
\end{align*}
\end{proof}

\begin{proof}[Proof of Theorem \ref{Thm:2.0}]
Let $X^{ij},\alpha\mapsto Y_\alpha^i$ be  as in  Lemma \ref{Lem:2.0.1}. Thanks to Lemma \ref{Lem:2.0.2},
we may assume that  $c^{ijk},c^{ij}(\alpha),c^i(\alpha,\beta)$ vanish. We are going to show that
there are
$\Gamma\times\Xi$-invariant elements $X^i\in \param\z^\tb(\CC(\underline{\Hsh}^\wedge))(U_i)$
of degree 2 with respect to $\K^\times$-action
such that the following condition (*) holds
\begin{itemize}
\item[(*)] $\Theta^{i}:=\exp(\frac{1}{\tb}\ad X^i)\Theta^{U_i}, \overline{X}^{ij}:=\tb\ln(\exp(\frac{1}{\tb}X^i)\exp(\frac{1}{\tb}X^{ij})\exp(-\frac{1}{\tb}X^j))$
    satisfy conditions (1)-(7) of Theorem \ref{Thm:2.0}, while $\overline{Y}^i_\alpha:= \exp(\frac{1}{\tb}\ad X^i)Y^i_\alpha+ \exp(\frac{1}{\tb}X^i)\underline{\hat{L}}_\alpha \exp(-\frac{1}{\tb} X^i)$ is zero.
\end{itemize}

The proof is in several  steps.

{\it Step 1.}
We can rewrite (\ref{eq:5.1.5}),(\ref{eq:5.1.4}) as
\begin{align}
&c^{ij}(\alpha)=Y^i_\alpha -\exp(\frac{1}{\tb}\ad X^{ij})Y^j_\alpha+ \exp(\frac{1}{\tb}X^{ij})
\underline{\hat{L}}_\alpha \exp(-\frac{1}{\tb}X^{ij}),\\\label{eq:5.1.new}
& c^i(\alpha,\beta)=\underline{\hat{L}}_\alpha Y^\beta- \underline{\hat{L}}_\beta Y^i_\alpha+ \frac{1}{\tb}[Y^i_\alpha,Y^i_\beta].
\end{align}
We remark that
$\Theta^{i},\overline{X}^{ij},\overline{Y}^i_\alpha$ obtained from $\Theta^{U_i},X^{ij}, Y^i_\alpha$ by using the formulas above still satisfy  conditions (2),(3),(5) of Theorem  \ref{Thm:2.0},
equality (\ref{eq:5.1.3}), and the vanishing conditions $c^{ijk}=0,c^{ij}(\alpha)=0,c^i(\alpha,\beta)=0$. This is checked by computations similar in spirit to those of Lemma \ref{Lem:2.0.2}. Clearly, the equality $c^{ijk}=0$
is just condition (6) of Theorem \ref{Thm:2.0}.

{\it Step 2.}
We consider $\alpha\mapsto \underline{\hat{L}}_\alpha$ as a connection $\underline{\hat{L}}$ on the sheaf
$\CC(\underline{\Hsh}^\wedge)$, i.e., as a sequence of maps $\underline{\hat{L}}^i: \CC(\underline{\Hsh}^\wedge)\otimes \Omega^i_{\underline{\Leaf}}\rightarrow \CC(\underline{\Hsh}^\wedge)\otimes \Omega^{i+1}_{\underline{\Leaf}}$.
The connection $\underline{\hat{L}}$ is flat.

The flat sheaf $\CC(\underline{\Hsh}^\wedge)$ is just the product of several copies of the sheaf $\prod_{i=0}^\infty S^i \Omega^1_{\underline{\Leaf}}$ with the ``diagonal'' connection: $\underline{\hat{L}}_\alpha$ is the difference of the ``fiberwise'' and the  ``base'' derivations in the direction of $\alpha$.  We claim that all higher cohomology of $\underline{\hat{L}}$ on $\prod_{i=0}^\infty S^i \Omega^1_{\underline{\Leaf}}$ vanish, while the 0th cohomology is identified with $\Str_{\underline{\Leaf}}$, an isomorphism $\Str_{\underline{\Leaf}}\rightarrow H^0(\prod_{i=0}^\infty S^i \Omega^1_{\underline{\Leaf}},\underline{\hat{L}})$ maps a section $f$ of $\Str_{\underline{\Leaf}}$
to its ``Taylor expansion''
\begin{equation}\label{eq:5.1.9}
T(f):=\sum_{\mathbf{i}} \frac{1}{\mathbf{i}!}\frac{\partial^{\mathbf{i}}f}{\partial\alpha^{\mathbf{i}}}\alpha^{\mathbf{i}},
\end{equation}
where $\mathbf{i}=(i_1,\ldots,i_n), \mathbf{i}!:=\prod_{j=1}^n i_j!,\alpha^{\mathbf{i}}=\prod_{j=1}^n (\alpha^j)^{i_j}$, etc.

To compute the cohomology we extend the map $T:\K[U]\rightarrow  \prod_{i=0}^\infty S^i \Omega^1_{\underline{\Leaf}}(U)$ to a continuous automorphism of $\prod_{i=0}^\infty S^i \Omega^1_{\underline{\Leaf}}(U)$ by requiring that it is the identity on the subspace $V_0\subset S^1\Omega^1_{\underline{\Leaf}}(U)$ of constant 1-forms. Then $T\circ \underline{\hat{L}}_\alpha\circ T^{-1}$  is (up to a sign)  the fiberwise derivation in the direction of $\alpha$. Now the claim is clear.

{\it Step 3.} Let us suppose that $Y^i_\alpha\in \param^k \z^\tb(\CC(\underline{\Hsh}^\wedge))(U_i), k\geqslant 1$.
Then $\frac{1}{\tb}[Y^i_\alpha, Y^i_\beta]\in \param^{k+1}\z^\tb(\CC(\underline{\Hsh}^\wedge))(U_i)$.
From (\ref{eq:5.1.new}) it follows that modulo
$\param^{k+1} \z^\tb(\CC(\underline{\Hsh}^\wedge))(U_i)$ we have
$\underline{\hat{L}}_\alpha Y^i_\beta-\underline{\hat{L}}_\beta Y^i_\alpha=0$. This means that modulo
$\param^{k+1} \z^\tb(\CC(\underline{\Hsh}^\wedge))(U_i)$
the map $\alpha\mapsto Y^i_\alpha$ is a 1-cocycle in the complex considered on Step 2. Thus
there is $X^i_k\in \param^k \z^\tb(\CC(\underline{\Hsh}^\wedge))(U_i))$
such that $Y^i_\alpha+\underline{\hat{L}}_\alpha X^i_k\in \param^{k+1}\z^\tb(\CC(\underline{\Hsh}^\wedge))(U_i)$.
In addition, we can assume
that $X^i_k$ satisfies the required equivariance conditions. Let us replace
$Y^i_\alpha$ with the expression analogous to $\overline{Y}^i_\alpha$ with $X^i_k$ instead of $X^i$.
So we achieve $Y^i_\alpha\in \param^{k+1} \z^\tb(\CC(\underline{\Hsh}^\wedge))(U_i)$.

Set $X^i=\tb\ln(\ldots\exp(\frac{1}{\tb}X^i_2)\exp(\frac{1}{\tb}X^i_1))$, the
expression converges because we have chosen $X^i_k$ in $ \param^k \z^\tb(\CC(\underline{\Hsh}^\wedge))(U_i))$.
With this $X^i$ the element $\overline{Y}^i_\alpha$ is 0. Also it is clear from the construction that
(\ref{eq:5.1.3}) becomes condition (4) of Theorem \ref{Thm:2.0}.

{\it Step 4.} It remains to show that (7) of Theorem \ref{Thm:2.0} holds, i.e., $\underline{\hat{L}}_\alpha \overline{X}^{ij}=0$.
The equalities (\ref{eq:5.1.3})  imply that $b^{ij}_\alpha:=\underline{\hat{L}}_\alpha\overline{X}^{ij}\in \param \K[U_{ij}][[\param^*]]$. The equality $c^{ij}(\alpha)=0$
can be rewritten as $b^{ij}_\alpha=0$.
\end{proof}

\subsection{Proof of Theorem \ref{Thm:2.0I}}\label{SUBSECTION ProofI}
Define the sheaf of algebras $\Fl(\Hsh^\wedge)$ on $\underline{\Leaf}$ as follows.
For an open subset $U\subset \underline{\Leaf}$ let $\Fl(\Hsh^\wedge)(U)$ denote the space
of flat (with respect to the connection $L_\bullet$) sections of $\Hsh^\wedge(U)$. This is
a $\K^\times\times \Xi$-equivariant sheaf of $\K$-algebras (but not of $\Str_{\underline{\Leaf}}$-algebras)
on $\underline{\Leaf}$.

Similarly, define the sheaf $\Fl(\underline{\Hsh}^\wedge)$ using the connection $\alpha\mapsto
\underline{\hat{L}}_\alpha$.

We remark that $\underline{\Halg}^{\wedge_{\underline{\Leaf}}}|_{\underline{\Leaf}}=\W_{\tb}|_{\underline{\Leaf}}
\widehat{\otimes}_{\K[[\tb]]} \underline{\Halg}^{+\wedge_0}, \Fl(\underline{\Hsh}^\wedge)=\Fl(\Str_{\underline{\Leaf}}\widehat{\otimes} \W_\tb|_{\underline{\Leaf}})\widehat{\otimes}_{\K[[\tb]]}\underline{\Halg}^{+\wedge_0}$. Following the argument
of Step 2 of the proof of Theorem \ref{Thm:2.0}, we get an isomorphism $\W_\tb|_{\underline{\Leaf}}=\Str_{\underline{\Leaf}}[[\tb]]\xrightarrow{\sim}\Fl(\Str_{\underline{\Leaf}}\widehat{\otimes} \W_\tb^{\wedge_\tb})$. This gives rise to an isomorphism
\begin{equation}\label{eq:2.0.4} T\otimes \operatorname{id}: \underline{\Halg}^{\wedge_{\underline{\Leaf}}}=\W^{\wedge_\tb}_{\tb}\widehat{\otimes}_{\K[[\tb]]} \underline{\Halg}^{+\wedge_0}\rightarrow \Fl(\Str_{\underline{\Leaf}}\widehat{\otimes} \W_{\tb}^{\wedge_\tb})\widehat{\otimes}_{\K[[\tb]]}\underline{\Halg}^{+\wedge_0}=
\Fl(\underline{\Hsh}^\wedge),\end{equation}
which will also be denoted by $T$.


Recall the elements $\overline{X}^{ij}\in \Fl(\CC(\underline{\Hsh}^\wedge(U_{ij})))$ from Theorem \ref{Thm:2.0}.
To simplify the notation we will write $X^{ij}$ instead of $\overline{X}^{ij}$.
Use these elements to twist the sheaf $\CC(\underline{\Halg}^{\wedge_{\underline{\Leaf}}})|_{\underline{\Leaf}}$
as explained in the discussion preceding Theorem \ref{Thm:2.0I}. Theorem \ref{Thm:2.0}
implies that the maps $T^{-1}\circ\Theta^U$ induce an isomorphism $\eta:\Fl(\Hsh^\wedge)\rightarrow \CC(\underline{\Halg}^{\wedge_{\underline{\Leaf}}})|_{\underline{\Leaf}}^{tw}$. This isomorphism
is $\K^\times\times\Gamma\times\Xi$-equivariant.

We have a natural embedding $\Halg\hookrightarrow  \Hsh(\underline{\Leaf})$.
Its image consists of flat and $\Xi$-invariant sections. So
$\Halg$ embeds into $\Fl(\Hsh^\wedge)(\underline{\Leaf})^\Xi$. Composing this homomorphism
with $\eta$, we obtain a homomorphism \begin{equation}\label{eq:2.2.51}\theta:\Halg\rightarrow \left[\CC(\underline{\Halg}^{\wedge_{\underline{\Leaf}}})|_{\underline{\Leaf}}^{tw}(\underline{\Leaf})\right]^\Xi.\end{equation}
We are going to prove that this homomorphism extends to an isomorphism $\theta: \Halg^{\wedge_{\Leaf}}|_{\Leaf}\rightarrow \left[\CC(\underline{\Halg}^{\wedge_{\underline{\Leaf}}})|_{\underline{\Leaf}}^{tw}\right]^{\Xi}$ that equals $\theta_0$
modulo $(\param)$.

First of all, we remark that $\theta_0:\Halg_{(r+1)}\rightarrow \CC(\underline{\Halg}^{\wedge_{\underline{\Leaf}}}_{(r+1)}|_{\underline{\Leaf}})$ coincides
with $T^{-1}\circ\Theta_0$ and so also coincides with the map induced by (\ref{eq:2.2.51}).
Localizing the  map (\ref{eq:2.2.51}) over $\Leaf$, we get a  homomorphism
\begin{equation}\label{eq:2.2.52}\theta: \Halg|_{\Leaf}\rightarrow \left[\CC(\underline{\Halg}^{\wedge_{\underline{\Leaf}}})|_{\underline{\Leaf}}^{tw}\right]^{\Xi}\end{equation}
of sheaves of algebras that coincides with $\theta_0:\Halg_{(r+1)}|_{\Leaf}\rightarrow \left[\CC(\underline{\Halg}_{(r+1)}^{\wedge_{\underline{\Leaf}}})|_{\underline{\Leaf}}^{tw}\right]^{\Xi}$
modulo $(\param)$.

We claim that (\ref{eq:2.2.52}) is continuous with respect to the $\p|_{\Leaf}$-adic topology
on $\Halg|_{\Leaf}$. To show this, it is enough to verify that (\ref{eq:2.2.51}) is continuous
(with respect to the $\p$-adic topology on $\Halg$ and the topology on $\Fl(\CC(\underline{\Hsh}^\wedge)^{tw})^\Xi$ induced from the $\CC(\underline{\m})$-adic topology on $\CC(\underline{\Hsh}^\wedge)$). This will follow if we check that $\theta(\p)\subset \CC(\underline{\m})(\underline{\Leaf})$.
Since both ideals contain $\param$, it is enough to show that $\theta_0(\p_{(r+1)})\subset \CC(\underline{\m}_{(r+1)})(\underline{\Leaf})$.
But this is straightforward from the definition of $\theta_0$.

The sheaf $\Halg^{\wedge_\Leaf}|_{\Leaf}$ is the completion of $\Halg|_{\Leaf}$ with respect to the
ideal $\p|_{\Leaf}$.
So we see that (\ref{eq:2.2.52}) extends to a homomorphism $\theta:\Halg^{\wedge_\Leaf}|_{\Leaf}\rightarrow \left[\CC(\underline{\Halg}^{\wedge_{\underline{\Leaf}}})|_{\underline{\Leaf}}^{tw}\right]^{\Xi}$ that coincides with
$\theta_0: \Halg_{(r+1)}^{\wedge_\Leaf}|_{\Leaf}\rightarrow \left[\CC(\underline{\Halg}_{(r+1)}^{\wedge_{\underline{\Leaf}}})|_{\underline{\Leaf}}^{tw}\right]^{\Xi}$
modulo $(\param)$. But we have seen that $\theta_0$ is an isomorphism. Since
$(\CC(\underline{\Halg}^{\wedge_{\underline{\Leaf}}})|_{\underline{\Leaf}}^{tw})^\Xi$
is $\K[[\param^*]]$-flat and $\Halg^{\wedge_\Leaf}|_{\Leaf}$ is complete in the $(\param)$-adic topology,
we see that $\theta$ is an isomorphism. The equality $\theta(\p^{\wedge_\Leaf}|_{\Leaf})
=[\CC(\underline{\p}^{\wedge_{\underline{\Leaf}}})|^{tw}_{\underline{\Leaf}}]^{\Xi}$ follows from the construction.

\section{Harish-Chandra bimodules}\label{SECTION_HC}
\subsection{Content of the section}
The goal of this section is to study Harish-Chandra bimodules over SRA's.

In Subsection \ref{SUBSECTION_Poisson} we introduce notions of {\it noncommutative}
Poisson algebras and their Poisson bimodules. These are algebras and bimodules equipped with an
additional structure: a Poisson bracket. A similar notion already appeared, for instance,
in \cite{BeKa}. After giving all necessary definitions we study some simple properties
of Poisson bimodules.

In Subsection \ref{SUBSECTION_HC_bimod} we define Harish-Chandra (shortly, HC) $\Halg$-bimodules
as graded finitely generated Poisson $\Halg$-bimodules. Then we define HC $\Hrm_{1,c}$-bimodules.
Also in this subsection we state the main result in our study of HC bimodules, Theorem \ref{Thm:5I},
and its version for the $\Hrm$-algebras: Theorem \ref{Thm:5}. These theorems claim that there are functors
 between certain categories of Harish-Chandra bimodules similar to the functors in \cite{HC},
 Theorem 1.3.1.

The proof of Theorem \ref{Thm:5I} occupies the next four subsections whose content
will be described in Subsection \ref{SUBSECTION_HC_bimod}. In the Subsection \ref{SUBSECTION_Thm5_proof}
we derive Theorem \ref{Thm:5} as well as Theorems \ref{Thm:4},\ref{Thm:4'} from
Theorem \ref{Thm:5I}. Finally, in Subsection \ref{SUBSECTION_Functors_b} we will provide
more simple minded versions of our functors. The constructions of this subsection will be used
in Section \ref{SECTION_Cherednik}.



\subsection{Notation and conventions}\label{SUBSECTION_HC_not}
{\it Algebras.}
Set $\tilde{\Halg}:=\Halg[\hbar]/(\tb-\hbar^2)$.
This is a flat $S(\tilde{\param})$-algebra, where $\tilde{\param}$ is a vector space
with basis $\hbar,\cb_1,\ldots,\cb_r$. Similarly define $\underline{\tilde{\Halg}},\underline{\tilde{\Halg}}^+,\W_\hbar:=\W_{\tb}[\hbar]/(\tb-\hbar^2)$. Recall that
$\CC(\underline{\Halg})$ stands for $Z(\Gamma,\underline{\Gamma}, \underline{\Halg})$.
We have the two-sided ideal $\tilde{\p}$ in $\tilde{\Halg}$ generated by $\p$
and $\hbar$. Similarly, we have the ideals $\underline{\tilde{\p}}\subset \underline{\tilde{\Halg}},
\underline{\tilde{\p}}^+\subset \underline{\tilde{\Halg}}^+$.

Then we can form the
sheafified (and completed) versions  $\tilde{\Halg}^{\wedge_\Leaf}|_{\Leaf}:=
\Halg^{\wedge_\Leaf}|_{\Leaf}[\hbar]/(\tb-\hbar^2), \underline{\tilde{\Halg}}^{\wedge_{\underline{\Leaf}}}|_{\underline{\Leaf}}\cong \W_\hbar|_{\underline{\Leaf}}\widehat{\otimes}_{\K[[\hbar]]}\underline{\tilde{\Halg}}^{+\wedge_0}$
of $\tilde{\Halg},\underline{\tilde{\Halg}}$.
We have an isomorphism $$\tilde{\theta}:\tilde{\Halg}^{\wedge_\Leaf}|_{\Leaf}\rightarrow \left(\CC(\W_\hbar|_{\underline{\Leaf}}\widehat{\otimes}_{\K[[\hbar]]}\underline{\tilde{\Halg}}^{+\wedge_0})^{tw}\right)^{\Xi}$$
induced by $\theta$ from  Theorem \ref{Thm:2.0I}.
Modulo $(\tilde{\param})$ the isomorphism $\tilde{\theta}$ coincides with $$\theta_0:\Halg^{\wedge_\Leaf}_{(r+1)}|_{\Leaf}\rightarrow \CC(\underline{\Halg}^{\wedge_{\underline{\Leaf}}}_{(r+1)}|_{\underline{\Leaf}})^\Xi,$$ see Subsection \ref{SUBSECTION_sheafifiedI}.

The reason why we need to consider the extensions $\tilde{\Halg}$ etc. is the following.
Recall that the algebra $\Hrm_{1,c}$ is filtered. So we can form the Rees algebra
$R_\hbar(\Hrm_{1,c})$. There is a unique homomorphism $\tilde{\Halg}\rightarrow R_\hbar(\Hrm_{1,c})$
given by $\gamma\mapsto \gamma, \hbar\mapsto \hbar, v\mapsto \hbar v, \cb_i\mapsto c_i \hbar^2, i=1,\ldots,r$.
It is clear that this homomorphism is surjective. So $R_\hbar(\Hrm_{1,c})$ is represented as
a quotient of $\tilde{\Halg}$.

{\it $D$-modules.} For a smooth algebraic variety $X$ let $\DCal_X$ denote the sheaf of linear differential
operators on $X$. For a $\DCal_X$-module $\M$ let $\Fl(\M)$ denote the space of flat sections
of $\M$. Recall that a section is called flat if it is  annihilated by all vector fields.

Let $\M$ be a $\DCal_X$-module and $X$ be equipped with an action of a group $G$. Recall that $\M$
is said to be weakly $G$-equivariant if $\M$ is equipped with an action of $G$ making the action
map $\DCal_X\otimes \M\rightarrow \M$ equivariant.

{\it Annihilators.} For an $\A$-bimodule $\M$ let $\LAnn(\M),\RAnn(\M)$ denote the left and the right
annihilators of $\M$ in $\A$.

\subsection{Poisson algebras and bimodules}\label{SUBSECTION_Poisson}
In this subsection we will introduce the notions of (not necessarily commutative) Poisson algebras
and their Poisson bimodules.

\begin{defi}\label{defi:3.9.1}
Let $\A$ be an associative unital $\K[\tb]$-algebra. We say that $\A$ is
{\it Poisson} if it is equipped with
a $\K[\tb]$-bilinear map $\z^{\tb}(\A)\otimes \A\rightarrow \A, z\otimes a\mapsto \{z,a\}$ such that
$\z^{\tb}(\A)$ is closed with respect to $\{\cdot,\cdot\}$ and
\begin{align}\label{eq:4.10.0}
& \{z,z\}=0, \\
\label{eq:4.10.1}
& \{\tb a,b\}=[a,b],\\\label{eq:4.10.2}
& \{z,ab\}=\{z,a\}b+ a\{z,b\}, \\\label{eq:4.10.3}
& \{z_1z_2,a\}=\{z_1,a\}z_2+z_1\{z_2,a\}, \\\label{eq:4.10.4}
& \{\{z_1,z_2\},a\}=\{z_1,\{z_2,a\}\}-\{z_2,\{z_1,a\}\},\\\nonumber
& \forall z,z_1,z_2\in \z^{\tb}(\A), a,b\in \A.
\end{align}
\end{defi}

\begin{defi}\label{defi:3.9.2}
Let $\A$ be a Poisson $\K[\tb]$-algebra and $\M$ be an $\A$-bimodule such that the left and right
actions of $\K[\tb]$ on $\M$ coincide. We say that $\M$ is a Poisson $\A$-bimodule if it is equipped
with a $\K[\tb]$-bilinear map $\z^{\tb}(\A)\otimes \M\rightarrow \M$ satisfying the following equalities:
\begin{align}\label{eq:4.11.1}
& \{\tb a,m\}=[a,m],\\\label{eq:4.11.2}
& \{z,am\}=\{z,a\}m+ a\{z,m\}, \{z,ma\}=\{z,m\}a+ m\{z,a\},\\\label{eq:4.11.3}
& \{z_1z_2,m\}=\{z_1,m\}z_2+z_1\{z_2,m\}, \\\label{eq:4.11.4}
& \{\{z_1,z_2\},m\}=\{z_1,\{z_2,m\}\}-\{z_2,\{z_1,m\}\}, \\\nonumber
& \forall z,z_1,z_2\in \z^{\tb}(\A), a\in \A, m\in \M.
\end{align}
\end{defi}

For instance, suppose $\A$ is commutative with zero action of $\tb$. Then we get the usual
definition of a Poisson algebra. Further, if $\M$ is an $\A$-module (that can be considered as
an $\A$-bimodule, where the left and  right actions coincide) and $\tb$ acts  on $\M$ by 0,
then we get a more standard notion of a Poisson $\A$-module.

As the other extreme, suppose $\A$ is $\K[\tb]$-flat. Then the bracket on $\A$ is uniquely
recovered from the multiplication: $\{z,a\}=\frac{1}{\tb}[z,a]$ (compare with Subsection \ref{SUBSECTION_spherical}).
Similarly, if $\M$ is $\K[\tb]$-flat, then $\{z,m\}=\frac{1}{\tb}[z,m]$.

Examples of Poisson algebras (in our sense) include $\Halg, \underline{\Halg},\CC(\underline{\Halg})$
and also Weyl algebras.

By a Poisson ideal in a Poisson algebra $\A$ we mean a two-sided ideal $\I$ that is a Poisson
sub-bimodule in $\A$. We remark that the quotient $\A/\I$ is, of course, a Poisson $\A$-bimodule
but  has no natural structure of a Poisson algebra.

\begin{Lem}\label{Lem:4.2.2}
Let $\A$ be a Poisson $\K[\tb]$-algebra and  $\M$ be a Poisson $\A$-bimodule.
\begin{enumerate}
\item If $\Ncal$ is another Poisson $\A$-bimodule, then $\M\otimes_\A\Ncal$ has a natural structure
of a Poisson $\A$-bimodule.
\item If $\I\subset \A$ is a Poisson two-sided ideal, then $\I\M\subset \M$ is a Poisson sub-bimodule.
\item Left and right annihilators of $\M$ in $\A$  are Poisson ideals.
\item The left and right actions of $\z^{\tb}(\A)/\tb\A$ on $\Malg/\tb \Malg$  coincide.
\end{enumerate}
\end{Lem}
\begin{proof}
(1): Define the bracket on $\M\otimes_\A\Ncal$ by $\{z,m\otimes n\}=\{z,m\}\otimes n+m\otimes \{z,n\}$.
It is easy to see that the bracket is well-defined and turns $\M\otimes_\A\Ncal$ into a Poisson bimodule.
The proofs of (2)-(4) follow directly from the definitions.
%
\end{proof}

Now suppose that $\K^\times$ acts on a Poisson algebra $\A$ by a $\K$-algebra automorphisms
such that $\tb$ has degree 2, while the bracket has degree $-2$. We say that a Poisson
$\A$-bimodule $\M$ is $\K^\times$-equivariant if it is equipped with a $\K^\times$-action
such that the multiplication map $\A\otimes\M\otimes \A\rightarrow \M$ is $\K^\times$-equivariant,
and the bracket map $\z^{\tb}(\A)\otimes \M\rightarrow \M$ has degree $-2$. We say that
a $\K^\times$-equivariant Poisson bimodule is graded if the $\K^\times$-action comes from a
grading.


Let $\A$ be a Poisson algebra. Set $\tilde{\A}:=\A[\hbar]/(\tb-\hbar^2)$. Then we still
say that $\tilde{\A}$ is a Poisson algebra (the bracket is extended to $\z^{\tb}(\A)\otimes \tilde{\A}\rightarrow
\tilde{\A}$ by the right $\K[\hbar]$-linearity). By a Poisson $\tilde{\A}$-bimodule we mean
a bimodule $\M$, where the left and right actions of $\K[\hbar]$ coincide, equipped with a map $\z^{\tb}(\A)\otimes\M\rightarrow\M$ that is $\K[\tb]$-linear in the
first argument, $\K[\hbar]$-linear in the second one and satisfies (\ref{eq:4.11.1})-(\ref{eq:4.11.4}).


Now suppose that $\A$ is a flat (Poisson) $\K[\tb]$-algebra and $e\in \A$ is an idempotent
such that we have a ``Satake isomorphism'', i.e., the map $\z(\A/\tb \A)\xrightarrow{e\cdot} \z(e(\A/\tb \A)e)$  is an isomorphism. It follows that $\z^{\tb}(e\A e)=e\z^{\tb}(\A)e$.

 We are going to relate Poisson $e\A e$- and
$\A$-bimodules.

First of all, let us note that $e\A e$ has a natural Poisson bracket. Namely, we set
\begin{equation}\label{eq:4.10.21}\{eze, eae\}=e\{z,eae\}e, z\in \z^{\tb}(\A), a\in \A.\end{equation}
The bracket on $e\A e$ is well-defined.
To check this let $z,z'\in \z^{\tb}(\A)$ be such that $eze=ez'e$. We need to show that $e\{z,eae\}e=e\{z',eae\}e$. Thanks to the Satake
isomorphism above, we see that $z-z'\in \tb \A$. Our claim follows easily from $\{\tb a,b\}=[a,b]$.
Also it is straightforward
to check that (\ref{eq:4.10.21})  satisfies (\ref{eq:4.10.0})-(\ref{eq:4.10.4}).

Now let $\M$ be a Poisson $\A$-bimodule. We claim that $e\M e$ can be  equipped with a natural Poisson bracket.
Namely, we set
\begin{equation}\label{eq:idemp_bracket}\{eze,eme\}=e\{z,eme\}e\end{equation} for any $z\in \z^{\tb}(\A), m\in \M$. Let us check that this bracket is well-defined.
%
It is straightforward to verify that the bracket satisfies (\ref{eq:4.11.1})-(\ref{eq:4.11.2}).
So we get a functor $\M\mapsto e\M e$ from the category of Poisson $\A$-bimodules to the category
of Poisson $e\A e$-bimodules. This functor is exact because $\M$ is completely reducible as a
$\Span_\K(1,e)$-bimodule.

The functor $\M\mapsto e\M e$ has a right inverse. 
Let $\Ncal$ be a Poisson $e\A e$-bimodule. Set $\widetilde{\Ncal}:=\A e_{\otimes e\A e} \Ncal_{\otimes e \A e} e\A$.
Let us equip this bimodule with a Poisson bracket. 
For $z\in \z^{\tb}(\A)$ we set $\{z, ae\otimes n\otimes eb\}=\{z,ae\}e\otimes n\otimes eb+ ae\otimes \{eze,n\}\otimes eb+ ae\otimes n\otimes e\{z,eb\}$. The verification that this bracket is well-defined
is analogous to the previous paragraph. 
The bracket on $\widetilde{\Ncal}$ satisfies (\ref{eq:4.11.1})-(\ref{eq:4.11.4}). Also it is easy to see that $e\widetilde{\Ncal}e=e\A\otimes_{\A}\widetilde{\Ncal}\otimes_\A \A e$ is canonically identified with $\Ncal$.
On the other hand, $\widetilde{e\M e}$ is naturally embedded into $\M$ with image $\A e\M e \A$.

\subsection{Harish-Chandra bimodules}\label{SUBSECTION_HC_bimod}
In this subsection we will introduce the notion of a Harish-Chandra bimodule over $\tilde{\Halg}$.
Using this notion we will define  Harish-Chandra $\Hrm_{1,c}$-$\Hrm_{1,c'}$ bimodules.

Recall that the algebra $\tilde{\Halg}$ is a graded Poisson algebra in the sense of the previous subsection.

\begin{defi}\label{defi_HC1}
By a Harish-Chandra (shortly, HC) $\tilde{\Halg}$-bimodule we mean a graded Poisson $\tilde{\Halg}$-bimodule
$\Malg$  that is finitely generated as an $\tilde{\Halg}$-bimodule.
\end{defi}

The category of HC bimodules is denoted by $\HC(\tilde{\Halg})$.
By definition, a morphism between two objects in $\HC(\tilde{\Halg})$ is a grading preserving
homomorphism of graded Poisson bimodules.

\begin{Rem}\label{Rem_HC}
Since the grading on $\tilde{\Halg}$ is positive, each graded subspace $\Malg_i$ is finite dimensional, and is zero for $i$ sufficiently small.
\end{Rem}

In the sequel we will also need the following simple lemma.

\begin{Lem}\label{Lem:3.10.2}
$\Malg\in \HC(\tilde{\Halg})$ is finitely generated both as a left and as a
right $\tilde{\Halg}$-module.
\end{Lem}
\begin{proof}
By definition, $\Malg$ is finitely generated as an $\tilde{\Halg}$-bimodule.
But $\tilde{\Halg}$ is finite as a left (or right) $\Zalg$-module.
So there is a finite dimensional graded subspace $M\subset \Malg$ with
$\Malg= \tilde{\Halg} M\Zalg$. Thanks to (\ref{eq:4.11.1}), $[\Zalg,\Malg]\subset \tb\Malg$.
So $\Malg= \tilde{\Halg} M+\tb\Malg$. It follows that $\Malg=\tilde{\Halg} M+\tb^n\Malg$ for any
$n$. Since $\tilde{\Halg} M$ is a graded subspace in $\Malg$, we deduce that any graded component
of $\Malg$ lies in $\tilde{\Halg} M$. So $\tilde{\Halg} M=\Malg$. Similarly, $\Malg=M\tilde{\Halg}$.
\end{proof}


\begin{defi}\label{defi_HC2}
Let $\M$ be an $\Hrm_{1,c}$-$\Hrm_{1,c'}$-bimodule. We say that $\M$ is HC if there is
a filtration $\F_i\M$ compatible with the filtrations $\F_j\Hrm_{1,c},\F_j\Hrm_{1,c'}$ and such that
$R_\hbar(\M)$ considered as an $\tilde{\Halg}$-bimodule (recall that $R_\hbar(\Hrm_{1,c}), R_\hbar(\Hrm_{1,c'})$ are quotients of $\tilde{\Halg}$) is HC,  where  $\hbar^i m$ has degree $i$,
and the bracket is uniquely recovered from the bimodule structure, see remarks after Definition \ref{defi:3.9.2}.
\end{defi}


The category of HC $\Hrm_{1,c}$-$\Hrm_{1,c'}$-bimodules will be denoted by
$_{\,c\!}\HC(\Hrm)_{c'}$.

The category $\HC(\tilde{\Halg})$ is monoidal, the tensor product functor is just tensoring over
$\tilde{\Halg}$. Similarly, we have the  bifunctor $$_{\,c\!}\HC(\Hrm)_{c'}\times _{\,c'\!}\HC(\Hrm)_{c''}\rightarrow _{\,c\!}\HC(\Hrm)_{c''}$$
given by tensoring over $\Hrm_{1,c'}$. To see that the product is HC one needs to notice that
for $\M_1\in _{\,c\!}\HC(\Hrm)_{c'}, \M_2\in _{\,c'\!}\HC(\Hrm)_{c''}$  the natural morphism  $R_\hbar(\M_1)\otimes_{\tilde{\Halg}}R_\hbar(\M_2)/(\hbar-1)
\rightarrow \M_1\otimes_{\Hrm_{1,c'}}\M_2$ is an isomorphism, compare with \cite{HC}, Proposition 3.4.1.


The category $\HC(\tilde{\Halg})$ has a direct sum decomposition that will be of importance later.
Namely, let $\Malg\in \HC(\tilde{\Halg})$. Consider the linear operators $C_i:=\{\cb_i,\cdot\}$ on $\Malg$.
Since $\{\cb_i,x\}=0$ for all $x\in \tilde{\Halg}$, we see that $C_i$ is an endomorphism of the
Poisson bimodule $\Malg$. Moreover,
the degree of $\cb_i$ is 2, while $\{\cdot,\cdot\}$ decreases degrees by 2. So $C_i$ is a graded
endomorphism of $\Malg$. Since $\Malg$ is generated by a finite number of the
graded components, we see that $\Malg$ decomposes into the finite direct sum $\Malg=\bigoplus_\lambda \Malg^\lambda$, where $\lambda\in \param_{(1)}^*$ and $$\Malg^\lambda:=\{m\in \Malg| (C_i-\langle\lambda,\cb_i\rangle)^n m=0 \text{ for some }n\}.$$ This decomposition produces the category decomposition
\begin{equation}\label{eq:3.13.1}\HC(\tilde{\Halg})=\bigoplus_\lambda \HC(\tilde{\Halg})^\lambda.\end{equation}

Let $\Lambda_\Halg$ denote the subspace in  $\param_{(1)}^*$ generated by all $\lambda\in \param_{(1)}^*$
such that $\HC(\tilde{\Halg})^\lambda$ is non-zero.

Let us proceed to defining the associated varieties.
Now let $\Malg\in \HC(\tilde{\Halg})$. Thanks to assertion 4 of Lemma \ref{Lem:4.2.2},  we can consider $\Malg/\tilde{\param}\Malg$ as a  $\Zalg_{(r+1)}=(SV)^\Gamma$-module. Since $\Halg_{(r+1)}$ is finite over $\Zalg_{(r+1)}$, we get that $\Malg/\tilde{\param}\Malg$ is  finitely generated
as a $\Zalg_{(r+1)}$-module. So we can define the associated variety $\VA(\Malg)$ as the support
of $\Malg/\tilde{\param}\Malg$ in $\Spec(\Zalg_{(r+1)})=V^*/\Gamma$
(the corresponding ideal in $\Zalg_{(r+1)}$ is the radical of the annihilator of $\Malg/\tilde{\param}\Malg$). Assertion (3) of Lemma \ref{Lem:4.2.2}  implies that
the annihilator of $\Malg/\tilde{\param}\Malg$ in $\tilde{\Halg}$  is a Poisson ideal.
From this it is easy to deduce that the annihilator of $\Malg/\tilde{\param}\Malg$
in $\Zalg_{(r+1)}$ is a Poisson ideal. In particular, $\VA(\Malg)$ is a Poisson subvariety in $V^*/\Gamma$.

For  $\Malg_1,\Malg_2\in \HC(\tilde{\Halg})$ we have $\Malg_1\otimes_{\tilde{\Halg}}\Malg_2/
\tilde{\param}(\Malg_1\otimes_{\tilde{\Halg}}\Malg_2)= (\Malg_1/\tilde{\param}\Malg_1)\otimes_{SV\#\Gamma}
(\Malg_2/\tilde{\param}\Malg_2)$. In particular, $\Malg_1\otimes_{\tilde{\Halg}}\Malg_2/
\tilde{\param}\Malg_1\otimes_{\tilde{\Halg}}\Malg_2$ is a quotient of
$(\Malg_1/\tilde{\param}\Malg_1)\otimes_{SV^\Gamma}
(\Malg_1/\tilde{\param}\Malg_1)$.
%
%
%
It follows that
\begin{equation}\label{eq:2.31}\VA(\Malg_1\otimes_{\tilde{\Halg}}\Malg_2)\subset \VA(\Malg_1)\cap \VA(\Malg_2).\end{equation}

Similarly, we can give the definitions of HC bimodules for the algebras $\underline{\tilde{\Halg}},\underline{\Hrm}_{1,c},
\underline{\tilde{\Halg}}^+,\underline{\Hrm}^+_{1,c}$.
Also we consider the category $\HC^{\Xi}(\underline{\tilde{\Halg}})$
consisting of all $\widetilde{\Xi}$-equivariant  bimodules $\Malg$, i.e., those equipped with
an action of $\widetilde{\Xi}$ enjoying  the following two properties:
\begin{itemize}
\item the structure maps $\underline{\tilde{\Halg}}\otimes \Malg\otimes \underline{\tilde{\Halg}}\rightarrow
\Malg, \underline{\Zalg}\otimes \Malg\rightarrow \Malg$ are $\widetilde{\Xi}$-equivariant,
\item the restriction of the $\widetilde{\Xi}$-action to $\underline{\Gamma}\subset \widetilde{\Xi}$
coincides with the adjoint action of $\underline{\Gamma}$.
\end{itemize}

For instance, $\underline{\tilde{\Halg}}\in \HC^\Xi(\underline{\tilde{\Halg}})$.

Let us state our main results concerning Harish-Chandra bimodules.

For a Poisson subvariety $Y\subset V^*/\Gamma$ define the full-subcategory
$\HC_Y(\tilde{\Halg})\subset \HC(\tilde{\Halg})$ consisting of all $\Malg\in \HC(\tilde{\Halg})$
such that $\VA(\Malg)\subset Y$. Define the full subcategory
$_{\,c\!}\HC_Y(\Hrm)_{c'}$ in ${\,_c\!}\HC(\Hrm)_{c'}$ similarly.
By (\ref{eq:2.31}), the categories $\HC_Y(\tilde{\Halg}), _{\,c\!}\HC_Y(\Hrm)_{c'}$ are closed
with respect to the tensor product functor.

Recall that we have  fixed a symplectic leaf $\Leaf\subset V^*/\Gamma$. Consider
the category $\HC_{\overline{\Leaf}}(\tilde{\Halg})$. It has the Serre subcategory
$\HC_{\partial \Leaf}(\tilde{\Halg})$, where $\partial\Leaf:=\overline{\Leaf}\setminus\Leaf$.
Form the quotient category
$$\HC_{\Leaf}(\tilde{\Halg})=\HC_{\overline{\Leaf}}(\tilde{\Halg})/\HC_{\partial \Leaf}(\tilde{\Halg}).$$
The category $\HC_{\Leaf}(\tilde{\Halg})$ inherits the tensor product functor from
$\HC_{\overline{\Leaf}}(\tilde{\Halg})$. Also the decomposition (\ref{eq:3.13.1}) descends to a decomposition $\HC_\Leaf(\tilde{\Halg})=\bigoplus_\lambda \HC_\Leaf(\tilde{\Halg})^\lambda$.

Similarly, we can define the category $_{\,c\!}\HC_{\Leaf}(\Hrm)_{c'}$.
Also we can consider categories of the form $\HC^\Xi_{Y_+}(\underline{\tilde{\Halg}}^+),_{\,c\!}\HC^\Xi_{Y_+}(\underline{\Hrm}^+)_{c'}$, where $Y_+$
is a Poisson subvariety in $V_+^*/\underline{\Gamma}$. We are particularly
interested in the case when $Y_+=\{0\}$. The category $\HC^\Xi_{0}(\underline{\tilde{\Halg}}^+)$
consists of all $\Malg\in \HC^\Xi(\underline{\tilde{\Halg}}^+)$ that are finitely generated over
$S(\tilde{\param})$, while $_{\,c\!}\HC^\Xi_{0}(\underline{\Hrm}^+)_{c'}$ consists of all finite dimensional
$\widetilde{\Xi}$-equivariant $\underline{\Hrm}^+_{1,c}$-$\underline{\Hrm}^+_{1,c'}$-bimodules.

As in (\ref{eq:3.13.1}), we have the decomposition
$\HC^\Xi(\underline{\tilde{\Halg}}^+)=\bigoplus_\lambda \HC^\Xi(\underline{\tilde{\Halg}}^+)^\lambda$.
Set $\HC^\Xi_{0}(\underline{\tilde{\Halg}}^+)^{\Lambda_\Halg}:=\bigoplus_{\lambda\in \Lambda_\Halg}\HC^\Xi_{0}(\underline{\tilde{\Halg}}^+)^\lambda$.

\begin{Thm}\label{Thm:5I}
There are functors $$\bullet_\dagger: \HC_{\overline{\Leaf}}(\tilde{\Halg})\rightarrow \HC_{0}^{\Xi}(\underline{\tilde{\Halg}}^+)^{\Lambda_\Halg},$$
$$\bullet^\dagger: \HC_0^{\Xi}(\underline{\tilde{\Halg}}^+)^{\Lambda_\Halg}\rightarrow
\HC_{\Leaf}(\tilde{\Halg})$$
with the following properties:
\begin{enumerate}
\item $\bullet_\dagger$ is  exact and descends to $\HC_{\Leaf}(\tilde{\Halg})$.
\item $\bullet^\dagger$ is right adjoint to $\bullet_\dagger$  and
$(\bullet_{\dagger})^{\dagger}\cong \id$ in $\HC_{\Leaf}(\tilde{\Halg})$.
\item The functor $\bullet_{\dagger}$  intertwines the tensor
product functors.
\item $\LAnn(\Malg_\dagger)=\LAnn(\Malg)_\dagger$ and $\RAnn(\Malg_\dagger)=\RAnn(\Malg)_{\dagger}$
for any $\Malg\in \HC_{\overline{\Leaf}}(\underline{\tilde{\Halg}})$.
\item  $\bullet_{\dagger}$ is an equivalence from $\HC_{\Leaf}(\tilde{\Halg})$
to some full subcategory of $\HC_0^{\Xi}(\underline{\tilde{\Halg}}^+)^{\Lambda_\Halg}$
closed under taking subquotients.
\end{enumerate}
\end{Thm}

\begin{Thm}\label{Thm:5}
Suppose that $c,c'\in \param_{(1)}^*$ are such that $c-c'\in \Lambda_\Halg$.
There are functors $$\bullet_\dagger: _{\,c\!}\HC_{\overline{\Leaf}}(\Hrm)_{c'}\rightarrow _{\,c\!}\HC_0^{\Xi}(\underline{\Hrm}^+)_{c'},$$
$$\bullet^\dagger: _{\,c\!}\HC_0^{\Xi}(\underline{\Hrm}^+)_{c'}\rightarrow
_{\,c\!}\HC_{\Leaf}(\Hrm)_{c'}.$$
with the following properties:
\begin{enumerate}
\item $\bullet_\dagger$ is  exact and descends to $_{\,c\!}\HC_{\Leaf}(\Hrm)_{c'}$.
\item $\bullet^\dagger$ is right adjoint to $\bullet_\dagger$  and
$(\bullet_{\dagger})^{\dagger}\cong \id$ in $_{\,c\!}\HC_{\Leaf}(\Hrm)_{c'}$.
\item The functor $\bullet_{\dagger}$  intertwines the tensor
product functors.
\item $\LAnn(\M_\dagger)=\LAnn(\M)_\dagger$ and $\RAnn(\M_\dagger)=\RAnn(\M)_{\dagger}$
for any $\M\in _{\,c\!}\HC_{\overline{\Leaf}}(\Hrm)_{c'}$.
\item  $\bullet_{\dagger}$ is an equivalence from $_{\,c\!}\HC_{\Leaf}(\Hrm)_{c'}$
to some full subcategory of $_{\,c\!}\HC_0^{\Xi}(\underline{\Hrm}^+)_{c'}$
closed under taking subquotients.
\end{enumerate}
\end{Thm}

The proof of Theorem \ref{Thm:5I} occupies basically the next four subsections.
Subsection \ref{SUBSECTION_DMod} is preparatory, there we discuss some facts about
$D$-modules we need in our construction. In (very long) Subsection \ref{SUBSECTION_fun_low_dag}
we construct the functor $\bullet_\dagger$ and study some of its properties.
In Subsection \ref{SUBSECTION_Fun_up_dag} we construct the functor $\bullet^\dagger$
and study its properties and its relation to $\bullet_\dagger$.
The constructions are very similar to the constructions of the analogous functors
in \cite{HC}, but they are more involved technically. The reason is that
in \cite{HC} we considered completions at a point, while here we need to deal with
completions along subvarieties. In Subsection \ref{Proof 5I} we prove
Theorem \ref{Thm:5I}. A key ingredient in the proof
is  Theorem \ref{Thm:3.16.1}, which is a complete analog
of Theorem 4.1.1 from \cite{HC}.


\subsection{D-modules and related objects}\label{SUBSECTION_DMod}
We will need some facts about weakly $\K^\times$-equivariant $D$-modules on certain $\K^\times$-varieties.
Namely, let $U$ be a vector space and $U^0$ be a $\K^\times$-stable open subset of $U$.


\begin{Prop}\label{Prop:0.1}
Suppose that  $\codim_U U\setminus U^0>1$.  Let $\M$ be a weakly $\K^\times$-equivariant $\DCal_{U^0}$-module
that is coherent as a $\Str_{U^0}$-module. Then there is an isomorphism $\M\cong \Str_{U^0}\otimes \Fl(\M)$
of $\DCal_{U^0}$-modules.
\end{Prop}
\begin{proof}
We remark that $U^0$ is simply connected because $\codim_{U}U\setminus U^0>1$.

By \cite{HHT}, Corollary 5.3.10,
it is enough to check  that any weakly $\K^\times$-equivariant $\DCal_{U^0}$-module that is coherent as
a $\Str_{U^0}$-module has regular singularities.
Set $\overline{U}:=\mathbb{P}(U\oplus \K)$. Equip $\overline{U}$ with a $\K^\times$-action induced
from the following action on $U\oplus \K$: $t.(u,x)=(tu,x)$. So we have  $\K^\times$-equivariant
inclusions $U^0\subset U\subset \overline{U}$.

Being a coherent $\Str_{U^0}$-module, a $\DCal_{U^0}$-module $\M$ is a vector bundle
equipped with a flat ($\K^\times$-invariant) connection, say, $\nabla$. We need to check that $\nabla$ has regular singularities. The only divisor in $\overline{U}\setminus U^0$ is $\mathbb{P}(U)=\overline{U}\setminus U$.
Pick a point $x\in \mathbb{P}(U)$ and let $[x]$ by the line in $U$  corresponding to $x$. By \cite{HHT}, Theorem 5.3.7, we only need to check that
the restriction $\nabla|_{[x]}$ of $\nabla$ to $[x]$ has regular singularities.
The restriction $\M|_{[x]^\times}$ of $\M$ to $[x]^\times=[x]\setminus \{0\}$
equivariantly trivializes. So the connection $\nabla|_{[x]^\times}$ is given
by $\nabla|_{[x]^\times}:=d+A(y)dy$, where $y$ denotes the coordinate on $[x]\cong \K$ and $A(y)$
is a $\K[y,y^{-1}]$-valued matrix. Since $\nabla|_{[x]}$ is $\K^\times$-invariant, we see
that $A(y)=y^{-1}A$, where $A$ is a constant matrix.
Therefore $\nabla|_{[x]}$ has regular singularities.
\end{proof}

We will mostly use a corollary of this proposition.


\begin{Cor}\label{Cor:0.3}
Let $U^0$ be as in Proposition \ref{Prop:0.1} and $\M$ be a weakly $\K^\times$-equivariant $\DCal_{U^0}$-module.
Suppose that there is a $\K^\times$-stable decreasing $\DCal_{U^0}$-module filtration $\F_i\M$ such that
$\M/\F_i\M$ is a coherent $\Str_{U^0}$-module and $\M$ is complete with respect to the filtration.
Then $\M\cong \Str_{U^0}\widehat{\otimes}_\K\Fl(\M)$.
\end{Cor}

%
%

In the sequel we will need some results about sheaves of TDO.
By a sheaf of twisted differential operators (TDO)  on a smooth
algebraic variety $X$ we mean a sheaf of algebras $\DCal'$ equipped with a monomorphism
$\Str_X\hookrightarrow \DCal'$ that is locally isomorphic to the monomorphism $\Str_X\hookrightarrow
\DCal_X$. ``Locally isomorphic'' means that
 there is a covering $X^i$ of $X$ by open subsets such that there is an isomorphism
$\iota^i:\DCal'|_{X^i}\rightarrow \DCal_{X^i}$ intertwining
the embeddings $\Str_{X^i}\hookrightarrow \DCal_{X^i},
\Str_{X^i}\hookrightarrow \DCal'|_{X^i}$.

It is known (and is easy to show) that the set of isomorphism classes of the sheaves $\DCal'$ of
TDO  (an isomorphism is supposed to intertwine the embeddings
of $\Str_{X}$) is naturally identified with $H^1(X,\Omega^1_{cl})$, where $\Omega^1_{cl}$
denotes the sheaf of closed 1-forms on $X$.

Namely, pick an open cover $X=\bigcup X^i$ and consider a 1-cocycle $\tilde{\alpha}=(\tilde{\alpha}_{ij})$ in $\Omega^1_{cl}$. Define the sheaf
$\DCal^{\tilde{\alpha}}_X$ as follows: $\DCal^{\tilde{\alpha}}_X|_{X^i}=\DCal_{X_i}$ and the transition function from $X^j$ to $X^i$
maps a vector filed $\xi$ to $\xi+\langle\xi, \tilde{\alpha}_{ij}\rangle$. Up to  an isomorphism $\DCal^{\tilde{\alpha}}_X$ depends only on the class $\alpha$ of $\tilde{\alpha}$ in $H^1(X,\Omega^{1}_{cl})$
and we write $\DCal^\alpha_X$ instead of $\DCal^{\tilde{\alpha}}_X$.

A result about TDO we need  is the following proposition.

\begin{Prop}\label{Prop:0.11.1}
Let $U$ be a  vector space and $U^0$ be its open subset such that
$\codim_U U\setminus U^0>1$. Suppose that a sheaf $\DCal^\alpha_{U^0}$ of TDO
has a $\Str_{U^0}$-coherent module $\M$. Then $\alpha=0$.
\end{Prop}
\begin{proof}
The proof is in two steps.

{\it Step 1.}
 Let $\M_1,\M_2$ be   $\DCal^{\alpha_1}_{U^0},\DCal^{\alpha_2}_{U^0}$-modules.
 Then $\M_1\otimes \M_2$ is a  $\DCal^{\alpha_1+\alpha_2}_{U^0}$-module.
 Indeed, pick an affine covering $X=\bigcup X_i$ and fix identifications $\iota^l_i:\DCal^{\alpha_l}_{X_i}\rightarrow \DCal_{X_i}$. Define the action of a vector field $\xi$ on $\M_1\otimes \M_2|_{X_i}$ by
 $\iota_i(\xi)=\iota_i^1(\xi)\otimes 1+ 1\otimes \iota_i^2(\xi)$. Pick a cocycle
 $\tilde{\alpha}^l_{ij}$ representing $\alpha^l$. If $\iota^l_i(\xi)-\iota^l_j(\xi)=\langle\widetilde{\alpha}^l_{ij},\xi\rangle$
for $l=1,2$, then $\iota_i(\xi)-\iota_j(\xi)=\langle\widetilde{\alpha}^1_{ij}+\widetilde{\alpha}^2_{ij},\xi\rangle$.

{\it Step 2.} Now let $\M$ be as in the statement of the proposition
and let $m$ denote the rank of $\M$. Then $\M^{\otimes m}$ is a $\DCal^{m\alpha}_{U^0}$-module.
We remark that $\bigwedge^m \M$ (the sheaf of anti-symmetric sections of $\M^{\otimes m}$)
is a $\DCal^{m\alpha}_{U^0}$-submodule in $\M^{\otimes m}$. But $\bigwedge^m \M$
is a line bundle on $U^0$. However $\operatorname{Pic}(U^0)=\{1\}$ because
$\codim_U U\setminus U^0\geqslant 2$. So $\Str_{U^0}$ is a $\DCal^{m\alpha}_{U^0}$-module.
From this it is easy to deduce that $m\alpha=0$. Since $H^1(X,\Omega^1_{cl})$ is a complex vector space, we see that  $\alpha=0$.
\end{proof}

Till the end of the subsection suppose $U$ is a symplectic vector space
and let $U^0$ be an arbitrary open subset (without restrictions on $\codim_U U\setminus U^0$).

Let $\M$ be a  Poisson $\W_\tb|_{U^0}$-bimodule (as usual, $\tb=\hbar^2$).
Suppose that $\M$  admits a decreasing filtration $\F_i\M$ by $\W_\tb|_{U^0}$-submodules such that
\begin{itemize} \item $\M$ is complete with respect to
this filtration, \item $\tb \F_i\M\subset \F_{i+1}\M$,
\item and $\F_i\M/\F_{i+1}\M$ is a coherent $\Str_{U^0}$-module for all $i$.
\end{itemize}
For instance, if $\M$ is  finitely generated
as a left $\W_\tb|_{U^0}$-module, then we can take the $\tb$-adic filtration. The claim that
this filtration is complete is proved analogously to  \cite{HC}, Lemma 2.4.4.

We want to equip
$\M$ with a structure of a $\DCal_{U^0}$-bimodule. Denote the product map $\W_\tb|_{U^0}\otimes \M\otimes \W_{\tb}|_{U^0}$ by $a\otimes m\otimes b\mapsto a*m*b$. For $u,v\in U$ and a section $m$ of $\M$ we set
$u\cdot m=\frac{1}{2}(u*m+m*u), \partial_v m=\{\check{v},m\}$, where $\check{v}$ is defined by
$\langle\check{v},u\rangle=\omega(v,u), v,u\in U$. Then the operators $u\cdot, v\cdot, \partial_u,\partial_v$ satisfy the commutation relations $[u\cdot,v\cdot]=0=[\partial_u,\partial_v], [\partial_u, v\cdot]=\langle u,v\rangle \operatorname{id}$,
so $\M$ becomes a sheaf of $\DCal_U(U)$-modules. Now let $U'$ be a principal open subset
of $U$ contained in $U^0$ defined by a polynomial $f$. We have $f\cdot m \equiv  f*m \operatorname{mod}\tb$
and so $f$ is invertible on $\M(U^0)$. It follows that $\M$ is a $\DCal_{U^0}$-module.
We remark that $\F_i\M$ is a $\DCal_{U^0}$-submodule of $\M$ and the quotient $\M/\F_i\M$ is a coherent
$\Str_{U^0}$-module.

This discussion has an interesting corollary.

\begin{Prop}\label{Prop:0.4}
Suppose that $U^0$ is stable with respect to the usual action of $\K^\times$ on $U$ and $\codim_U U\setminus U^0\geqslant 2$. Let $\M$ be a $\K^\times$-equivariant Poisson $\W_\tb|_{U^0}$-bimodule equipped with a
filtration $\F_i\M$ as above. Then the natural map $\W_\tb|_{U^0}\widehat{\otimes}_{\K[[\tb]]} \M(U^0)^{\ad U}\rightarrow \M$ is an isomorphism; here $\M(U^0)^{\ad U}$ stands for the kernel of $\{U,\cdot\}$ in
$\M(U^0)$.
\end{Prop}
\begin{proof}
View $\M$ as a $\DCal_{U^0}$-module and apply Corollary \ref{Cor:0.3}.
\end{proof}

\subsection{Functor $\bullet_\dagger:\HC(\tilde{\Halg})\rightarrow\HC^\Xi(\underline{\tilde{\Halg}}^+)^{\Lambda_\Halg}$}\label{SUBSECTION_fun_low_dag}
 This functor is obtained as a composition of several functors
between some intermediate categories. Let us list these functors. The definitions of the categories involved
will be given later.
\begin{align*}
& \Fun_1: \HC(\tilde{\Halg})\rightarrow \HC(\tilde{\Halg}^{\wedge_\Leaf}|_{\Leaf})^{\Lambda_\Halg},\\
& \Fun_2: \HC(\tilde{\Halg}^{\wedge_\Leaf}|_{\Leaf})^{\Lambda_\Halg}\rightarrow \HC\left((\CC(\underline{\tilde{\Halg}}^{\wedge_{\underline{\Leaf}}}|_{\underline{\Leaf}})^{tw})^\Xi\right)^{\Lambda_\Halg},\\
&\Fun_3:\HC\left((\CC(\underline{\tilde{\Halg}}^{\wedge_{\underline{\Leaf}}}|_{\underline{\Leaf}})^{tw})^\Xi\right)^{\Lambda_\Halg}\rightarrow
\HC^\Xi\left(\CC(\underline{\tilde{\Halg}}^{\wedge_{\underline{\Leaf}}}|_{\underline{\Leaf}})^{tw}\right)^{\Lambda_\Halg},\\
&\Fun_4:\HC^\Xi\left(\CC(\underline{\tilde{\Halg}}^{\wedge_{\underline{\Leaf}}}|_{\underline{\Leaf}})^{tw}\right)^{\Lambda_\Halg}\rightarrow
\HC^\Xi\left(\CC(\underline{\tilde{\Halg}}^{\wedge_{\underline{\Leaf}}}|_{\underline{\Leaf}})\right)^{\Lambda_\Halg},\\
&\Fun_5: \HC^\Xi\left(\CC(\underline{\tilde{\Halg}}^{\wedge_{\underline{\Leaf}}}|_{\underline{\Leaf}})\right)^{\Lambda_\Halg}
\rightarrow \HC^\Xi\left(\CC(\underline{\tilde{\Halg}}^{+\wedge_0})\right)^{\Lambda_\Halg},\\
&\Fun_6:\HC^\Xi\left(\CC(\underline{\tilde{\Halg}}^{+\wedge_0})\right)^{\Lambda_\Halg}\rightarrow\HC^\Xi(\underline{\tilde{\Halg}}^{+\wedge_0})^{\Lambda_\Halg},\\
&\Fun_7:\HC^\Xi(\underline{\tilde{\Halg}}^{+\wedge_0})^{\Lambda_\Halg}\rightarrow
\HC^\Xi(\underline{\tilde{\Halg}}^{+})^{\Lambda_\Halg}.
\end{align*}

All these functors but $\Fun_1$ turn out to be category equivalences.

{\bf The functor $\Fun_1$: the (localized) completion.}

Pick $\Malg\in \HC(\tilde{\Halg})$. Define the (micro)localization $\Malg|_{\X^0}$ similarly to
$\Halg|_{\X^0}$ in Subsection \ref{SUBSECTION_sheafifiedI}. According to the discussion preceding
Corollary \ref{Cor:1.24.1}, $\Zalg|_{\X^0}=\z^{\tb}(\Halg|_{\X^0})$.
Therefore the localization $\Malg|_{\X^0}$ of $\Malg$
has a natural structure of a Poisson $\tilde{\Halg}|_{\X^0}$-bimodule. We can form the completion
$\Malg^{\wedge_\Leaf}:=\varprojlim_{k} \Malg|_{\X^0}/\Malg|_{\X^0}(\tilde{\p}|_{\X^0})^k$.
Then, by definition, $\Fun_1(\Malg):=\Malg^{\wedge_\Leaf}|_{\Leaf}$ is the sheaf-theoretic restriction of
$\Malg^{\wedge_\Leaf}$ to $\Leaf$.

The following lemma describes the properties of $\Fun_1(\Malg)$ we need.

\begin{Lem}\label{Lem:3.10.1}
Let $\Malg\in \HC(\tilde{\Halg})$.
\begin{enumerate}
\item $\Fun_1(\Malg)$ has a natural structure of a $\K^\times$-equivariant Poisson $\tilde{\Halg}^{\wedge_\Leaf}|_{\Leaf}$-bimodule.
This module is finitely generated both as a left and as a right $\tilde{\Halg}^{\wedge_\Leaf}|_{\Leaf}$-module.
\item The natural maps $\tilde{\Halg}^{\wedge_\Leaf}|_{\Leaf}\otimes_{\tilde{\Halg}} \Malg, \Malg\otimes_{\tilde{\Halg}} \tilde{\Halg}^{\wedge_\Leaf}|_{\Leaf}\rightarrow \Malg^{\wedge_{\Leaf}}|_{\Leaf}$ are isomorphisms.
\item The completion functor $\Malg\mapsto \Fun_1(\Malg)$ is exact and intertwines
the tensor products (of Poisson bimodules).
\item As a $\Zalg_{(r+1)}^{\wedge_\Leaf}|_{\Leaf}$-module $\Fun_1(\Malg)/\tilde{\param}\Fun_1(\Malg)$
coincides with the usual (commutative) completion of the $\Zalg_{(r+1)}$-module $\Malg/\tilde{\param}\Malg$.
In particular, $\Fun_1(\Malg)=0$ if and only if $\VA(\Malg)\cap \Leaf=\varnothing$.
\end{enumerate}
\end{Lem}
\begin{proof}
We have $[\tilde{\p},\Malg]|_{\X^0}\subset \hbar \Malg|_{\X^0}$. From this it follows that the filtrations
$\Malg\tilde{\p}^n|_{\X^0}$ and $\tilde{\p}^n\Malg|_{\X^0}$ coincide. So $\Malg^{\wedge_\Leaf}$ has a natural structure of a $\tilde{\Halg}^{\wedge_\Leaf}$-bimodule (a unique structure extended from $\Halg|_{\X^0}\otimes\Malg|_{\X^0}\otimes \Halg|_{\X^0}\rightarrow \Malg|_{\X^0}$ by continuity). Also, thanks to Corollary \ref{Cor:1.24.1},
we can extend the bracket from $\Zalg|_{\X^0}\otimes \Malg|_{\X^0}\rightarrow \Malg|_{\X^0}$
to $\z^{\tb}(\Halg^{\wedge_\Leaf})\otimes \Malg^{\wedge_\Leaf}\rightarrow \Malg^{\wedge_\Leaf}$ by continuity.
This bracket then transfers to the pull-backs to $\Leaf$.



(2) and the exactness part in (3) are standard corollaries of Proposition \ref{Prop:1.33I}, see Lemma \ref{Lem:1.23.1}.
The compatibility with tensor products in (3) follows from (2).

(4) follows from the construction and Proposition \ref{Cor:1.31I}.
\end{proof}


Now let us define the category $\HC(\tilde{\Halg}^{\wedge_\Leaf}|_{\Leaf})$. By definition
it consists of $\K^\times$-equivariant Poisson $\tilde{\Halg}^{\wedge_\Leaf}|_{\Leaf}$-bimodules
that are finitely generated as left $\tilde{\Halg}^{\wedge_\Leaf}|_{\Leaf}$-modules. Since $\tilde{\Halg}^{\wedge_\Leaf}|_{\Leaf}$ is a sheaf of Noetherian algebras complete in the
$\p^{\wedge_\Leaf}|_{\Leaf}$-adic topology, we can use the claim similar to Lemma 2.4.2 from \cite{HC}
to show that $\HC(\tilde{\Halg}^{\wedge_\Leaf}|_{\Leaf})$ is an abelian category and
that any object in  $\HC(\tilde{\Halg}^{\wedge_\Leaf}|_{\Leaf})$ is complete and separated in
the $\p^{\wedge_\Leaf}|_{\Leaf}$-adic topology. The category
$\HC(\tilde{\Halg}^{\wedge_\Leaf}|_{\Leaf})$ is equipped with a tensor product functor:
the tensor product of Poisson bimodules.
So  the   functor $\Fun_1:\Malg\mapsto \Malg^{\wedge_\Leaf}|_{\Leaf}:\HC(\tilde{\Halg})\rightarrow\HC(\tilde{\Halg}^{\wedge_\Leaf}|_{\underline{\Leaf}})$
is monoidal.

We want to introduce an analog of the decomposition (\ref{eq:3.13.1}) for $\HC(\tilde{\Halg}^{\wedge_\Leaf}|_{\Leaf})$. Let $\M\in\HC(\tilde{\Halg}^{\wedge_\Leaf}|_{\Leaf})$ and set $\M_n:=\M/\M(\tilde{\p}^{\wedge_\Leaf}|_{\Leaf})^n$.
We have the endomorphisms $C_i=\{\cb_i,\cdot\}$ of a Poisson $\tilde{\Halg}^{\wedge_\Leaf}|_{\Leaf}$-bimodule. The quotients $\M_n/\M_{n+1}$ are coherent $\Str_{\Leaf}$-modules and Poisson $\tilde{\Halg}^{\wedge_\Leaf}|_{\Leaf}$-bimodules. The annihilator of such a module in $\Halg^{\wedge_\Leaf}|_{\Leaf}$ and hence in $\Str_{\Leaf}$ is a Poisson ideal. But $\Str_{\Leaf}$ has no nontrivial
Poisson ideals for $\Leaf$ is a symplectic variety. It follows that any quotient $\M_n/\M_{n+1}$ has finite length.
Therefore we have the decomposition $\M_n/\M_{n+1}:=\bigoplus_\lambda (\M_n/\M_{n+1})^\lambda$ into the sum
of the generalized eigensheaves for the operators $C_i$. We remark that the natural homomorphism
$(\tilde{\p}^{\wedge_\Leaf}|_{\Leaf})^n/(\tilde{\p}^{\wedge_\Leaf}|_{\Leaf})^{n+1}\otimes \M/\M_1\rightarrow \M_n/\M_{n+1}$ is surjective. Hence $(\tilde{\p}^{\wedge_\Leaf}|_{\Leaf})^n/(\tilde{\p}^{\wedge_\Leaf}|_{\Leaf})^{n+1}\otimes (\M/\M_1)^\lambda\rightarrow (\M_n/\M_{n+1})^\lambda$ is surjective. Set $\M^\lambda:=\varprojlim_{n} (\M_n)^\lambda$.
Then we have the direct sum  decomposition $\M=\bigoplus_\lambda \M^\lambda$ and the corresponding category
decomposition $\HC(\tilde{\Halg}^{\wedge_\Leaf}|_{\Leaf})= \bigoplus_\lambda
\HC(\tilde{\Halg}^{\wedge_\Leaf}|_{\Leaf})^\lambda$. The subcategory $\HC(\tilde{\Halg}^{\wedge_\Leaf}|_{\Leaf})^{\Lambda_\Halg}\subset \HC(\tilde{\Halg}^{\wedge_\Leaf}|_{\Leaf})$ is now defined in an obvious way. Clearly,
the image of $\Fun_1$ lies in $\HC(\tilde{\Halg}^{\wedge_\Leaf}|_{\Leaf})^{\Lambda_\Halg}$.

\begin{Rem}
In fact, one can show that $\HC(\tilde{\Halg}^{\wedge_\Leaf}|_{\underline{\Leaf}})^{\Lambda_\Halg}$
coincides with the whole category $\HC(\tilde{\Halg}^{\wedge_\Leaf}|_{\underline{\Leaf}})$.
This follows from the construction above and Proposition \ref{Prop:3.15.2} below.
\end{Rem}

{\bf The functor $\Fun_2$: the push-forward with respect to $\tilde{\theta}$.}

Recall the isomorphism $\tilde{\theta}:\tilde{\Halg}^{\wedge_\Leaf}|_{\Leaf}\rightarrow (\CC(\underline{\tilde{\Halg}}^{\wedge_{\underline{\Leaf}}}|_{\underline{\Leaf}})^{tw})^\Xi$.
Define the category $\HC((\CC(\underline{\tilde{\Halg}}^{\wedge_{\underline{\Leaf}}}|_{\underline{\Leaf}})^{tw})^\Xi)$
analogously to $\HC(\tilde{\Halg}^{\wedge_\Leaf}|_{\Leaf})$.
We set $\Fun_2:=\tilde{\theta}_*$. The functor $\Fun_2$ restricts to an equivalence
$\HC(\tilde{\Halg}^{\wedge_\Leaf}|_{\underline{\Leaf}})^{\Lambda_\Halg}\rightarrow \HC((\CC(\underline{\tilde{\Halg}}^{\wedge_{\underline{\Leaf}}}|_{\underline{\Leaf}})^{tw})^\Xi)^{\Lambda_\Halg}$.

{\bf The functor $\Fun_3$: lifting to $\underline{\Leaf}$.}

Let $\M\in \HC((\CC(\underline{\tilde{\Halg}}^{\wedge_{\underline{\Leaf}}}|_{\underline{\Leaf}})^{tw})^\Xi)$.
Set
\begin{equation}\label{eq:3.13.3}
\M|_{\underline{\Leaf}}:=
\CC(\underline{\tilde{\Halg}}^{\wedge_{\underline{\Leaf}}}|_{\underline{\Leaf}})^{tw}\otimes_
{(\CC(\underline{\tilde{\Halg}}^{\wedge_{\underline{\Leaf}}}|_{\underline{\Leaf}})^{tw})^\Xi}\M.
\end{equation}
By definition, $\M|_{\underline{\Leaf}}$ is a $\CC(\underline{\tilde{\Halg}}^{\wedge_{\underline{\Leaf}}}|_{\underline{\Leaf}})^{tw}$-
$(\CC(\underline{\tilde{\Halg}}^{\wedge_{\underline{\Leaf}}}|_{\underline{\Leaf}})^{tw})^\Xi$-bimodule.
It is equipped with a bracket with $\z^{\tb}((\CC(\underline{\Halg}^{\wedge_{\underline{\Leaf}}}|_{\underline{\Leaf}})^{tw})^\Xi)$
defined using the Leibnitz rule.

\begin{Lem}
\begin{enumerate}
\item There is a unique bracket map $\z^{\tb}(\CC(\underline{\Halg}^{\wedge_{\underline{\Leaf}}}|_{\underline{\Leaf}})^{tw})\otimes
    \M|_{\underline{\Leaf}}\rightarrow \M|_{\underline{\Leaf}}$
    extending the bracket with $\z^{\tb}((\CC(\underline{\Halg}^{\wedge_{\underline{\Leaf}}}|_{\underline{\Leaf}})^{tw})^\Xi)$
    and satisfying (\ref{eq:4.11.2})-(\ref{eq:4.11.4}).
\item There is a unique right multiplication map
$ \CC(\underline{\tilde{\Halg}}^{\wedge_{\underline{\Leaf}}}|_{\underline{\Leaf}})^{tw}
\otimes \M|_{\underline{\Leaf}}\rightarrow \M|_{\underline{\Leaf}}$ satisfying (\ref{eq:4.11.1})
so that $\M|_{\underline{\Leaf}}$ becomes a Poisson $\CC(\underline{\tilde{\Halg}}^{\wedge_{\underline{\Leaf}}}|_{\underline{\Leaf}})^{tw}$-bimodule.
\item The natural map $\M\rightarrow (\M|_{\underline{\Leaf}})^\Xi$ is an isomorphism.
\end{enumerate}
\end{Lem}
\begin{proof}
First of all, we remark that the right product map is recovered uniquely from the left product and the bracket
map. So (1) implies (2).

(1) and (3) are local so it is enough to prove the similar claims for the restrictions
of the sheaves of interest to a $\Xi$-stable affine open subset $U_i\subset \underline{\Leaf}$.
But there $$\CC(\underline{\tilde{\Halg}}^{\wedge_{\underline{\Leaf}}}|_{\underline{\Leaf}})^{tw}|_{U_i}\cong
\CC(\underline{\tilde{\Halg}}^{\wedge_{\underline{\Leaf}}}|_{U_i}).$$

We also remark that the natural homomorphism
$$\W_{\tb}|_{U_i}\otimes_{(\W_\tb|_{U_i})^\Xi}
\CC(\underline{\tilde{\Halg}}^{\wedge_{\underline{\Leaf}}}|_{U_i})^\Xi\rightarrow \CC(\underline{\tilde{\Halg}}^{\wedge_{\underline{\Leaf}}}|_{U_i})$$
is an isomorphism. This again follows from the fact that
$\pi':\underline{\X}^{\wedge}_{\underline{\Leaf}}\rightarrow \X^\wedge_\Leaf$ is the
quotient map for the free action of $\Xi$. (3) follows. Also to prove (1) it is enough to show that the bracket can be uniquely extended from $(\W_\tb|_{U_i})^\Xi\otimes \M|_{U_i}\rightarrow \M|_{U_i}$ to $\W_\tb|_{U_i}\otimes \M|_{U_i}\rightarrow \M|_{U_i}$.

Further, consider the filtration $\M|_{U_i}\left(\CC(\underline{\p}^{\wedge_{\underline{\Leaf}}}|_{U_i})^\Xi\right)^n$
of $\M|_{U_i}$. The quotient $$\M_n:=\M|_{U_i}/\M|_{U_i}\left(\CC(\underline{\p}^{\wedge_{\underline{\Leaf}}}|_{U_i})^\Xi\right)^n$$ is an
$(\W_\tb|_{U_i})^\Xi/(\tb)^n$-bimodule that is finitely generated as a left (or right) module.
If we prove assertion (1)
for  each $\M_m$, then (1) for $\M$ will follow. So it is enough to consider
the case when $\M$ is finitely generated as a left module over $(\W_\tb|_{U_i})^\Xi$.

Pick a point $u\in \underline{\Leaf}$. Choose a $\Xi$-stable affine open neighborhood $U$ of $u$ in $U_i$.
To simplify the notation we will write $\M$ instead of $\M(U)$
and $\A$ instead of $\W_\tb(U)$.
Shrinking $U$ if necessary (so that still $u\in U$) we may assume that  $\K[U]$ is a free  $\K[U]^\Xi$-module.
Pick a free basis $f_1,\ldots,f_k\in \K[U]$ over $\K[U]^\Xi$. We can shrink $U$ again to find elements $a^i_j\in \K[U]^\Xi, i=1,\ldots,k, j=0,1,\ldots,k-1$ such that
$P_i(f_i):=f_i^k+a^i_{k-1}f^{k-1}+\ldots+ a_0^i$ vanishes in $\K[U]$ but $P_i'(f_i)$ is invertible in $\K[U]$.
 We remark that $f_1,\ldots,f_k$ is still a  basis of the (left or right)
$\A^\Xi$-module $\A$. Set $\widetilde{\M}:=\A\otimes_{\A^\Xi}\M$ so that $\M=\widetilde{\M}^\Xi$.

Now recall an identification $\A\cong \K[U][[\tb]]$, see Subsection \ref{SUBSECTION_sheafifiedI}.
Let $\hat{P}_i(f_i),\hat{P}'_i(f_i)$ be the polynomials,
where the commutative products are replaced with $*$ (the order of the factors is the same). So in $\A$ we have
$\hat{P}_i(f_i)=\tb g_i$ for some $g_i\in \A$ and $\hat{P}'_i(f_i)$ is still invertible.
For $m\in \widetilde{\M}$ set $Q(f_i,m)=f_i^{k-1}\{a_{k-1}^i,m\}+\ldots+ f_i\{a_1^i,m\}+ \{a^i_0,m\}$.

%


Suppose for a moment that there is
lifting of $\{\cdot,\cdot\}:\A^\Xi\otimes \widetilde{\M}\rightarrow \widetilde{\M}$ to
$\A\otimes \widetilde{\M}\rightarrow \widetilde{\M}$ such that
\begin{align}\label{eq:4.11.2'}&\{a,bm\}=\{a,b\}m+b\{a,m\}, \forall a,b\in \A,\\\label{eq:4.11.3'}
    &\{ab,m\}=a\{b,m\}+b\{a,m\}-\tb\{b,\{a,m\}\},\\\nonumber \forall a,b\in \A, m\in \widetilde{\M}.
\end{align}

Then we have $\tb\{g_i,m\}=\hat{P}_i'(f_i)\{f_i,m\}+Q(f_i,m)+\tb F(f_i,(a_i^j)_{j=0}^{k-1},m)$
where  $F(f_i,(a_i^j)_{j=0}^{k-1},m)$ is an expression involving $f_i,a_i^j,m$, various brackets and products.
Then we clearly have
\begin{equation}\label{eq:4.11.6}
\{f_i,m\}=\hat{P}_i'(f_i)^{-1}[\tb(\{g_i,m\}- F(f_i,(a_i^j)_{j=0}^{k-1},m))-Q(f_i,m)].
\end{equation}
Conversely, given (\ref{eq:4.11.6}) we can extend $\{\cdot,\cdot\}:\A^\Xi\otimes \widetilde{\M}\rightarrow
\widetilde{\M}$ to a unique map $\A\otimes\widetilde{\M}\rightarrow \widetilde{\M}$
using the equalities
\begin{equation}\label{eq:4.11.7}\{f_ia, m\}=f_i\{a,m\}+a\{f_i,m\}-\tb\{a,\{f_i,m\}\}\end{equation} and the $\A^\Xi$-linearity (recall that $f_1,\ldots,f_k$
is a basis of the right $\A^\Xi$-module $\A$).

 We can use (\ref{eq:4.11.6}),(\ref{eq:4.11.7}) to define
a  map $\{\cdot,\cdot\}:\A\otimes \M\rightarrow \M$ recursively. This is possible
because $\M$ is complete and separated in the $\tb$-adic topology.  For $f\in \A$ we get
$\{f,m\}$ in the form
\begin{equation}\label{eq:4.11.8}
\{f,m\}=\sum_i F_i\{b_i,m\},
\end{equation}
where $b_i\in \A^\Xi$ and the sequence $F_i$ converges to $0$ in the $\tb$-adic topology.

It remains to prove that the  map $\{\cdot,\cdot\}:\A\otimes\widetilde{\M}\rightarrow \widetilde{\M}$
defined in this way satisfies (\ref{eq:4.11.2})-(\ref{eq:4.11.4}).

Pick a point $b\in U$. It follows from the definition of the map $m\mapsto \{f,m\}, f\in \A,$ that
this map extends to  $\A^{\wedge_b}\otimes \widetilde{\M}^{\wedge_b}\rightarrow \widetilde{\M}^{\wedge_b}$, where
 $\A^{\wedge_b}, \widetilde{\M}^{\wedge_b}$ denote the completions of $\A,\widetilde{\M}$ at $b$. These completions are defined as the completions above. Namely, let $\n_b \subset\A=\K[U][[\tb]]$ denote the maximal
ideal  of $b$. Then we can form the completions $\A^{\wedge_b}:=\varprojlim_{k} \A/\n_b^k, \widetilde{\M}^{\wedge_b}:=\varprojlim_{k} \widetilde{\M}/\n_b^k\widetilde{\M}$. Using the machinery of Subsection \ref{SUBSECTION_flatness}, we see that the functor $\widetilde{\M}\mapsto \widetilde{\M}^{\wedge_b}$ is exact and $\widetilde{\M}^\wedge_b=\A^{\wedge_b}\otimes_\A \widetilde{\M}$.

Similarly, we can consider the completion $(\A^\Xi)^{\wedge_{\pi'(b)}}$ of $\A^\Xi$ at $\pi'(b)$.
We remark that the inclusion $\A^\Xi\hookrightarrow \A$ induces an isomorphism $(\A^\Xi)^{\wedge_{\pi'(b)}}\cong
\A^{\wedge_b}$. Now $\widetilde{\M}^{\wedge_b}=\A^{\wedge_b}\otimes_\A\widetilde{\M}=\A^{\wedge_b}\otimes_\A \A\otimes_{\A^\Xi}\M=(\A^\Xi)^{\wedge_{\pi'(b)}}\otimes_{\A^\Xi}\M=\M^{\wedge_{\pi'(b)}}$.

As in  Lemma \ref{Lem:3.10.1}, there is a unique continuous bracket
map $(\A^\Xi)^{\wedge_{\pi'(b)}}\otimes\M^{\wedge_{\pi'(b)}}\rightarrow \M^{\wedge_{\pi'(b)}}$ extending
the map $\A^\Xi\otimes \M\rightarrow \M$. We have a natural embedding $\A\hookrightarrow \A^{\wedge_b}\cong
(\A^\Xi)^{\wedge_{\pi'(b)}}$, which gives rise to a bracket map $\A\otimes \widetilde{\M}^{\wedge_b}
\rightarrow \widetilde{\M}^{\wedge_b}$. By construction, this map
satisfies (\ref{eq:4.11.2})-(\ref{eq:4.11.4}). Hence this map satisfies
(\ref{eq:4.11.6}),(\ref{eq:4.11.7}). So it coincides with the map
induced (=extended by continuity) from $\A\otimes \widetilde{\M}\rightarrow\widetilde{\M}$.

So it remains to show that $\widetilde{\M}\rightarrow \widetilde{\M}^{\wedge_b}$ is injective.
Let $\Ncal$ be the kernel. Then $\Ncal^{\wedge_b}=\{0\}$. We want to show that
$\Ncal=\{0\}$. Since the completion functor is exact and the $\tb$-adic filtration on $\Ncal$
 is separated,  we may replace $\Ncal$ with
$\Ncal/\tb \Ncal$ and assume that $\tb$ acts trivially on $\Ncal$. By a discussion preceding
Proposition \ref{Prop:0.4}, $\Ncal$ can be equipped with a structure of a $\DCal_{\underline{\Leaf}}$-module.
Then $\Ncal$ is also a coherent $\Str_{\underline{\Leaf}}$-module and, in particular,  a vector bundle. So $\Ncal^{\wedge_b}=\{0\}$ implies $\Ncal=\{0\}$.
\end{proof}

Introduce the category $\HC^\Xi\left((\CC(\underline{\tilde{\Halg}}^{\wedge_{\underline{\Leaf}}}|_{\underline{\Leaf}})^{tw}\right)$ similarly to $\HC\left((\CC(\underline{\tilde{\Halg}}^{\wedge_{\underline{\Leaf}}}|_{\underline{\Leaf}})^{tw})^\Xi\right)$
with the only difference that the Poisson bimodules in the  new category are supposed to be
additionally $\Xi$-equivariant. So $\Fun_3: \M\mapsto \M|_{\underline{\Leaf}}$ is an equivalence
$ \HC\left((\CC(\underline{\tilde{\Halg}}^{\wedge_{\underline{\Leaf}}}|_{\underline{\Leaf}})^{tw})^\Xi\right)
\rightarrow \HC^\Xi\left((\CC(\underline{\tilde{\Halg}}^{\wedge_{\underline{\Leaf}}}|_{\underline{\Leaf}})^{tw}\right)$, the quasi-inverse
functor is given by taking $\Xi$-invariants. The claims about the equivalences comes from the fact that $\Xi$
acts freely on $\underline{\Leaf}$.

These equivalences restrict to equivalences between the categories $\HC\left((\CC(\underline{\tilde{\Halg}}^{\wedge_{\underline{\Leaf}}}|_{\underline{\Leaf}})^{tw})^\Xi\right)^{\Lambda_\Halg}$ and $\HC^\Xi\left((\CC(\underline{\tilde{\Halg}}^{\wedge_{\underline{\Leaf}}}|_{\underline{\Leaf}})^{tw}\right)^{\Lambda_\Halg}$.

{\bf The functor $\Fun_4$: untwisting.}

Recall the elements $X_{ij}\in \param\z^{\tb}(\CC(\underline{\Halg}^{\wedge_{\underline{\Leaf}}}|_{\underline{\Leaf}})(U_{ij}))$
from Theorem \ref{Thm:2.0I}.
We have isomorphisms $\CC(\underline{\tilde{\Halg}}^{\wedge_{\underline{\Leaf}}}|^{tw}_{\underline{\Leaf}})|_{U_i}\cong
\CC(\underline{\tilde{\Halg}}^{\wedge_{\underline{\Leaf}}}|_{U_i})$ and the transition functions from $U_j$ to $U_i$ for the sheaf $\CC(\underline{\tilde{\Halg}}^{\wedge_{\underline{\Leaf}}}|^{tw}_{\underline{\Leaf}})$ are
$\exp(\frac{1}{\tb}\ad X_{ij})=\exp(\{X_{ij},\cdot\})$. Of course,  the endomorphism $\{ X_{ij},\cdot\}$
of $\M|_{U_{ij}}$
makes sense but its exponential does not because, in general, it diverges  (with the exception
of the case when all $C_i$ are nilpotent).

However, we have the following result:

\begin{Lem}\label{Lem:3.10.4}
There are elements $b_i^l\in \W_{\tb}(U_i)^{\K^\times\times \Xi}, l=1,\ldots,r,$ such that the operator
$m\mapsto \{X_{ij},m\}-\sum_{l=1}^r (b_i^l-b_j^l)C_i(m)$ on $(\M/\CC(\underline{\p}^{\wedge_{\underline{\Leaf}}})\M)|_{U_{ij}}$
is nilpotent for any  $\M\in \HC^\Xi\left((\CC(\underline{\tilde{\Halg}}^{\wedge_{\underline{\Leaf}}}|^{tw}_{\underline{\Leaf}})\right)$.
\end{Lem}
\begin{proof}
We can represent $X_{ij}$ as $\tb X_{ij}^0+\sum_{l=1}^r \cb_i X_{ij}^l$, where $X_{ij}^0,\ldots, X_{ij}^r\in \z^{\tb}(\CC(\underline{\Halg}^{\wedge_{\underline{\Leaf}}}|_{\underline{\Leaf}})(U_{ij}))$. We remark that
the elements $X_{ij}^l, l=1,\ldots,r$ are not unique but their classes modulo $\tb\Halg+ \param_{(1)}\z^{\tb}(\CC(\underline{\Halg}^{\wedge_{\underline{\Leaf}}}|_{\underline{\Leaf}})(U_{ij}))$ are.
This follows from the fact that
$\z(\CC(\underline{\Halg}_{(1)}^{\wedge_{\underline{\Leaf}}}|_{\underline{\Leaf}})(U_{ij}))$ is flat over $\K[[\param_{(1)}^*]]$.
To verify the last statement one uses the descending induction on $k$ to show that  $\z(\CC(\underline{\Halg}_{(k)}^{\wedge_{\underline{\Leaf}}}|_{\underline{\Leaf}})(U_{ij}))$ is flat over $\K[[\param_{(k)}^*]]$
(compare with the proofs of Lemma \ref{Lem:auxil} and Corollary \ref{Cor:1.24}).

Consider the image $x^l_{ij}$ of the cochain $X^l_{ij}$ in  $\Str_{\underline{\Leaf}}$ and  let $a^l$ denote its class in $H^1(\underline{\Leaf},\Str_{\underline{\Leaf}})$. Thanks to the uniqueness property for $X^l_{ij}$
from the previous paragraph and the fact that the collection $\exp(\frac{1}{\tb}X_{ij})$ is a cocycle,
we deduce that the elements $x_{ij}^l$
form a cocycle for each  $l$. We also remark that for an  object $\M$ in $\HC^\Xi((\CC(\underline{\tilde{\Halg}}^{\wedge_{\underline{\Leaf}}}|_{\underline{\Leaf}})^{tw})$
annihilated by $\CC(\underline{\tilde{\p}}^{\wedge_{\underline{\Leaf}}}|_{\Leaf})^{tw}$ the (left or right) multiplication
by $X_{ij}^l$ is the same as the multiplication by $x_{ij}^l$.

We are going to
check that there are $b_{i}^l$ such that \begin{equation}\label{eq:4.9.11}\sum_{l=1}^r (x_{ij}^l-b_i^l+b_j^l)\lambda_l=0,\end{equation} for any eigenvalue $\lambda=(\lambda_1,\ldots,\lambda_r)$
of $(C_1,\ldots,C_r)$ on an object $\M$ as in the previous paragraph
(such an object is automatically a coherent $\Str_{\Leaf}$-module).

The existence of the cochains $b_i^l$
satisfying (\ref{eq:4.9.11}) is equivalent to $a_\lambda:=\sum_{l=1}^r \lambda_l a^l=0$.
We have a natural map $d:H^1(\underline{\Leaf},\Str_{\underline{\Leaf}})\rightarrow H^1(\underline{\Leaf},\Omega^1_{cl})$ induced by the De Rham differential $d$. This map
is injective because $H^1(\underline{\Leaf},\K)=0$ (all open subsets intersect).

Let $\lambda$ be an eigenvalue of $(C_1,\ldots,C_r)$ in $\M$ and $\M_\lambda$ be the corresponding eigensheaf.
Pick $v\in V_0$ and let $\{v,\cdot\}_i$ denote the Poisson bracket with $v$ on $\M_\lambda|_{U_i}$. So $\{v,\cdot\}_i-\{v,\cdot\}_j=\{\sum_{l=1}^r \cb_l \partial_v x^l_{ij},\cdot\}=\sum_{l=1}^r \lambda_l (\partial_v x^l_{ij})$. Recall the construction of the correspondence between Poisson bimodules and $\DCal$-modules in Subsection  \ref{SUBSECTION_DMod}. From this construction it follows that $\M_\lambda$ is a $\DCal^{da_\lambda}_{\underline{\Leaf}}$-module. But $\M_\lambda$ is a coherent $\Str_{\underline{\Leaf}}$-module. Proposition \ref{Prop:0.11.1} implies that $da_\lambda=0$. Hence $a_\lambda=0$.

Let us check that the elements $b_i^l$ we have found satisfy the condition of the lemma.
Let $\M$ be such as in the lemma  statement.
We may replace $\M$ with $\M/\CC(\underline{\p}^{\wedge_{\underline{\Leaf}}})|^{tw}_{\underline{\Leaf}}\M$
and assume that $\M$ is annihilated by $\CC(\underline{\p}^{\wedge_{\underline{\Leaf}}})|^{tw}_{\underline{\Leaf}}$.
The statement of the lemma will follow if we check that $\M$ has finite length.
The action of $\z^{\tb}(\CC(\underline{\Halg}^{\wedge_{\Leaf}})|^{tw}_{\underline{\Leaf}})$ on $\M$ reduces
to the action of $\Str_{\underline{\Leaf}}$ so that $\M$ becomes a coherent $\Str_{\underline{\Leaf}}$-module.
Consider $\M|_{U_{i}}$ as a Poisson $\W_{\tb}|_{U_{i}}$-bimodule. This module is annihilated  by
$\tb \W_\tb|_{U_{i}}$ and so, according to Proposition \ref{Prop:0.4}, is a coherent $\DCal_{U_{i}}$-module.
It follows that the Poisson $\CC(\underline{\tilde{\Halg}}^{\wedge_{\Leaf}})|^{tw}_{\underline{\Leaf}}$-bimodule $\M$ has finite length. 
\end{proof}

Fix  $b_i^l$ as in Lemma \ref{Lem:3.10.4}.
Set $B_i:=\sum_{l=1}^r \cb_l b_i^l$.

Now we are ready to construct an ``untwisting'' functor $$\Fun_4:\HC^\Xi(\CC(\underline{\tilde{\Halg}}^{\wedge_{\underline{\Leaf}}}|_{\underline{\Leaf}})^{tw})^{\Lambda_\Halg}
\rightarrow \HC^\Xi(\CC(\underline{\tilde{\Halg}}^{\wedge_{\underline{\Leaf}}}|_{\underline{\Leaf}}))^{\Lambda_\Halg}.$$

Set \begin{equation}\label{eq:4.11.13}\hat{X}_{ij}=\tb \ln(\exp(\frac{1}{\tb}B_j)\exp(\frac{1}{\tb}X_{ij})\exp(-\frac{1}{\tb}B_i)).\end{equation}
We remark that $\hat{X}_{ij}\in \param\z^{\tb}(\CC(\underline{\Halg}^{\wedge_{\underline{\Leaf}}}|_{\underline{\Leaf}}))(U_{ij})$
and moreover, \begin{equation}\label{eq:4.11.12}\hat{X}_{ij}\equiv -B_i+X_{ij}+B_j \operatorname{mod} \param^2 \z^{\tb}(\CC(\underline{\Halg}^{\wedge_{\underline{\Leaf}}}|_{\underline{\Leaf}}))(U_{ij}).\end{equation}
(\ref{eq:4.11.12}) holds because all brackets
appearing in the Lie series (\ref{eq:4.11.13}) lie in $\param^2 \z^{\tb}(\CC(\underline{\Halg}^{\wedge_{\underline{\Leaf}}}|_{\underline{\Leaf}}))(U_{ij})$.

We deduce from (\ref{eq:4.11.12}) and the definition of the elements $B_i$ that $\exp(\{\hat{X}_{ij},\cdot\})$
is a well-defined map on $\M|_{U_{ij}}$. Let $\Fun_4(\M)$ be the sheaf obtained from $\M$ by regluing
the sheaves $\M|_{U_i}$ using the transition functions  $\exp(\{\hat{X}_{ij},\cdot\})$.
Introduce a new structure   of a Poisson $\CC(\underline{\tilde{\Halg}}^{\wedge_{\underline{\Leaf}}}|_{U_i})$-bimodule
on $\M|_{U_i}$ by sending a section $x$ of $\CC(\underline{\tilde{\Halg}}^{\wedge_{\underline{\Leaf}}}|_{U_i})$
to $\exp(\{B_i,\cdot\})x$. Then these Poisson bimodule structures glue together to
a Poisson $\CC(\underline{\tilde{\Halg}}^{\wedge_{\underline{\Leaf}}}|_{\underline{\Leaf}})$-bimodule
structure on $\F_4(\M)$. Since all transition functions are $\K^\times\times\Gamma\times\Xi$-equivariant,
we see that $\F_4(\M)\in \HC^\Xi(\CC(\underline{\tilde{\Halg}}^{\wedge_{\underline{\Leaf}}}|_{\underline{\Leaf}}))^{\Lambda_\Halg}$.

Conversely, let $\Ncal\in \HC^\Xi(\CC(\underline{\tilde{\Halg}}^{\wedge_{\underline{\Leaf}}}|_{\underline{\Leaf}})^{\Lambda_\Halg}$.
Then $\exp(-\{\hat{X}_{ij},\cdot\})$ is  well-defined on $\Ncal|_{U_{ij}}$. So reversing the untwisting
procedure explained before we get a functor $$\HC^\Xi(\CC(\underline{\tilde{\Halg}}^{\wedge_{\underline{\Leaf}}}|_{\underline{\Leaf}}))^{\Lambda_\Halg}\rightarrow
\HC^\Xi(\CC(\underline{\tilde{\Halg}}^{\wedge_{\underline{\Leaf}}}|_{\underline{\Leaf}}^{tw})^{\Lambda_\Halg}.$$
It is clear that this functor is (quasi)inverse to $\Fun_4$.

{\bf The functor $\Fun_5$: taking flat sections.}
Let $\M\in \HC^\Xi(\CC(\underline{\tilde{\Halg}}^{\wedge_{\underline{\Leaf}}}|_{\underline{\Leaf}}))$.
Let $\Fun_5(\M)$ be the subspace of $\M(\underline{\Leaf})$ consisting of all sections
annihilated by $\{V_0,\cdot\}$. We remark that being a finitely generated Poisson $\W_\tb|_{\underline{\Leaf}}$-bimodule,
$\M$ satisfies the assumptions of Proposition \ref{Prop:0.4} (the filtration $\M\CC(\underline{\p}^{\wedge_{\underline{\Leaf}}}|_{\underline{\Leaf}})^n$ has the required properties).
So $\M=\W_\tb|_{\underline{\Leaf}}\widehat{\otimes}_{\K[[\tb]]} \Fun_5(\M)$.

Define the category $\HC^\Xi(\CC(\underline{\tilde{\Halg}}^{+\wedge_0}))$ by analogy
with the category  $\HC^\Xi((\CC(\underline{\tilde{\Halg}}^{\wedge_{\underline{\Leaf}}}|_{\underline{\Leaf}})^{tw})$.
In particular, all Poisson bimodules in $\HC^\Xi(\CC(\underline{\tilde{\Halg}}^{+\wedge_0}))$ are complete in
the $\CC(\underline{\p}^{+\wedge_0})$-adic topology.

So we see that $\Fun_5$ induces an equivalence $\HC^\Xi(\CC(\underline{\tilde{\Halg}}^{\wedge_{\underline{\Leaf}}}|_{\underline{\Leaf}}))\rightarrow
\HC^\Xi(\CC(\underline{\tilde{\Halg}}^{+\wedge_0}))$,
a quasiinverse equivalence is given by $(\W_\hbar|_{\underline{\Leaf}})\widehat{\otimes}_{\K[[\hbar]]}\bullet$.
Clearly, $\Fun_5,\Fun_5^{-1}$ restrict to equivalences of $\HC^\Xi(\CC(\underline{\tilde{\Halg}}^{\wedge_{\underline{\Leaf}}}|_{\underline{\Leaf}}))^{\Lambda_\Halg}$ and $
\HC^\Xi(\CC(\underline{\tilde{\Halg}}^{+\wedge_0}))^{\Lambda_\Halg}$.

{\bf The functor $\Fun_6(\bullet)=e(\underline{\Gamma})\bullet e(\underline{\Gamma})$.}

Recall the element $e(\underline{\Gamma})\in \K\Gamma\subset \CC(\underline{\tilde{\Halg}}^{+\wedge_0})$
defined in Subsection \ref{SUBSECTION_centralizer}
and the natural isomorphism $\underline{\tilde{\Halg}}^{+\wedge_0}\cong e(\underline{\Gamma})\CC(\underline{\tilde{\Halg}}^{+\wedge_0})e(\underline{\Gamma})$.

So we can define $\Fun_6: \HC^\Xi(\CC(\underline{\tilde{\Halg}}^{+\wedge_0}))\rightarrow
\HC^\Xi(\underline{\tilde{\Halg}}^{+\wedge_0})$ by sending $\M$ to $e(\underline{\Gamma})\M e(\underline{\Gamma})$,
see the discussion at the end of Subsection \ref{SUBSECTION_Poisson}. The functor $\Fun_6$ is an equivalence.
Indeed, the right inverse functor constructed
in Subsection \ref{SUBSECTION_centralizer} is nothing else but $\Ncal\mapsto \CC(\Ncal)$. It is inverse to $\Fun_6$
by the remarks there.

{\bf The functor $\Fun_7$: taking finite vectors.}

Let $\M\in \HC^\Xi(\underline{\tilde{\Halg}}^{+\wedge_0})$. We say that $m\in \M$
is locally finite with respect to $\K^\times$ (shortly, $\K^\times$-l.f.) if it lies in
a finite dimensional $\K^\times$-stable subspace of $\M$. The space of all $\K^\times$-l.f.
elements in $\M$ will be denoted by $\M_{\K^\times-l.f.}$. It is clear that
$\M_{\K^\times-l.f.}$ is a $\K^\times\times \widetilde{\Xi}$-equivariant Poisson
$\underline{\tilde{\Halg}}^+$-bimodule. We claim that $\M_{\K^\times-l.f.}$ is dense in
$\M$ and is finitely generated as a $\underline{\tilde{\Halg}}$-bimodule. So we get the functor
$\Fun_7: \HC^\Xi(\underline{\tilde{\Halg}}^{+\wedge_0})\rightarrow \HC^\Xi(\underline{\tilde{\Halg}}^+),
\M\mapsto \M_{\K^\times-l.f.}$. This functor is an equivalence, a quasiinverse functor
sends $\Nalg\in \HC^\Xi(\underline{\tilde{\Halg}}^+)$ to its completion in the
$\underline{\tilde{\p}}^+$-adic topology.

All claims are quite standard and follow easily from the fact that $\underline{\tilde{\Halg}}^+$
is positively graded, compare with the proof of Proposition 3.3.1 in \cite{HC}.

It is clear that $\Fun_7$ gives an equivalence  $\HC^\Xi(\underline{\tilde{\Halg}}^{+\wedge_0})^{\Lambda_\Halg}
\rightarrow \HC^\Xi(\underline{\tilde{\Halg}}^{+})^{\Lambda_\Halg}$.

{\bf The functor $\bullet_\dagger: \HC(\Halg)\rightarrow \HC^\Xi(\underline{\Halg}^+)^{\Lambda_\Halg}$.}

For $i>j$ set $\Fun_{i,j}:=\Fun_i\circ \Fun_{i-1}\circ\ldots \Fun_j$.
Let $\bullet_\dagger:=\Fun_{7,1}$. Let us list some properties of $\bullet_\dagger$.

\begin{Prop}\label{Prop:3.10.1}
\begin{enumerate}
\item
The functor $\bullet_\dagger$ is $S(\tilde{\param})$-bilinear, exact and monoidal. Also
it commutes with the operators $C_i$.
\item For $\Malg\in \HC(\Halg)$ we have
$\Malg_\dagger= 0$ if and only if $\Leaf\cap \VA(\Malg)=\varnothing$.
\item If
$\Malg$ is annihilated by $\tilde{\p}$ (from the left or from the right, these two conditions are equivalent),
then   $\Malg_\dagger$ is annihilated  by $\underline{\tilde{\p}}^+$.
\item $\VA(\Malg_\dagger)$ can be recovered from $\VA(\Malg)$ in the following way. Let $\VA(\Malg)=\bigcup_{i=1}^l \overline{\Leaf}_i$ be the decomposition into irreducible components and let $\Gamma_i$ be a subgroup
    in $\Gamma$ corresponding to $\Leaf_i$. Let $\Gamma_i^1,\ldots, \Gamma_i^{j_i}$ be all subgroups
    in $\underline{\Gamma}$ conjugate to $\Gamma_i$ inside $\Gamma$. Finally, let $\underline{\Leaf}_i^j$ be the symplectic leaf in $V_+/\underline{\Gamma}$ corresponding to $\Gamma_i^j$. Then
    $\VA(\Malg_\dagger)=\bigcup_{i,j}\underline{\overline{\Leaf}}_i^j$.
\end{enumerate}
\end{Prop}
When we say that $\bullet_\dagger$ is $S(\tilde{\param})$-bilinear we mean the following. For $\Malg\in \HC(\tilde{\Halg})$ let $\Malg[k]$ denote the same module with the grading shifted by $k$:
$\Malg[k]^n=\Malg^{n+k}$. Similarly, one defines the grading shifts for objects in
$\HC^\Xi(\underline{\tilde{\Halg}}^+)$. It is clear from the definition that $\Malg[k]_\dagger= \Malg_\dagger[k]$.
Now we have morphisms $\Malg[-2]\rightarrow \Malg$ given by left and right multiplications by
$\cb_i$ (and also the morphism $\Malg[-1]\rightarrow \Malg$ given by the multiplication by $\hbar$).
When we say that the functor $\bullet_\dagger$ is $S(\tilde{\param})$-bilinear, we mean that it is additive and it sends the morphism $[m\mapsto \cb_i m]$ to $[n\mapsto \cb_i n]$ (and similarly for the right multiplications and for the multiplication by $\hbar$).
\begin{proof}
All functors $\Fun_i$ are $S(\tilde{\param})$-bilinear and commute with the $C_i$'s.
The functor $\Fun_1$ is exact and the other functors $\Fun_i$ are equivalences.
Further, all functors $\Fun_i$ are monoidal (the tensor product functors on the intermediate categories
are tensor products of  (sheaves of) Poisson bimodules). Hence (1).

 (3) follows from the observation
that all functors $\Fun_i$ preserve the condition that a bimodule is annihilated by an appropriate ideal.

Let us prove  (4).
Since $\bullet_\dagger$ is exact, we may replace $\Malg$ with $\Malg/\tilde{\param}\Malg$ and assume that
$\tilde{\param}$ annihilates $\Malg$. The associated formal scheme  of $\Fun_1(\Malg)$ is nothing else but
  $\VA(\Malg)^{\wedge_\Leaf}:=\VA(\Malg)\cap\X^{\wedge}_\Leaf$. This follows from assertion (4) of Lemma \ref{Lem:3.10.1}. The associated formal scheme  of $\Fun_{4,1}(\Malg)$ is the preimage of
$\VA(\Malg)^{\wedge_\Leaf}$  in $\underline{\X}^{\wedge}_{\underline{\Leaf}}=\underline{\Leaf}\times
(V^*_+)^{\wedge}_0/\underline{\Gamma}$. It equals
$\underline{\Leaf}\times \VA(\Fun_{5,1}(\Malg))$. The functor $\Fun_6$ does not change the associated schemes.
Finally, $\VA(\Fun_{6,1}(\Malg))$ is just the intersection of $\VA(\Malg_\dagger)\subset V_+^*/\underline{\Gamma}$ with
$(V^*_+)^{\wedge}_0/\underline{\Gamma}$.

The claim of (4) easily follows from the description in the previous paragraph.

(2) stems from assertion (4).
\end{proof}


\begin{Rem}\label{Rem:fun_ext}
We can extend $\bullet_\dagger$ to a functor between the ind-completions $\widehat{\HC}(\tilde{\Halg}),
\widehat{\HC}^\Xi(\underline{\tilde{\Halg}}^+)$. The category $\widehat{\HC}(\tilde{\Halg})$ consists
of all graded Poisson $\tilde{\Halg}$-bimodules that are sums of their HC bimodules and the category
$\widehat{\HC}^\Xi(\underline{\tilde{\Halg}}^+)$ admits a similar description.
\end{Rem}

We will need the following remark in the proof of Theorem \ref{Thm:4'}.

\begin{Rem}\label{Rem:3.10.2}
Pick $l=1,\ldots,r+1$ and let $\hbar_l$ be a new independent variable.
Set $\tilde{\Halg}_{(l)}=\Halg_{(l)}[\hbar_l]/(\cb_l-\hbar_l^2), \tilde{\Zalg}_{(l)}=\z(\tilde{\Halg}_{(l)})=
\Zalg_l[\hbar_l]/(\cb_l-\hbar_l^2)$ (the latter equality stems, for example, from the Satake isomorphism, Proposition \ref{Prop:1.31} and from  the fact that $\Halg_{(l)}$ is a flat $S(\param_{(l)})$-module). Similarly,
we can introduce the algebras $\tilde{\underline{\Halg}}_{(l)}^+, \tilde{\underline{\Zalg}}_{(l)}^+$.

Let $\widehat{\HC}(\tilde{\Zalg}_{(l)})$ denote the category of locally finitely generated graded Poisson $\tilde{\Zalg}_{(l)}$-modules
(=bimodules). Define $\widehat{\HC}^\Xi(\underline{\tilde{\Zalg}}_{(l)}^+)$ in a similar way. We can define the functor
$\bullet_\dagger: \widehat{\HC}(\tilde{\Zalg}_{(l)})\rightarrow \widehat{\HC}^\Xi(\underline{\tilde{\Zalg}}_{(l)}^+)$ completely analogously to $\bullet_\dagger: \widehat{\HC}(\tilde{\Halg})\rightarrow \widehat{\HC}^\Xi(\underline{\tilde{\Halg}}^+)$ (the functor
$\Fun_6$ should be replaced with the identity functor). All properties of $\bullet_\dagger$ including
Proposition \ref{Prop:3.10.1} are transferred to the present situation without any noticeable modifications.
\end{Rem}

%
%
%

\subsection{Functor $\bullet^\dagger:\HC^\Xi(\underline{\tilde{\Halg}}^+)^{\Lambda_\Halg}\rightarrow \HC(\tilde{\Halg})$}\label{SUBSECTION_Fun_up_dag}
In this subsection we will construct a functor $\bullet^\dagger$ from $ \widehat{\HC}^\Xi(\underline{\tilde{\Halg}}^+)^{\Lambda_\Halg}$ to $\widehat{\HC}(\tilde{\Halg})$. This functor equals $\Gun_1\circ \Fun_{7,2}^{-1}$, where $\Gun_1$ is defined as follows.

Pick $\M\in \HC(\tilde{\Halg}^{\wedge_\Leaf}|_{\Leaf})$. Then its space $\M(\Leaf)$ of global sections
is a $\K^\times$-equivariant Poisson $\tilde{\Halg}$-bimodule. Let $\M(\Leaf)_{l.f.}$ denote the sum of
all  HC sub-bimodules of $\M(\Leaf)$. Note, however, that $\M(\Leaf)_{l.f.}$ is not
a priori HC itself because we do not know at the moment that $\M(\Leaf)_{l.f.}$ is finitely generated.
Set $\Gun_1(\M)=\M(\Leaf)_{l.f.}$. This is a functor from $\HC^\Xi(\underline{\tilde{\Halg}}^+)$ to the category $\widehat{\HC}(\tilde{\Halg})$ from Remark \ref{Rem:fun_ext}.
Below we will see that $\Nalg^\dagger\in \HC_{\overline{\Leaf}}(\tilde{\Halg})$ provided
$\Nalg\in \HC^\Xi_0(\underline{\tilde{\Halg}}^+)^{\Lambda_\Halg}$. It is easy to see that $\bullet^\dagger$ naturally extends
to a functor  $\widehat{\HC}^\Xi(\underline{\tilde{\Halg}}^+)\rightarrow \widehat{\HC}(\tilde{\Halg})$.


Let us list some properties of the functor $\bullet^\dagger$.

\begin{Prop}\label{Prop:3.15.1}
\begin{enumerate}
\item The functor $\bullet^\dagger$ is left exact, $S(\tilde{\param})$-bilinear
and commutes with the operators $C_i$.
\item The functor $\bullet^\dagger: \widehat{\HC}^\Xi(\underline{\tilde{\Halg}}^+)^{\Lambda_\Halg}\rightarrow
\widehat{\HC}(\tilde{\Halg})$ is right adjoint to $\bullet_\dagger: \widehat{\HC}(\tilde{\Halg})
\rightarrow \widehat{\HC}^\Xi(\underline{\tilde{\Halg}}^+)^{\Lambda_{\Halg}}$.
\end{enumerate}
\end{Prop}
\begin{proof}
To prove (1) we remark that $\Gun_1$ is easily seen to be left exact, while $\Fun_i, i=2,\ldots,7$
are equivalences (of abelian categories). Also all functors under consideration are $S(\tilde{\param})$-bilinear
and commute with the $C_i$'s.

Let us prove (2). It is enough to show that for $\Malg\in \HC(\tilde{\Halg})$ and $\M\in \HC(\tilde{\Halg}^{\wedge_\Leaf}|_{\Leaf})^{\Lambda_\Halg}$ the Hom-spaces
$\Hom(\Malg, \M(\Leaf)_{l.f.})$ and $\Hom(\Fun_1(\Malg), \M)$ are naturally
isomorphic.

Let us remark that $\Hom(\Malg,\M(\Leaf)_{l.f})=\Hom(\Malg, \M(\Leaf))$. Any homomorphism
$\varphi\in \Hom(\Malg, \M(\Leaf))$  can be localized to a unique homomorphism  $\varphi|_\Leaf:\Malg|_{\Leaf}\rightarrow \M$ (here $\Malg|_{\Leaf}$ stands for the sheaf
theoretic pull-back of $\Malg|_{\X}$ to $\Leaf$). It is clear that $\varphi|_\Leaf$  is continuous in
the $\tilde{\p}|_{\Leaf}$-adic topology. The completion of $\Malg|_\Leaf$  with respect to the
$\tilde{\p}|_{\Leaf}$-adic topology coincides with $\Fun_1(\Malg)$. So $\varphi|_\Leaf$
extends to a unique continuous homomorphism $\Fun_1(\Malg)\rightarrow \M$. Summarizing, we get a
natural map  $\Hom(\Malg,\M(\Leaf)_{l.f})\rightarrow \Hom(\Malg^{\wedge_\Leaf}|_{\Leaf},\M)$.

Conversely, pick $\psi\in \Hom(\Fun_1(\Malg),\M)$. Then we get the corresponding
homomorphism $\psi(\Leaf): \Fun_1(\Malg)(\Leaf)\rightarrow \M(\Leaf)$
between the spaces of global sections. Compose
$\psi(\Leaf)$ with a natural homomorphism $\Malg\rightarrow \Fun_1(\Malg)(\Leaf)$.
So we get the natural map $\Hom(\Fun_1(\Malg),\M)\rightarrow \Hom(\Malg, \M(\Leaf)_{l.f.})$.
It is clear that the maps we have constructed are inverse to each other.
\end{proof}

Now we are going to study the restriction of  $\bullet^\dagger$ to a certain subcategory of $\HC^\Xi(\underline{\tilde{\Halg}}^+)^{\Lambda_\Halg}$. Namely,
for $\lambda\in \Lambda_{\Halg}$ let $\HC^\Xi_{\underline{\tilde{\p}}^+}(\underline{\tilde{\Halg}}^+)_\lambda$
denote the full subcategory of $\HC^\Xi(\underline{\tilde{\Halg}}^+)^{\Lambda_\Halg}$
consisting of all bimodules satisfying the following two conditions
\begin{itemize}
\item[(i)] they are annihilated by $\underline{\tilde{\p}}^+$ from the left
(and then automatically also from the right).
\item[(ii)] The operator $C_i=\{\cb_i,\cdot\}$ is the scalar $\lambda_i$ for every $i=1,\ldots,r$.
\end{itemize}
 Also let $\HC_{\tilde{\p}}(\tilde{\Halg})_\lambda$
be the full subcategory of $\HC(\tilde{\Halg})$ consisting of all modules annihilated
by $\tilde{\p}$ from the left (and from the right) and satisfying (ii). Tautologically,  $\HC_{\tilde{\p}}(\tilde{\Halg})_\lambda\subset
\HC_{\overline{\Leaf}}(\tilde{\Halg})$. Set $\HC_{\tilde{\p},\partial\Leaf}(\tilde{\Halg})_\lambda:=\HC_{\tilde{\p}}(\tilde{\Halg})_\lambda\cap
\HC_{\partial\Leaf}(\tilde{\Halg}), \HC_{\tilde{\p},\Leaf}(\tilde{\Halg})_\lambda:=\HC_{\tilde{\p}}(\tilde{\Halg})_\lambda/\HC_{\tilde{\p},\partial\Leaf}(\tilde{\Halg})_\lambda$.
According to Proposition \ref{Prop:3.10.1}, $\bullet_\dagger$ induces a functor
$\HC_{\tilde{\p},\Leaf}(\tilde{\Halg})_\lambda\rightarrow \HC_{\underline{\tilde{\p}}^+}(\underline{\tilde{\Halg}}^+)_\lambda$.
The induced functor will also be denoted by $\bullet_\dagger$. We remark that $\underline{\tilde{\p}}^+$ has finite
codimension in $\underline{\tilde{\Halg}}^+$ and so all modules in $\HC_{\underline{\tilde{\p}}^+}(\underline{\tilde{\Halg}}^+)_\lambda$
are finite dimensional.

\begin{Prop}\label{Prop:3.15.2}
The image of $\HC_{\underline{\tilde{\p}}^+}(\underline{\tilde{\Halg}}^+)_\lambda$ under $\bullet^\dagger$
lies in $\HC_{\tilde{\p}}(\tilde{\Halg})_\lambda$. The induced functor $\bullet^\dagger: \HC_{\underline{\tilde{\p}}^+}(\underline{\tilde{\Halg}}^+)_\lambda\rightarrow
\HC_{\tilde{\p},\Leaf}(\tilde{\Halg})_\lambda$ is quasiinverse to $\bullet_\dagger:
\HC_{\tilde{\p},\Leaf}(\tilde{\Halg})_\lambda\rightarrow \HC_{\underline{\tilde{\p}}}(\underline{\tilde{\Halg}}^+)_\lambda$.
\end{Prop}
\begin{proof}
The proof is in several steps.

{\it Step 1.} Let $\Nalg\in \HC_{\underline{\tilde{\p}}^+}(\underline{\tilde{\Halg}}^+)_\lambda$.
The  left and right multiplication actions of the sheaf $\CC(\W_{\hbar}|_{\underline{\Leaf}}\widehat{\otimes}_{\K[[\hbar]]}\underline{\tilde{\Halg}}^{+\wedge_0})$
on $\Fun_{7,5}^{-1}(\Nalg)$ factor through $\CC(\Str_{\underline{\Leaf}}\#\underline{\Gamma})$. The latter sheaf is naturally identified with $\Str_{\underline{\Leaf}}\otimes \CC(\K\underline{\Gamma})$ because $\underline{\Gamma}$
acts trivially on $\Str_{\underline{\Leaf}}$.
Furthermore, as a $\Xi\times \K^\times$-equivariant $\CC(\Str_{\underline{\Leaf}}\#\underline{\Gamma})$-bimodule, $\Fun_{7,5}^{-1}(\Nalg)$
is just $\Str_{\underline{\Leaf}}\otimes \CC(\Nalg)$.

The multiplication actions of $\Halg^{\wedge_\Leaf}|_{\underline{\Leaf}}$ on $\Fun_{7,4}^{-1}(\Nalg)$ still
factors through   $\Str_{\underline{\Leaf}}\otimes \CC(\K\underline{\Gamma})$. This follows directly from
the construction of the functor $\Fun_4$. Moreover, we claim that $\Fun_{7,4}^{-1}(\Nalg)$ is still
isomorphic to $\Str_{\underline{\Leaf}}\otimes \CC(\Nalg)$ as a
$\Xi\times \K^\times$-equivariant $\Str_{\underline{\Leaf}}\otimes \CC(\K\underline{\Gamma})$-bimodule.
This again follows from the construction of $\Fun_4$ and (ii): indeed, (\ref{eq:4.9.11}) implies that
$\Fun_4^{-1}$ is just a regluing by means of a coboundary with the necessary invariance properties.

So as a
$(\Str_{\underline{\Leaf}}\otimes \CC(\K\underline{\Gamma}))^\Xi$-module $\Fun_{7,3}^{-1}(\Nalg)$ is just $(\Str_{\underline{\Leaf}}\otimes \CC(\Nalg))^\Xi$. Now $\Fun_2^{-1}$ is just the pull-back
with respect to the isomorphism $$\theta_0:\iota^*(SV\#\Gamma|_{\X})\xrightarrow{\sim} (\Str_{\underline{\Leaf}}\otimes \CC(\K\underline{\Gamma}))^\Xi$$ induced from (\ref{eq:0.2.40}); here $\iota$ is the inclusion $\Leaf\hookrightarrow \X$.
In particular, we see that the space of global sections of  $\Fun_{7,2}^{-1}(\Nalg)$ is just
$(SV_0\otimes \CC(\Nalg))^\Xi$. The algebra $SV_0$ is finite over $\K[\overline{\Leaf}]$
and so $(SV_0\otimes \CC(\Nalg))^\Xi$ is a finitely generated $\K[\overline{\Leaf}]$-module.
So we have checked that
\begin{itemize}
\item
$\Nalg^\dagger$  is naturally identified with $(SV_0\otimes \CC(\Nalg))^\Xi$,
\item  the left and right multiplication  actions  of $\tilde{\Halg}$ on $\Nalg^\dagger$ factor through
$\overline{\iota}^*(SV\#\Gamma)$, where $\overline{\iota}: \overline{\Leaf}\hookrightarrow
\X$ is the inclusion,
\item the multiplication action of $\overline{\iota}^*(SV\#\Gamma)$ on $(SV_0\otimes \CC(\Nalg))^\Xi$
is obtained from the natural homomorphism $\overline{\iota}^*(SV\#\Gamma)\hookrightarrow
(SV_0\otimes \CC(\K\underline{\Gamma}))^\Xi$.
\item $\Nalg^\dagger$ is a finitely generated  $\overline{\iota}^*(SV\#\Gamma)$-module.
\end{itemize}

In particular, we have checked that $\bullet^\dagger$ maps $\HC_{\underline{\tilde{\p}}^+}(\underline{\tilde{\Halg}}^+)_\lambda$
 to $\HC_{\tilde{\p}}(\tilde{\Halg})_\lambda$.

{\it Step 2.} Now let $\Malg\in \HC_{\tilde{\p}}(\tilde{\Halg})_\lambda$. The action of $\tilde{\Halg}$ on $\Malg$ factors
through $\overline{\iota}^*(SV\#\Gamma)$. The natural map $\Malg\rightarrow \overline{\iota}^*(\Malg)$ is therefore
an isomorphism. Then $\Fun_1(\Malg)$ is just the pull-back  $\iota^*(\Malg)$. Since $\Fun_{7,2}$ is a category
equivalence, we deduce from the previous step that $\Gun_1(\Fun_1(\Malg))=\iota_*(\iota^*(\Malg))$ is a  finitely
generated $\K[\overline{\Leaf}]$-module. We are going to prove that both the kernel and the cokernel
of the natural map $\Malg\rightarrow \iota_*(\iota^*(\Malg))$ are supported on $\partial \Leaf$. This homomorphism is nothing else but the natural
homomorphism $\Malg\rightarrow \kappa_*(\kappa^*(\Malg))$, where $\kappa$ stands for the inclusion
$\Leaf\hookrightarrow \overline{\Leaf}$. Its (sheaf-theoretic) restriction to $\Leaf$ is just
the isomorphism $\kappa^*(\Malg)\rightarrow \kappa^*(\Malg)$.

So we see that $(\bullet_\dagger)^\dagger$ is isomorphic to the identity functor of $\HC_{\tilde{\p},\Leaf}(\tilde{\Halg})_\lambda$.

{\it Step 3.} Now it remains to prove that $(\bullet^\dagger)_\dagger$ is isomorphic to
the identity functor of $\HC_{\underline{\tilde{\p}}^+}(\underline{\tilde{\Halg}}^+)_\lambda$.
Again, this reduces to checking that the natural homomorphism $\Fun_1(\Gun_1(\M))\rightarrow
\M$ is an isomorphism for any $\M\in \HC(\tilde{\Halg}^{\wedge_\Leaf}|_{\Leaf})_{\lambda}$ annihilated
by $\tilde{\p}^{\wedge_\Leaf}|_{\Leaf}$. But on the level of $\Str_{\Leaf}$-modules the functor
$\Fun_1\circ\Gun_1$ is just $\kappa^*\circ \kappa_*$. Since  $\kappa_*(\M)$ is a finitely generated
$\K[\overline{\Leaf}]$-module (see step 1), we see that the natural homomorphism $\M\rightarrow
\kappa^*(\kappa_*(\M))$ is an isomorphism.
\end{proof}

Proposition \ref{Prop:3.10.1} implies that $\bullet_\dagger$ restricts to a functor $\HC_{\overline{\Leaf}}(\tilde{\Halg})\rightarrow \HC_{0}^\Xi(\underline{\tilde{\Halg}}^+)^{\Lambda_\Halg}$.
The latter induces a functor $\bullet_\dagger: \HC_{\Leaf}(\tilde{\Halg})\rightarrow \HC_{0}^\Xi(\underline{\tilde{\Halg}}^+)^{\Lambda_\Halg}$.

\begin{Prop}\label{Cor:3.15.2}
\begin{enumerate}
\item $\Nalg^\dagger\in \HC_{\overline{\Leaf}}(\tilde{\Halg})$ for $\Nalg\in \HC_{0}^\Xi(\underline{\tilde{\Halg}}^+)^{\Lambda_\Halg}$.
\item For any $\Malg\in \HC_{\overline{\Leaf}}(\tilde{\Halg})$ the kernel and the cokernel of the natural homomorphism $\Malg\rightarrow (\Malg_\dagger)^\dagger$ lie in $\HC_{\partial\Leaf}(\tilde{\Halg})$.
    In other words, $\bullet^\dagger: \HC^\Xi_0(\underline{\tilde{\Halg}}^+)^{\Lambda_\Halg}\rightarrow
    \HC_\Leaf(\tilde{\Halg})$ is the left inverse to $\bullet_\dagger:
    \HC_{\Leaf}(\tilde{\Halg})\rightarrow \HC_{0}^\Xi(\underline{\tilde{\Halg}}^+)^{\Lambda_\Halg}$.
\item The functor $\bullet_\dagger: \HC_\Leaf(\tilde{\Halg})\rightarrow \HC^\Xi_0(\underline{\tilde{\Halg}})^{\Lambda_\Halg}$
is a fully faithful embedding.
\end{enumerate}
\end{Prop}
\begin{proof}
(1). We are going to prove the following statement using the descending induction on $l$.
\begin{itemize}
\item[($*_l$)] Suppose $\Nalg$ is annihilated by $\param^{(l)}$. Then $\Nalg^\dagger\in \HC_{\overline{\Leaf}}(\tilde{\Halg})$.
\end{itemize}

The base is $l=r$. In this case $\Nalg$ is annihilated by $\param$ and hence is finite dimensional.
So it has a finite filtration, whose quotients are objects from $\HC^\Xi_{\underline{\p}^+}(\underline{\tilde{\Halg}}^+)_\lambda$
for various $\lambda\in \Lambda_\Halg$. To prove $(*_r)$ use Proposition \ref{Prop:3.15.2}
and the left exactness of $\bullet^\dagger$.

Now suppose that $(*_{l+1})$ is proved and let $\Nalg$ be annihilated by $\param^{(l)}$.
Consider the  homomorphism $\Nalg[-2]\rightarrow \Nalg$ given by the left multiplication by $\cb_l$ (recall that $\cb_0:=\tb$). Clearly, $\bullet^\dagger$ commutes with the grading
shifts and is a $S(\param)$-bilinear functor, Proposition \ref{Prop:3.15.1}.
Let $\Nalg^0$ stand for the $\cb_l$-torsion in $\Nalg$. Applying $(*_{l+1})$ and a filtration
argument, we see that $\Nalg^{0\dagger}\in \HC_{\overline{\Leaf}}(\tilde{\Halg})$. Since $\bullet^\dagger$
is left exact, this reduces the proof to the case when $\Nalg$ is $\K[\cb_l]$-flat.

So we get a monomorphism $\Nalg^{\dagger}[-2]\xrightarrow{\cb_l\cdot} \Nalg^\dagger$.
Its  cokernel is contained in $(\Nalg/\cb_l \Nalg)^\dagger$. The latter lies in $\HC_{\overline{\Leaf}}(\tilde{\Halg})$
by $(*_{l+1})$. It remains to show that $\Nalg^\dagger$ is a finitely generated $\Halg$-bimodule
(or a left module). This is done completely analogously to the proof of Lemma 3.3.3 in \cite{HC}.

(2). Let $\Malg',\Malg''$ denote the kernel and the cokernel of the natural homomorphism
$\Malg\rightarrow (\Malg_\dagger)^\dagger$. First  of all, let us check  that $\Malg'\in
\HC_{\partial \Leaf}(\tilde{\Halg})$. By Proposition \ref{Prop:3.10.1}, the latter is equivalent
to $\Malg'_\dagger=0$. But since $\bullet_\dagger$ is exact, we have the following exact
sequence: $0\rightarrow \Malg'_\dagger\rightarrow \Malg_\dagger\rightarrow [(\Malg_\dagger)^\dagger]_\dagger$.
From Proposition \ref{Prop:3.15.1} it follows that $\bullet^\dagger: \HC^\Xi_0(\underline{\tilde{\Halg}}^+)
\rightarrow \HC_{\overline{\Leaf}}(\tilde{\Halg})$ is right adjoint to $\bullet_\dagger:
\HC_{\overline{\Leaf}}(\tilde{\Halg})\rightarrow \HC^\Xi_0(\underline{\tilde{\Halg}}^+)$. In particular,
the natural homomorphism $\Malg_\dagger\rightarrow [(\Malg_\dagger)^\dagger]_\dagger$ is injective.
So $\Malg'_\dagger=0$.

It remains to show that $\Malg''_\dagger=0$. First of all, suppose that $\Malg$ is annihilated by
$\tilde{\p}$. Then the claim follows from Proposition \ref{Prop:3.15.2}. Next, consider the case
when $\Malg$ is annihilated by $\param$. In this case $\Malg$ admits a finite filtration
whose quotients lie in $\HC_{\underline{\tilde{\p}}^+}(\underline{\tilde{\Halg}}^+)_\lambda$. Since both functors $\bullet_\dagger, [(\bullet_\dagger)^\dagger]_\dagger$ are left exact, we see that   $\Malg''_\dagger=0$,
compare with the proof of  \cite{HC}, Proposition 3.3.4.

Consider the general case. By the previous paragraph,  the natural morphism
$\Malg_\dagger/\param\Malg_\dagger\cong [\Malg/\param\Malg]_\dagger\rightarrow
[([\Malg/\param\Malg]_\dagger)^\dagger]_\dagger$ is an isomorphism. Now a simple diagram chase shows  that the projection
of $\param[(\Malg_\dagger)^\dagger]_\dagger$ to $\Malg''_\dagger$ is surjective. Therefore $\param \Malg''_\dagger=\Malg''_\dagger$ and so $\Malg''_\dagger=0$.

(3) follows from (2) and some standard abstract nonsense.
\end{proof}


\begin{Rem}\label{Rem:3.15.2}
We use the notation introduced in Remark \ref{Rem:3.10.2}.

Similarly to the above we can define a functor
$\bullet^\dagger: \HC^\Xi(\underline{\tilde{\Zalg}}^+_{(l)})\rightarrow \widehat{\HC}(\tilde{\Zalg}_{(l)})$.  Propositions \ref{Prop:3.15.1},\ref{Prop:3.15.2} and Corollary \ref{Cor:3.15.2} generalize to the present situation
in a straightforward way. In particular, we have a functor $\bullet^\dagger:
\HC^\Xi_0(\underline{\tilde{\Zalg}}^+_{(l)})\rightarrow \HC_\Leaf(\tilde{\Zalg}_{(l)})$.
\end{Rem}

\subsection{Proof of Theorem \ref{Thm:5I}}\label{Proof 5I}
Let $\Malg\in \HC(\tilde{\Halg})$. Let $\Nalg\subset \Malg_\dagger$ be an HC sub-bimodule.
Let $\Nalg^{\dagger_\Malg}$ denote the preimage of $\Nalg^\dagger\subset (\Malg_\dagger)^{\dagger}$ under the natural
homomorphism $\Malg\rightarrow (\Malg_\dagger)^\dagger$.

The following lemma describes the properties of $\Nalg^{\dagger_\Malg}$.

\begin{Lem}\label{Lem:3.16.1}
\begin{enumerate}
\item $\Malg':=\Nalg^{\dagger_\Malg}$ is the largest (with respect to inclusion) HC sub-bimodule
in $\Malg$ with $\Malg'_\dagger\subset \Nalg$.
\item Suppose $\Malg_1\subset \Nalg^{\dagger_\Malg}$. Then $\Nalg^{\dagger_\Malg}/\Malg_1=(\Nalg/\Malg_{1\dagger})^{\dagger_{\Malg/\Malg_1}}$.
\end{enumerate}
\end{Lem}
\begin{proof}
(1):
Let us check that $(\Nalg^{\dagger_\Malg})_\dagger\subset \Nalg$ or, in other words,
that the composition $(\Nalg^{\dagger_\Malg})_\dagger\hookrightarrow \Malg_\dagger\twoheadrightarrow \Malg_\dagger/\Nalg$ is zero. By the adjointness property, the composition $(\Nalg^{\dagger_\Malg})_\dagger\hookrightarrow \Malg_\dagger\twoheadrightarrow \Malg_\dagger/\Nalg$ gives rise to a morphism $\Nalg^{\dagger_\Malg}\rightarrow
[\Malg_\dagger/\Nalg]^\dagger$ and it is enough to check that the latter is zero. By the construction,
the last morphism factors as $$\Nalg^{\dagger_\Malg}\rightarrow \Malg\rightarrow [\Malg_\dagger]^\dagger\rightarrow [\Malg_\dagger]^\dagger/\Nalg^\dagger\rightarrow [\Malg_\dagger/\Nalg]^\dagger.$$ The composition of the first three
morphisms in the sequence above is zero by the definition of $\Nalg^{\dagger_\Malg}$. So $(\Nalg^{\dagger_\Malg})_\dagger\subset \Nalg$.

Now let $\Malg'\subset \Malg$ be such that $\Malg'_\dagger\subset \Nalg$. We need to show that
the image of $\Malg'$ in $(\Malg_\dagger)^\dagger$ lies in $\Nalg^\dagger$. This follows easily
from the adjointness property similarly to the previous paragraph.

(2): This follows from the exactness of $\bullet_\dagger$ and assertion (1).
\end{proof}

The main result of this subsection
and a crucial step in the proof of Theorem \ref{Thm:5I} is the following result.
This theorem (together with a proof) is a complete analog of Theorem 4.1.1 from \cite{HC}.

\begin{Thm}\label{Thm:3.16.1}
Let $\Malg\in \HC(\tilde{\Halg})$ and $\Nalg\subset \Malg_\dagger$ be such that
$\Malg_\dagger/\Nalg\in \HC_0^\Xi(\underline{\tilde{\Halg}}^+)$. Then
$\Nalg=(\Nalg^{\dagger_\Malg})_\dagger$.
\end{Thm}
\begin{proof}
Thanks to the exactness of $\bullet_\dagger$ and assertion (2) of Lemma \ref{Lem:3.16.1}, we may  assume that $\Nalg^{\dagger_\Malg}=0$.
So $\Malg$ embeds into $(\Malg_\dagger/\Nalg)^\dagger$. The latter is an object in
$\HC_{\overline{\Leaf}}(\tilde{\Halg})$ by Proposition \ref{Cor:3.15.2}. So
$\Malg$ lies in $\HC_{\overline{\Leaf}}(\tilde{\Halg})$. Since $\Malg_\dagger=
[(\Malg_\dagger)^\dagger]_\dagger$, below in the proof we may (and will) assume that $\Malg=(\Malg_\dagger)^\dagger$.
 So $\Nalg^\dagger=0$. We want to show that
$\Nalg=0$.

We will prove the following claim using the descending induction on $l$.
\begin{itemize}
\item[($*_l$)]
$\Nalg=0$ provided $\Malg$ is annihilated by the left multiplication by $\param^{(l)}$.
\end{itemize}

The base is the case $l=r$. In this case $\Malg$ is annihilated by $\param$.

We remark that since $\Nalg$ lies in $\HC^\Xi_0(\underline{\tilde{\Halg}}^+)$ and is annihilated
by $\param$, there is a nonzero HC sub-bimodule $\Nalg_1$ in $\Nalg$ annihilated by $\underline{\tilde{\p}}^+$.
But then $\Nalg_1^\dagger$ is nonzero by Proposition \ref{Prop:3.15.2} and embeds into $\Nalg^\dagger$.
Contradiction.

Now suppose we have proved $(*_{l+1})$. Let $\Malg$ be annihilated by $\param^{(l)}$.
First of all, consider the case when $\Malg$ is annihilated, in addition, by the left multiplication by
$\cb_l^n$ for some $n$. Let $\Malg_1$ be the annihilator of $\cb_l$ in $\Malg$. Since
 $\bullet_\dagger$ is exact and $\cb_l$-linear,  we see that $\Malg_{1\dagger}$ is the annihilator of
$\cb_l$ in $\Malg_\dagger$, compare with Corollary 2.4.3 in \cite{HC} (it does not matter
whether we consider left or right actions). Also $\cb_l$ acts locally nilpotently on $\Nalg$. So
if $\Nalg\neq 0$, then $\Nalg\cap \Malg_{1\dagger}\neq 0$. It follows that $(\Nalg\cap \Malg_{1\dagger})^\dagger\neq
0$ by $(*_{l+1})$ and hence $\Nalg^\dagger\neq 0$. So $(*_l)$ is proved provided $\Malg$ is annihilated
by $\cb_l^n$ for some $n$.

Consider the general case. By the above, we may assume that $\Nalg$ is a flat left $\K[\cb_l]$-module. For sufficiently
large $n$ the module $\cb_l^n \Malg$ has no torsion.  So we may assume that
$\Malg$ is $\K[\cb_l]$-flat. Therefore $\Nalg_1^\dagger=0$ for any HC sub-bimodule $\Nalg_1\subset \Malg_\dagger$
with $\cb_l \Nalg_1\subset \Nalg$.
So we may assume, in addition, that $\Nalg\subset \Malg^\dagger$ is $\cb_l$-saturated, i.e., $\cb_l m\in \Nalg$
implies $m\in \Nalg$, equivalently $\Malg_\dagger/\Nalg$ is $\K[\cb_l]$-flat.

Set $\Nalg_n:=\Nalg+ \cb_l^{n} \Malg_\dagger$. We remark that $(\Nalg_n^{\dagger_\Malg})_\dagger=\Nalg_n$. Indeed,
apply assertion (2) of Lemma \ref{Lem:3.16.1} to $\cb_l^{n}\Malg\subset \Nalg_n^\dagger$ and use the equality $[(\Nalg_n/\cb_l^n\Malg_\dagger)^{\dagger_{\Malg/\cb_l^n\Malg}}]_\dagger=\Nalg_n/\cb_l^n \Malg_\dagger$
that follows from the above.


Set $T_k:=\Nalg_k^{\dagger}/\Nalg_{k+1}^{\dagger}$. We have the natural maps
$\Nalg_{k}[-2]\rightarrow \Nalg_{k+1}$ given by the left multiplication by
$\cb_l$.  Applying
$\bullet^{\dagger}$, we get a morphism
\begin{equation}\label{eq:3.16.1} \Nalg_{k}^{\dagger}[-2]\rightarrow \Nalg^{\dagger}_{k+1}\end{equation}
again given by the multiplication by $\cb_l$.
We deduce that  $\Nalg_{k+1}^{\dagger}$
contains $\cb_l\Nalg_k^{\dagger}$. In other words, $\cb_l$ acts trivially on $T_k$.
Also the map (\ref{eq:3.16.1}) gives rise to a morphism
\begin{equation}\label{eq:3.16.2}T_{k}[-2]\rightarrow T_{k+1}.\end{equation}
We can form the graded $\Halg_{(l+1)}$-module $T:=\bigoplus_{k=0}^\infty T_k$.
Equip $T$ with the structure of $\K[\cb_l]$-module by setting $\cb_l m$ for $m\in T_k$ to be the
image of $m$ under the homomorphism (\ref{eq:3.16.2}). So $T$ becomes an $\Halg_{(l+1)}[\cb_l]$-module.

Suppose that we know that $T$ is  finitely generated as an $\Halg_{(l+1)}[\cb_l]$-module.
In particular, there is $N\in \N$ such that $T_{N+n}=\cb_l^n T_N$ for any $n\geqslant 0$
or, equivalently, $\Nalg_{N+n}^{\dagger}= \cb_l^n\Nalg_{N}^{\dagger}+
\Nalg_{N+n+1}^\dagger$. From this using the induction on $k$ we can check  that \begin{equation}\label{eq:3.6.13}\Nalg_{N+n}^{\dagger}= \cb_l^n \Nalg_N^{\dagger}+ \Nalg_{N+n+k}^\dagger.\end{equation}
Set $\Malg':=\bigcap_n \Nalg_n^\dagger$.  Then $\Malg'$ is an  HC sub-bimodule in $\Malg$, whose image in $\Malg/\cb_l^n\Malg$ coincides with the image of
$\Nalg_m^\dagger, m\gg 0$. The latter is nonzero since $(\Nalg_m^\dagger)_\dagger=\Nalg_m$. Therefore $\Malg'$ is nonzero. Applying assertion (1) of Lemma \ref{Lem:3.16.1},  we see that $\Malg'_\dagger\subset \Nalg_m$
 for any $m$. But the intersection of $\Nalg_m$ is just $\Nalg$, so $\Malg'_\dagger\subset\Nalg$. Applying
Lemma \ref{Lem:3.16.1}  again, we get $\Malg'\subset \Nalg^\dagger$, contradiction with $\Nalg^\dagger=0$.

It remains to verify that $T$ is  finitely generated as an $\Halg_{(l+1)}[\cb_l]$-module.
Set $C:=(\Malg_{\dagger}/\Nalg_1)^\dagger$.
This is an object in $\HC_{\overline{\Leaf}}(\tilde{\Halg})$
annihilated by $\param^{(l+1)}$ and, in particular, a graded $\Halg_{(l+1)}$-bimodule that is finitely
generated as a left $\Halg_{(l+1)}$-module. The claim that $T$ is finitely generated stems from
the following lemma. The proof of the lemma will complete that of the theorem.
\end{proof}
\begin{Lem}\label{Lem:3.16.2}
There is an embedding $T\hookrightarrow C[\cb_l]$ of $\Halg_{(l+1)}[\cb_l]$-modules.
\end{Lem}
\begin{proof}
We will construct embeddings $\iota_i: T_i\rightarrow C$ such that
$\iota_{i+1}(\cb_l m)=\iota_i(m)$. Since $\Malg_\dagger/\Nalg$ is a flat $\K[\cb_l]$-module, we have the following exact sequence
\begin{equation}\label{eq:3.16.4}
0\rightarrow  \Malg_{\dagger}/\Nalg_1\xrightarrow{\cb_l^n\cdot} \Malg_{\dagger}/\Nalg_{n+1}\rightarrow \Malg_{\dagger}/\Nalg_n\rightarrow 0.
\end{equation}
 Applying $\bullet^\dagger$ to
(\ref{eq:3.16.4}), we get the exact sequence
\begin{equation}\label{eq:3.16.4'}
0\rightarrow  (\Malg_{\dagger}/\Nalg_1)^\dagger\xrightarrow{\cb_l^n\cdot} (\Malg_{\dagger}/\Nalg_{n+1})^\dagger\rightarrow (\Malg_{\dagger}/\Nalg_n)^\dagger.
\end{equation}
We have a natural embedding $T_n\hookrightarrow \Malg/\Nalg_{n+1}^\dagger\hookrightarrow (\Malg_\dagger/\Nalg_{n+1})^\dagger$, whose image in $(\Malg_{\dagger}/\Nalg_n)^\dagger$
vanishes. This gives rise to an embedding $\iota_n:T_n\hookrightarrow C$.
The claim on the compatibility with the multiplication by $\cb_l$ stems from the following commutative diagram.

\begin{picture}(120,70)
\put(15,62){$0$}\put(15,42){$C$}
\put(4,22){$(\Malg_\dagger/\Nalg_{n+1})^\dagger$}
\put(5,2){$(\Malg_\dagger/\Nalg_{n})^\dagger$}
\put(85,62){$0$}\put(85,42){$C$}
\put(77,22){$(\Malg_\dagger/\Nalg_{n+2})^\dagger$}
\put(77,2){$(\Malg_\dagger/\Nalg_{n+1})^\dagger$}
\put(30,35){$T_n$} \put(64,35){$T_{n+1}$}
\put(16,60){\vector(0,-1){12}} \put(16,40){\vector(0,-1){12}}
\put(16,20){\vector(0,-1){12}} \put(86,60){\vector(0,-1){12}}
\put(86,40){\vector(0,-1){12}} \put(86,20){\vector(0,-1){12}}
\put(20,44){\vector(1,0){63}} \put(35,36){\vector(1,0){27}}
\put(28,24){\vector(1,0){48}} \put(25,4){\vector(1,0){50}}
\put(29,34){\vector(-1,-1){8}} \put(72,34){\vector(1,-1){8}}
\put(29,37){\vector(-2,1){9}} \put(72,37){\vector(2,1){9}}
\put(50,45){\tiny $\operatorname{id}$}
\put(50,37){\tiny $\cb_l\cdot$}
\put(50,25){\tiny $\cb_l\cdot$}
\put(50,5){\tiny $\cb_l\cdot$}
\end{picture}
\end{proof}

Let us complete the proof of Theorem \ref{Thm:5I}.

\begin{proof}[Proof of Theorem \ref{Thm:5I}]
(1) and (3) follow from Proposition \ref{Prop:3.10.1}. (2) follows  from Propositions \ref{Prop:3.15.1}, \ref{Cor:3.15.2}. The claim that the image of $\bullet_\dagger$ is closed under
taking quotients follows from Theorem \ref{Thm:3.16.1}. Now (5) follows easily from (2)
and the exactness of $\bullet_\dagger$.

The proof of (4) is completely analogous to the proof of assertion (4) of Theorem 1.3.1
in \cite{HC}.
\end{proof}

\begin{Rem}\label{Rem:3.16.2}
We preserve the notation of Remarks \ref{Rem:3.10.2}, \ref{Rem:3.15.2}.
Again, all constructions of this subsection  apply to the functors $\bullet_\dagger: \HC(\tilde{\Zalg}_{(l)})
\rightarrow \HC^\Xi(\underline{\tilde{\Zalg}}_{(l)}^+), \bullet^\dagger: \HC^\Xi(\underline{\tilde{\Zalg}}^+_{(l)})\rightarrow
\widehat{\HC}(\tilde{\Zalg}_{(l)})$. The straightforward analogs of Theorems \ref{Thm:3.16.1}, \ref{Thm:5I}
hold for these functors.
\end{Rem}

\subsection{Proofs of Theorems  \ref{Thm:4},\ref{Thm:4'},\ref{Thm:5}}\label{SUBSECTION_Thm5_proof}
We start by constructing functors required in Theorem \ref{Thm:5}.

Let $\M\in _{\,c\!}\HC(\Hrm)_{c'}$. Pick a filtration $\F_i\M$ as in Definition \ref{defi_HC2}
and set $\Malg:=R_\hbar(\M)$. Then set $\M_\dagger:=\Malg_\dagger/(\hbar-1)\Malg_{\dagger}$. It is easy
to see that $\M_\dagger\in _{\,c\!}\HC^\Xi(\underline{\Hrm}^+)_{c'}$. As in Subsection 3.4 of \cite{HC}, we see that $\M_\dagger$ does not depend on the choice
of the filtration $\F_i\M$ (or more precisely, the objects $\M_\dagger^{\F},\M_\dagger^{\F'}$
constructed from different filtrations $\F,\F'$ are isomorphic via a distinguished isomorphism)
so the assignment $\M\mapsto \M_\dagger$ does define a functor $_{\,c\!}\HC(\Hrm)_{c'}\rightarrow
_{\,c\!}\HC^\Xi(\underline{\Hrm}^+)_{c'}$. We remark that $\bullet_\dagger$ is a tensor
functor, this is proved completely analogously to assertion (1) of Proposition 3.4.1 in \cite{HC}.

Similarly, for $\Ncal\in _{\,c\!}\HC_0^\Xi(\underline{\Hrm}^+)_{c'}$ one can define
$\Ncal^\dagger\in _{\,c\!}\HC_{\overline{\Leaf}}(\Hrm)_{c'}$.

Now Theorem \ref{Thm:5}  can be deduced  from Theorem \ref{Thm:5I} in a straight-forward way.

To prove Theorems \ref{Thm:4},\ref{Thm:4'} we will need their analogs
for the algebras $\tilde{\Halg},\underline{\tilde{\Halg}}^+$ and $\tilde{\Zalg}_{(l)},\underline{\tilde{\Zalg}}^+_{(l)}$.
For a  Poisson
two-sided ideal  $\Ialg\subset\underline{\tilde{\Halg}}^+$ set
 $\Ialg^{\ddag}:=(\bigcap_{\gamma\in \Xi} \gamma.\Ialg)^{\dagger_{\tilde{\Halg}}}$.
 In other words, $\Ialg^{\ddag}$ is the kernel of the map $\tilde{\Halg}\rightarrow
 (\underline{\tilde{\Halg}}^+/\bigcap_{\gamma\in \Xi} \gamma.\Ialg)^\dagger$.
Let $\Id(\tilde{\Halg})$ denote the set of graded Poisson two-sided ideals of $\tilde{\Halg}$.
Let $\Id_\Leaf(\tilde{\Halg})$ stand for the subset of $\Id(\tilde{\Halg})$ consisting of all $\Jalg\in \Id(\tilde{\Halg})$
with $\VA(\tilde{\Halg}/\Jalg)=\overline{\Leaf}$. Define the sets $\Id(\underline{\tilde{\Halg}}^+),\Id_0(\underline{\tilde{\Halg}}^+)$
in a similar way. Finally, let $\Id_0^\Xi(\underline{\tilde{\Halg}}^+)$ denote the set of
$\Xi$-invariants in $\Id_0(\underline{\tilde{\Halg}}^+)$. It is clear that
$\Jalg_\dagger\in \Id^\Xi_0(\underline{\tilde{\Halg}}^+)$ provided
$\Jalg\in \Id_{\Leaf}(\tilde{\Halg})$ and that $\Ialg^{\ddag}\in \Id_{\Leaf}(\tilde{\Halg})$
whenever $\Ialg\in\Id_0(\underline{\tilde{\Halg}}^+)$.

\begin{Thm}\label{Thm:4I}
The maps introduced above enjoy the following properties.
\begin{enumerate}
\item $\Jalg\cap S(\tilde{\param})\subset \Jalg_\dagger\cap S(\tilde{\param})$ and
$\Ialg\cap S(\tilde{\param})\subset \Ialg^\ddag\cap S(\tilde{\param})$.
\item $\Jalg\subset (\Jalg_\dagger)^{\ddag}, \Ialg\supset (\Ialg^\ddag)_{\dagger}$ for
any $\Jalg\in \Id(\tilde{\Halg}), \Ialg\in \Id^\Xi(\underline{\tilde{\Halg}}^+)$.
\item $\{\Jalg_\dagger: \Jalg\subset \Id_{\Leaf}(\Halg)\}=\Id_0^{\Xi}(\underline{\tilde{\Halg}}^+)$.
\item We have $\VA((\Jalg_\dagger)^{\ddag}/\Jalg)\subset \partial\Leaf$.
\item Consider the restriction of the map $\Ialg\mapsto \Ialg^{\ddag}$ to the set of all maximal
ideals in $\Id_{0}^\Xi(\underline{\tilde{\Halg}}^+)$. The image of this restriction is the set of all prime
ideals $\Jalg\in\Id_{\Leaf}(\tilde{\Halg})$ such that $S(\tilde{\param})\cap\Jalg$ is a maximal
homogeneous ideal in $S(\tilde{\param})$.\footnote{Of course, $S(\tilde{\param})$ has the augmentation ideal $(\tilde{\param})$ that contains any
other homogeneous ideal. When we speak about maximal homogeneous ideals we mean homogeneous ideals
that are strictly contained in $(\tilde{\param})$ and are maximal among such.} Further, each fiber of the restriction is a single $\Xi$-orbit.
\end{enumerate}
\end{Thm}

Recall that a two-sided ideal $I$ in an associative algebra $A$
is called prime if $aAb\subset I$ implies $a\in I$ or $b\in I$.

\begin{proof}[Proof of Theorem \ref{Thm:4}]
Assertion (1) follows easily from the fact that the functors $\bullet_\dagger$
and $\bullet^\dagger$ are $S(\tilde{\param})$-bilinear.
Assertion (2) follows from Lemma \ref{Lem:3.16.1}.
Assertion (3) follows from assertion (3) of Theorem \ref{Thm:5I} applied to
$\Ialg/(\Ialg^{\ddag})_\dagger\subset (\tilde{\Halg}/\Ialg^{\ddag})_\dagger$. In fact,
Theorem \ref{Thm:3.16.1} implies that  $(\Ialg^{\ddag})_\dagger=\Ialg$ provided $\Ialg\in \Id_0^\Xi(\tilde{\Halg})$.
Assertion (4) follows from assertion (2) of Theorem \ref{Thm:5I}.

Let us prove assertion (5). Our argument closely follows the lines of the proof of the similar statements for W-algebras,
see \cite{Wquant}, the proof of assertion (viii) of Theorem 1.3.1, and \cite{HC}, the proof of
Conjecture 1.2.1.  The proof is in several steps.

 Let $\Id(\tilde{\Halg})_{cc},\Id(\underline{\tilde{\Halg}}^+)_{cc}$
denote the subsets in $\Id(\tilde{\Halg}),\Id(\underline{\tilde{\Halg}}^+)$
consisting of all ideals whose intersections with $S(\tilde{\param})$
are maximal homogeneous.

{\it Step 1.} We claim that for any prime Poisson ideal $\Jalg\in \Id(\tilde{\Halg})_{cc}$ the associated variety
$\VA(\tilde{\Halg}/\Jalg)$ is irreducible.
Our claim can be easily  deduced from Ginzburg's results, \cite{Ginzburg_irr},
but we will present a proof based on our techniques. Let $\Leaf_1$ be a symplectic
leaf of maximal dimension in $\VA(\tilde{\Halg}/\Jalg)$. Then we have an inclusion $\Jalg\subset \Jalg_1:=(\Jalg_\dagger)^{\ddag}$, where $\bullet_\dagger, \bullet^{\ddag}$ are taken with respect to $\Leaf_1$. Since the Gelfand-Kirillov dimensions of $\tilde{\Halg}/\Jalg$ and $\tilde{\Halg}/\Jalg_1$ coincide (and coincide with $\dim \Leaf_1+1$), Corollar 3.6. from \cite{BoKr} implies $\Jalg=\Jalg_1$. This argument also implies the
surjectivity statement in (5).


{\it Step 2.}
Now we claim that an element $\Jalg\in \Id_\Leaf(\tilde{\Halg})_{cc}$ is a prime ideal
if and only if $\Jalg$ is maximal in $\Id_\Leaf(\tilde{\Halg})_{cc}$ with respect to inclusion.
Indeed, a non-maximal element cannot
be prime thanks to Corollar 3.6. from \cite{BoKr}. Now let $\Jalg$ be a maximal element in
$\Id_\Leaf(\tilde{\Halg})_{cc}$. Consider the radical $\sqrt{\Jalg}$.  For any $z\in \Zalg$ the map
 $b\mapsto \{z,b\}$ is a  derivation of $\tilde{\Halg}$. If an ideal is stable with respect
  to some derivation $d$, then its radical is also $d$-stable, see \cite{Dixmier}, Lemma 3.3.3.
So we see that $\sqrt{\Jalg}$ is a Poisson graded ideal and hence  $\sqrt{\Jalg}\in \Id_{\Leaf}(\tilde{\Halg})_{cc}$.
It follows that
$\Jalg=\sqrt{\Jalg}$. Therefore we have the prime decomposition $\Jalg=\bigcap_{i=1}^n \Jalg_i$.
All $\Jalg_i$ are again Poisson and graded, this is proved as above using \cite{Dixmier}, Lemma 3.3.3.
So $\VA(\tilde{\Halg}/\Jalg_i)$ is the closure of a  symplectic leaf, thanks to Step 1. It follows that
some $\Jalg_i$ lies in $\Id_\Leaf(\tilde{\Halg})_{cc}$ and hence $\Jalg=\Jalg_i$.

{\it Step 3.} To complete the proof it remains to show that for a given maximal (=prime) ideal $\Jalg\in
\Id_\Leaf(\tilde{\Halg})_{cc}$ the set $A$ of the maximal ideals $\Ialg\in \Id_{0}(\underline{\tilde{\Halg}}^+)_{cc}$
with $\Ialg^\ddag=\Jalg$ is a single $\Xi$-orbit (it is clear from the definition that this
set is $\Xi$-stable). Pick $\Ialg\in A$. Then, as we have noticed in the proof of (3), $\Jalg_\dagger=(\Ialg^{\ddag})_\dagger$. The r.h.s. coincides with $\bigcap_{\gamma\in \Xi}\gamma.\Ialg$. It follows that
$\bigcap_{\gamma\in \Xi} \gamma.\Ialg_1=\bigcap_{\gamma\in \Xi}\gamma.\Ialg_2$ for any $\Ialg_1,\Ialg_2\in A$.
So $\Ialg_1,\Ialg_2$ are conjugate under the $\Xi$-action.
\end{proof}

\begin{proof}[Proof of Theorem \ref{Thm:4}]
Embed $\Id(\Hrm_{1,c}),\Id(\underline{\Hrm}_{1,c}^+)$ into $\Id(\tilde{\Halg}),\Id(\underline{\tilde{\Halg}}^+)$,
respectively, by sending, say, $\J\in \Id(\Hrm_{1,c})$ to $R_\hbar(\J)$.
Assertion (1) of Theorem \ref{Thm:4I} implies that the maps in Theorem \ref{Thm:4I}
restrict to maps between $\Id_\Leaf(\Hrm_{1,c}),\Id_0(\underline{\Hrm}_{1,c}^+)$.
The proofs of assertions (1),(2),(3) of Theorem \ref{Thm:4} now follow from the corresponding
assertions of Theorem \ref{Thm:4I}. It is easy to see that $\J\in \Id(\Hrm_{1,c})$ is prime provided $R_\hbar(\J)$
  is. Assertion (5) of Theorem \ref{Thm:4I} implies the analog of assertion (4) of Theorem \ref{Thm:4}, where
``primitive'' is replaced with ``prime''.
Now  it remains to use the fact that any prime two-sided ideal in $\Hrm_{1,c}$ is primitive,
see Theorem A1 in \cite{ES_appendix}.
\end{proof}

\begin{proof}[Proof of Theorem \ref{Thm:4'}]
Choose the minimal number $l$ such that $c_l\neq 0$. We may assume that $c_l=1$. Then $R_{\hbar_l}(\Zrm_{c})$ is the quotient of $\tilde{\Zalg}_{(l)}$ by the ideal generated by $\cb_i-\hbar_l^2 c_i,i=l,\ldots,r$.

Thanks to Remarks \ref{Rem:3.10.2}, \ref{Rem:3.15.2}, \ref{Rem:3.16.2} we can follow the proof
of Theorem \ref{Thm:4I} to see that there are maps
$\bullet_\dagger: \Id_{\Leaf}(\Zrm_c)\rightarrow \Id_0^\Xi(\underline{\Zrm}_{c}^+),
\bullet^\ddag: \Id_0(\underline{\Zrm}_{c}^+)\rightarrow \Id_{\Leaf}(\Zrm_{c})$ enjoying the following properties.
\begin{enumerate}
\item The image of  $\J\mapsto \J_\dagger$ coincides with $\Id_{0}^{\Xi}(\underline{\Zrm}^+_{c})$.
\item We have $\VA((\J_\dagger)^{\ddag}/\J)\subset \partial\Leaf$.
\item Consider the restriction of the map $\I\mapsto \I^{\ddag}$ to the set of all maximal
ideals in $\Id_{0}(\underline{\Zrm}^+_{c})$. The image of this restriction is the set of all prime
ideals $\J\in\Id_{\Leaf}(\Zrm_{c})$. Further, each fiber of the restriction is a single $\Xi$-orbit.
\end{enumerate}
Now the proof of Theorem \ref{Thm:4'} basically repeats that of Theorem \ref{Thm:4}.
\end{proof}
\begin{Rem}
As in the proof of Theorem \ref{Thm:4I}, $\VA(\J)$ is irreducible for
any prime ideal $\J\in \Id(\Zrm_c)$. Alternatively, this follows from Martino's
results, \cite{Martino}.
\end{Rem}

\subsection{Definitions of functors using $\Theta^b$}\label{SUBSECTION_Functors_b}
For applications in Section \ref{SECTION_Cherednik} we will need more simple minded versions of the functors
$\bullet_\dagger,\bullet^\dagger$. Namely, above we used the isomorphism $\theta:\tilde{\Halg}^{\wedge_\Leaf}|_{\Leaf}\xrightarrow{\sim} \CC(\underline{\Halg}^{\wedge_{\underline{\Leaf}}}|_{\underline{\Leaf}}^{tw})^\Xi$
to pass between the $\Halg$- and $\underline{\Halg}$-sides. Now we are going to use the isomorphism
$\tilde{\Theta}^b: \tilde{\Halg}^{\wedge_b}\xrightarrow{\sim} \CC(\underline{\tilde{\Halg}}^{\wedge_0})$,
where $b\in \Leaf$.
The last isomorphism can be obtained from $\tilde{\theta}$ by passing to completions, compare
with the discussion in the end of Subsection \ref{SUBSECTION_completion_I}.
The main disadvantage of using $\tilde{\Theta}^b$ is that we do not see the action of $\Xi$.
An advantage, however, is that the construction using $\tilde{\Theta}^b$ is easier to deal with.

We are going to construct  functors $\bullet_{\dagger,b}: \HC(\tilde{\Halg})\rightarrow \HC(\underline{\tilde{\Halg}}^+)$
and $\bullet^{\dagger,b}: \HC(\underline{\tilde{\Halg}})\rightarrow \widehat{\HC}(\tilde{\Halg}^+)$.

The functor $\bullet_{\dagger,b}$ is given by $\bullet_{\dagger,b}=\Fun_7\circ \Fun_6\circ \Fun_5'\circ
\tilde{\Theta}^b_*(\bullet^{\wedge_b})$. Here $\bullet^{\wedge_b}$ is the completion functor at $b$
mapping from $\HC(\tilde{\Halg})$ to $\HC(\tilde{\Halg}^{\wedge_b})$, where the latter
is defined analogously to $\HC(\tilde{\Halg}^{\wedge_\Leaf}|_{\Leaf})$ (the $\K^\times$-equivariance
condition is replaced with the equivariance with respect to the derivation $E$ induced
from the Euler derivation on $\tilde{\Halg}$). The functor
$\Fun_5': \HC(\CC(\underline{\tilde{\Halg}})^{\wedge_0})\rightarrow
\HC(\CC(\underline{\tilde{\Halg}}^+)^{\wedge_0})$ is constructed analogously to $\Fun_5$ (see Subsection \ref{SUBSECTION_fun_low_dag}) and is a category equivalence.

The functor $\bullet^{\dagger,b}$ is defined as $[(\tilde{\Theta}^{b}_*)^{-1}\circ\Fun_5'^{-1}\circ\Fun_6^{-1}\circ \Fun_{7}^{-1}(\bullet)]_{l.f.}$, where the subscript ``$l.f.$'' has the same meaning as in the beginning of  Subsection \ref{SUBSECTION_Fun_up_dag}.

The following lemma describes some properties of the functors $\bullet_{\dagger,b},\bullet^{\dagger,b}$.

\begin{Lem}\label{Lem:fun_b}
\begin{enumerate}
\item The bifunctors $\Hom(\bullet_{\dagger,b}, \diamond)$ and $\Hom(\bullet,\diamond^{\dagger,b})$ from
$\widehat{\HC}(\tilde{\Halg})^{op}\times \widehat{\HC}(\underline{\tilde{\Halg}}^+)$ to the category of vector spaces are isomorphic.
\item There is a functor isomorphism between $\bullet_{\dagger,b}$ and the composition of
$\bullet_\dagger$ with the forgetful functor $\HC^\Xi(\underline{\tilde{\Halg}}^+)\rightarrow \HC(\underline{\tilde{\Halg}}^+)$.
\item There is an embedding of $\bullet^{\dagger}$ (or, more precisely, of the composition
of $\bullet^{\dagger}$ with the corresponding
forgetful functor) into $\bullet^{\dagger,b}$.
\end{enumerate}
\end{Lem}
\begin{proof}
The first assertion is proved in the same way as assertion (2) of Proposition \ref{Prop:3.15.1}.

Let us prove the second assertion. For $\Malg\in \HC(\tilde{\Halg})$ the object $\Theta^{b}_*(\Malg^{\wedge_b})$ is naturally identified with the completion of $\Fun_{4,1}(\Malg)$ at a point $\underline{b}\in \underline{\Leaf}$ lying over $b$. So we have a natural map $\Fun_{5,1}(\Malg)\rightarrow \Fun_5'\circ \Theta^{b}_*(\Malg^{\wedge_b})$.
The proof that $\Fun_5$ is an equivalence given in Subsection \ref{SUBSECTION_fun_low_dag}
implies that the map under consideration is an isomorphism. Assertion (2) follows.

To prove the third assertion we notice that $(\tilde{\Theta}^{b}_*)^{-1}\circ\Fun_5'^{-1}\circ\Fun_6^{-1}\circ \Fun_{7}^{-1}(\Nalg)$ is naturally identified with the completion of $\Fun_{7,2}^{-1}(\Nalg)$ at $b$.
The only global section of $\Fun_{7,2}^{-1}(\Nalg)$ that vanishes at the formal neighborhood of $b$
is zero. This implies assertion (3).
\end{proof}

\begin{Rem}\label{Rem:fin_gen}
Proposition \ref{Cor:3.15.2} implies that $\Nalg^{\dagger}$ is a finitely generated bimodule
provided $\Nalg\in \HC^{\Xi}_0(\underline{\Halg}^+)$. In fact, $\Nalg^{\dagger,b}$ is finitely
generated for any finite dimensional $\Nalg$ as well. The proof is completely analogous to that given in Propositions
\ref{Prop:3.15.2},\ref{Cor:3.15.2}.
\end{Rem}

For an ideal $\I\subset \underline{\tilde{\Halg}}^+$ we can define the ideal $\I^{\ddag,b}\subset \tilde{\Halg}$ similarly to $\I^{\ddag}$, see the beginning of Subsection \ref{Proof 5I}.
The following statement follows from assertions (2),(3) of Lemma \ref{Lem:fun_b} and the constructions of
$\I^{\ddag,b},\I^{\ddag}$.

\begin{Cor}\label{Cor:fun_b}
$\I^{\ddag,b}=\I^{\ddag}$
 for any ideal $\I\subset \underline{\tilde{\Halg}}^+$.  \end{Cor}


\begin{Cor}\label{Cor:up_dag_primit}
If $\I\subset \underline{\Hrm}^+_{1,c}$ is a primitive ideal, then so is $\I^\ddag\subset \Hrm_{1,c}$.
\end{Cor}
\begin{proof}
It is enough to show that $\I^\dagger$ is prime, see the main theorem from \cite{ES_appendix}.
The proof basically repeats that of assertion (iv) of Theorem 1.2.2 in \cite{Wquant}.
We will provide the proof for reader's convenience.

First of all, we note that $\I^\ddag$ is prime provided $R_\hbar(\I^\ddag)=R_\hbar(\I)^\ddag$
is. The ideal in $R_\hbar(\Hrm_{1,c})^{\wedge_b}$ corresponding to $R_\hbar(\I)$ under the bijections
explained above is prime. So we need to check that $\J_\hbar:=R_\hbar(\Hrm_{1,c})\cap \mathcal{J}_\hbar'$ is prime for
any prime ideal $\J'_\hbar\subset R_\hbar(\Hrm_{1,c})^{\wedge_b}$. Assume the contrary: let $\J^1_\hbar,\J^2_\hbar$
be two-sided ideals with $\J^1_\hbar \J^2_\hbar\subset \J_\hbar$ and $\J_\hbar\subsetneq \J^1_\hbar,\J^2_\hbar$.
Then $(\J^{1}_\hbar)^{\wedge_b}(\J^{2}_\hbar)^{\wedge_b}\subset \J'_\hbar$ hence, say, $(\J^{1}_\hbar)^{\wedge_b}
\subset \J'_\hbar$.
Therefore $\J^1_\hbar\subset (\J^1)^{\wedge_b}_\hbar\cap R_\hbar(\Hrm_{1,c})\subset \J_\hbar$,
a contradiction. So $\J_\hbar$ is a prime ideal.
\end{proof}

\section{Applications to general SRA}\label{SECTION_Application}
\subsection{Proof of Theorem \ref{Thm:3}}
We begin the proof with the following algebro-geometric lemma that should be standard,
compare with \cite{BG}, the proof of Theorem 4.2. The proof was essentially suggested to me by I. Gordon.

\begin{Lem}\label{Lem:6.1.1}
Let $X_1,X_2$ be varieties, $Y_1,Y_2$ irreducible divisors of $X_1,X_2$ contained in
all irreducible components of $X_1,X_2$, respectively, and
$\A_i$ be an $\Str_{X_i}$-coherent sheaf of associative algebras on $X_i, i=1,2$. Suppose that:
\begin{enumerate}
\item $\A_2$ is a trivial sheaf of algebras on $X_2\setminus Y_2$.
\item There is an isomorphism $\vartheta:(X_1)^{\wedge}_{Y_1}\rightarrow (X_2)^\wedge_{Y_2}$ of the completions such that the sheaves of algebras $\vartheta^*(\iota_2^*(\A_2)),\iota_1^*(\A_1)$ are isomorphic.
Here $\iota_i$ is the canonical embedding $(X_i)^{\wedge}_{Y_i}\hookrightarrow X_i$.
\end{enumerate}
Then there is a neighborhood $U_1$ of $Y_1$ in $X_1$ such that the fiber of $\A_1$
at any point $u\in U_1\setminus Y_1$ is isomorphic as an algebra to a fiber of $\A_2$.
\end{Lem}
\begin{proof}
Shrinking $X_1,X_2$ if necessary we may (and will) assume that $\A_1$ is a free  sheaf
on $X_1\setminus Y_1$ of rank, say, $r$.  Also we can assume that $X_1$ is affine. Let $f_1$ denote a function in $\K[X_1]$ that vanishes on $Y_1$.

Consider the space
$M$ of all algebra structures on $\K^r$, i.e.,  maps $\K^{r}\otimes \K^{r}\rightarrow \K^{r}$.
Let $Z$ denote the locally closed subvariety in $M$ consisting of all algebra structures isomorphic
to a fiber of $\A_2$ (at any point of $X_2\setminus Y_2$). Let $I,\tilde{I}\subset \K[M]$ denote the ideals of the closure $\overline{Z}$ and of the boundary $\partial Z$. Let $z$ be a point in $Z$. It corresponds to
a choice of a basis in some fiber of $\A_2$.

Fix a section $s$ of $\A_1$. It defines a morphism $s:X_1\setminus Y_1\rightarrow M$.  We claim that $s^*(I)$ is zero and $s^*(\tilde{I})$ is not. The algebra $\K[X_1]_{f_1}$ maps to  $A:=(\K[X_1]^{\wedge}_{Y_1})_{f_1}$.
(2) can be interpreted as follows: there is an element $g\in \GL_r(A)$ such that $s^*(f)=f(gz)$ (the equality in
$A$).

First of all, let us show that $f(gz)=0$ for $f\in I$.
 We have $f(gz)=0$ (in $\K$) for all $f\in \GL_r(\K)$
and hence $f(gz)=0$ (in $\K[X]_{f_1}$) for all  $g$ with coefficients in
$\K[X]_{f_1}$ (not necessarily invertible).
Any $g\in \GL_r(A)$ can be written as $f_1^Ng_1$, where $g_1$ has entries in $\K[X_1]^\wedge_{Y_1}$.
Now any $g_1\in \GL_r(\K[X_1]^\wedge_{Y_1})$ is
the limit of a sequence of elements with entries in $\K[X]$. It follows that
$f(gz)=0$ for all $g\in \GL_r(A)$. So $s^*(f)=0$ and hence $s^*(I)=\{0\}$.

The proof that $s^*(\tilde{I})\neq \{0\}$ is similar.
\end{proof}

\begin{proof}[Proof of Theorem \ref{Thm:3}]
We preserve the notation used in the proof of Theorem \ref{Thm:4'}.
Let $\Leaf$ be such that  $\q\in \Id_\Leaf(\Zrm_c)$.

Set $\tilde{X}_1=\Spec(R_{\hbar_l}(\Zrm_c)),
\tilde{X}_2=\Spec(R_{\hbar_l}(\underline{\Zrm}_c)^\Xi)=V_0^*\times \Spec(R_{\hbar_l}(\underline{\Zrm}_c^+)^\Xi)$.
Further, let $\tilde{\A}_1,\tilde{\A}_2$ be the sheaves on $\tilde{X}_1,\tilde{X}_2$ corresponding
to $R_{\hbar_l}(\Hrm_{0,c}), \CC(R_{\hbar_l}(\underline{\Hrm}_{0,c}))^\Xi$. The latter sheaf is just the product
of $\Str_{V_0^*}$ and the sheaf on  $\Spec(R_{\hbar_l}(\underline{\Zrm}_c^+)^\Xi)$ corresponding to
$\CC(R_{\hbar_l}(\underline{\Hrm}^+_{0,c}))^\Xi$.  Theorem \ref{Thm:2.0I} implies that there are open subsets
$\tilde{X}^0_1, \tilde{X}^0_2$ of $\tilde{X}_1,\tilde{X}_2$ such that there is an isomorphism
$\vartheta:(\tilde{X}^0_1)^\wedge_{\Leaf}\rightarrow (\tilde{X}^0_2)^\wedge_{\Leaf}$ satisfying
$\vartheta^*(\iota_2^*(\A_2))\cong \iota_1^*(\A_1)$. Here $\iota_1,\iota_2$ have  similar meanings to
Lemma \ref{Lem:6.1.1} and $(\tilde{X}^0_i)^\wedge_{\Leaf}$ stand for the formal neighborhood of $\Leaf\cap \tilde{X}^0_i$ in $\tilde{X}^0_i$  (we assume that $\tilde{X}^0_i\cap \overline{\Leaf}\subset \Leaf$).

Now let us introduce $X_1,X_2,Y_1,Y_2,\A_1,\A_2$ for which we are going to apply Lemma \ref{Lem:6.1.1}.
Let $X_1$  be the intersection of $\tilde{X}^0_1$ with the set of zeros of $R_{\hbar_l}(\q)$.
Set $X_2:=(V_0^*\times X_2')/\Xi$, where $X_2'$ is the set of zeros of $R_{\hbar_l}(\q_\dagger)$.
For $Y_i$ we take the divisor of zeros
of $\hbar_l$ in $X_i$. These divisors are isomorphic to open subsets in $\Leaf$ hence are irreducible.
Finally, let $\A_i, i=1,2,$ be the restrictions of $\tilde{\A}_i$ to $X_i$.

The claim that $\vartheta$ restricts to an isomorphism $(X_1)^\wedge_{Y_1}\rightarrow (X_2)^\wedge_{Y_2}$
follows directly from the construction of $R_{\hbar_l}(\q)_\dagger$.

So we see that for a general point $y$ of $\Leaf$ we have   $$\Hrm_{0,c}/\Hrm_{0,c}\n_y\cong \CC(\underline{\Hrm}^+_{0,c}/\underline{\Hrm}^+_{0,c}\underline{\n})$$ (the r.h.s. is a typical fiber of
$\A_2$). Thanks to \cite{BG}, Theorem 4.2, the algebras in the l.h.s. are pairwise isomorphic for all $y\in \hat{\Leaf}$.
 This proves Theorem \ref{Thm:3} for any $y\in \hat{\Leaf}$.
\end{proof}

\subsection{Simplicity of $\Hrm_{1,c}$ for general $c$}\label{SUBSECTION_simplicity}
The following theorem and its proof (modulo Theorem \ref{Thm:1}) were communicated to me by P. Etingof. This was the first motivation to state Theorem \ref{Thm:1}.

Let $\overline{\Q}$ denote the field of algebraic numbers.

\begin{Thm}
There is a finitely generated subgroup  $\Lambda\subset \overline{\Q}^r$ such that
the algebra $\Hrm_{1,c}$ is simple whenever $\sum_{i=1}^r \lambda_i c_i\not\in \mathbb{Z}$ for all $(\lambda_i)_{i=1}^r\in \Lambda\setminus\{0\}$.
\end{Thm}
\begin{proof}
{\it Step 1.} First of all, let us study the set of all parameters $c$ such that $\Hrm_{1,c}$
has a finite dimensional representation. We claim that there is a lattice
$\Lambda^0\subset \overline{\Q}^r$ such that
the algebra $\Hrm_{1,c}$ has no finite dimensional representations whenever
$\sum_{i=1}^r \lambda_i c_i\not\in \mathbb{Z}$ for all $ (\lambda_i)_{i=1}^r\in \Lambda^0$.

Let us reduce the proof to the case when $\Gamma$ is generated by symplectic reflections and
is symplectically irreducible. The reduction is standard but we provide it for reader's convenience.

First, we claim that  $\Hrm_{1,c}(V,\Gamma)$ has a finite dimensional representation if and only if $\Hrm_{1,c}(V,\Gamma')$
does, where $\Gamma'$ is the subgroup of $\Gamma$ generated by all symplectic reflections.  Indeed,
it is easy to see from (\ref{eq:1}) that $\Hrm_{1,c}(V,\Gamma)=\Hrm_{1,c}(V,\Gamma')\#_{\Gamma'}\Gamma$.
Here the meaning of the right hand side is as follows. Let $\A$ be an associative algebra containing
$\K \Gamma'$ and acted on by $\Gamma$ in such a way that the restriction of the $\Gamma$-action to
$\Gamma'$ coincides with the adjoint action of $\Gamma'\subset \A$. Then as a vector space $\A\#_{\Gamma'}\Gamma$
is just $\A\otimes_{\K\Gamma'}\K\Gamma$ and the product is defined in the same way as for $\A\#\Gamma$.
Now if $V$ is a non-zero left $\A$-module, then $\K\Gamma\otimes_{\K\Gamma'}V$ is a non-zero $\A\#_{\Gamma'}\Gamma$-module. Since $\A$ is included into $\A\#_{\Gamma'}\Gamma$, this shows the equivalence
stated in the beginning of the paragraph.


Let us now explain the reduction to the case when the $\Gamma$-module $V$ is symplectically irreducible.
The $\Gamma$-module $V$ is uniquely decomposed into the direct sum of symplectically irreducible $\Gamma$-submodules
$V=\bigoplus_{i=1}^l V_i$ ($V_i$ is said to be symplectically irreducible if it cannot be decomposed into
the direct sum of two proper {\it symplectic} submodules). Moreover, since $\Gamma$ is generated by symplectic reflections, it  decomposes into the
direct product $\prod_{i=1}^l \Gamma_i$, where $\Gamma_i\subset \Sp(V_i)$ is generated by symplectic reflections.
In particular, $S=\coprod_{i=1}^l (S\cap \Gamma_i)$.
As a consequence of this discussion, we have a tensor product decomposition
$\Hrm_{1,c}(V,\Gamma)=\bigotimes_{i=1}^l \Hrm_{1,c^i}(V_i,\Gamma_i)$, where
$c^i$ is the restriction of $c$ to $S\cap\Gamma_i$. Clearly, $\Hrm_{1,c}$ has a finite dimensional
module if and only if every $\Hrm_{1,c^i}(V_i,\Gamma_i)$ does. So it is enough to assume that
$V$ is symplectically irreducible.

Let $M$ be  a finite dimensional $H_{1,c}(V,\Gamma)$-module.
Then the trace of $[x,y]$ on $M$ is 0 for any $x,y\in V\subset H_{1,c}(V,\Gamma)$.
Therefore (\ref{eq:1}) implies that
\begin{equation}\label{eq:7.1}
(\dim M)\omega(x,y)+ \sum_{i=1}^r n_i(M) c_i\omega_i(x,y)=0,
\end{equation}
where $n_i(M)$ denotes the trace of $s\in S_i$ in $M$, and $\omega_i:=\sum_{s\in S_i} \omega_s$.
We remark that $\omega_i$ is a $\Gamma$-invariant skew-symmetric form.
Therefore $\omega_i$ is proportional to $\omega$, say $\omega_i=m_i\omega$. So we can rewrite (\ref{eq:7.1})
as
\begin{equation}\label{eq:7.2}
\dim M+\sum_{i=1}^r n_i(M)m_i c_i=0.
\end{equation}

Consider the subgroup $\Lambda^0\subset \overline{\Q}^{r}$ generated by $(m_i n_i(M'))_{i=1}^r$, where
$M'$ runs over all irreducible $\Gamma$-modules. (\ref{eq:7.2}) implies that $\Lambda^0$ satisfies
the condition in the beginning of the step.

{\it Step 2.} Now let $\Leaf$ be a symplectic leaf of $V^*/\Gamma$ and $\underline{\Gamma}\subset\Gamma$ be
such as above. Let $\param^{\Leaf}$ denote the
quotient of $\param_{(1)}$ by the space of all $\cb_i$ with $S_i\cap \underline{\Gamma}= \varnothing$. The space
$(\param^{\Leaf})^*$ is a subspace of $\param^{*}_{(1)}$ spanned by some basis vectors. So the form
$(\param^{\Leaf})^*(\overline{\Q})$ makes sense and it is a subspace in $\param_{(1)}^{*}(\overline{\Q})=\overline{\Q}^r$.
Consider the algebra $\underline{\Hrm}^+_{1,c}$ constructed from $\underline{\Leaf}$.
According to step 1, there is a finitely generated subgroup
$\Lambda^{\Leaf}\subset (\param^{\Leaf})^*(\overline{\Q})$ such that
$\underline{\Hrm}_{1,c}$ has no finite dimensional representations provided $\sum_{i=1}^r c_i\lambda_i\not\in \mathbb{Z}$ for every $(\lambda_i)_{i=1}^r \in \Lambda^{\Leaf}$. Let $\Lambda$ be the subgroup in
$\overline{\Q}^r$ generated by  $\Lambda^{\Leaf}$ for all symplectic leaves $\Leaf\subset V^*/\Gamma$.
Let us prove that this $\Lambda$ satisfies  the conditions of the theorem.

Let $\J\subset \Hrm_{1,c}$ be a two-sided ideal. Pick a symplectic leaf $\Leaf$ such that $\overline{\Leaf}$ is an irreducible
component of $\VA(\Hrm_{1,c}/\J)$. Then $\J_\dagger\subset \underline{\Hrm}_{1,c}$ is of finite codimension.
It follows that $\underline{\Hrm}_{1,c}$ has a finite dimensional module. So some linear combination
like in the statement of the theorem should be integral.
\end{proof}

\section{Rational Cherednik algebras}\label{SECTION_Cherednik}
\subsection{Content}
In this section we consider the case of rational Cherednik algebras in more detail.
Recall that a rational Cherednik algebra is a special case of an SRA corresponding to a special
pair $(V,\Gamma)$. Namely, let $\h$ be a complex vector space and $W$ be a finite subgroup in
$\GL(\h)$. Then we set $V:=\h\oplus \h^*$ and take the image of $W$ under a natural embedding
$\GL(\h)\hookrightarrow \Sp(V)$ for $\Gamma$. The corresponding SRA is given by the following relations,
see \cite{EG}.

\begin{equation}\label{eq:Cherednik_relations}
\begin{split}
&[x,x']=[y,y']=0,\quad x,x'\in \h^*, y,y'\in \h.\\
&[y,x]=\tb\langle y,x\rangle-\sum_{s\in S} \cb(s)\langle x,\alpha_s^\vee\rangle\langle y,\alpha_s\rangle s.
\end{split}
\end{equation}

Here $\alpha_s\in \h^*, \alpha_s^\vee\in \h$ are elements that vanish on the fixed point hyperplanes $\h^{s}, (\h^*)^s$
and satisfy $\langle\alpha_s, \alpha_s^\vee\rangle=2$. We remark that the independent variables $\cb(s)$ here
are not exactly the same as in (\ref{eq:1}) -- to get that setting we need to multiply $\cb(s)$ from (\ref{eq:Cherednik_relations})
by $-1/2$. Since it is common to present the rational Cherednik algebras in the form (\ref{eq:Cherednik_relations}),
we will use this version of $\cb(s)$'s from now on.

Many questions considered in the present paper in the case
of Cherednik algebras were studied previously. One goal of this section is to compare our results
and constructions with the existing ones. Another goal is to strengthen some of our results
in this special case.

There is a version of the isomorphism of completions theorem for rational Cherednik algebras
due to Bezrukavnikov and Etingof, \cite{BE}. In the first subsection we will compare their result
with ours. In Subsection \ref{SUBSECTION_assoc_var} we will use the explicit form of the completions isomorphism
obtained in \cite{BE} to relate the associated varieties of the ideals $\I,\I^{\ddag}$.

In Subsection \ref{SUBSECTION_RCA_HC} we will show that our notion of Harish-Chandra bimodules agrees
with one from \cite{BEG_HC}. Subsection \ref{SUBSECTION_functors_more}
contains an alternative definitions of functors $\bullet_{\dagger,b},\bullet^{\dagger,b}$
that is based on the coincidence of the definitions.
In \cite{BE} the isomorphism of completions was used to construct certain {\it induction}
and {\it restriction} functors.  In Subsection \ref{SUBSECTION_ResInd}
we will apply the description of Subsection \ref{SUBSECTION_functors_more}
 to  establish a relationship between those functors and the $\bullet_\dagger,\bullet^{\ddag}$-maps
between the sets of ideals.
We also apply that description in Subsection \ref{SUBSECTION_loc_fin} to study the Harish-Chandra
part of the space of maps between two modules in the category $\mathcal{O}$.

Finally, in Subsection \ref{SUBSECTION_type_A} we describe the two-sided ideals in the
rational Cherednik algebra of type $A$, i.e., corresponding to the symmetric group $S_n$.

Below we will write $W$ for $\Gamma$, and $\underline{W}$ for $\underline{\Gamma}$.

\subsection{Bezrukavnikov-Etingof's isomorphism of completions}\label{SUBSECTION_BE_completions}
As before take the symplectic leaf $\Leaf\subset V^*/W$ corresponding to a subgroup
$\underline{W}\subset W$. Pick a point $b\in \h$ with $W_b=\underline{W}$. In \cite{BE} Etingof and Bezrukavnikov
discovered an isomorphism $\vartheta^b:\Halg^{\wedge_b}\rightarrow \CC(\underline{\Halg}^{\wedge_0})$.
The homomorphism $\vartheta^b$ is the only continuous homomorphism satisfying
\begin{equation}\label{eq:def_iso1}
\begin{split}
& [\vartheta^b(u)f](w)=f(wu),\\
& [\vartheta^b(x_\alpha)f](w)= (\underline{x}_{w\alpha}+\langle b,w\alpha\rangle) f(w),\\
& [\vartheta^b(y_a)f](w)=\underline{y}_{wa} f(w)+\sum_{i=1}^r\sum_{s\in S_i\setminus \underline{W}}\frac{2\cb_i}{1-\lambda_s}\frac{\alpha_s(wa)}{\underline{x}_{\alpha_s}+\langle b,\alpha_s\rangle}(f(sw)-f(w)).\\
& u,w\in W, \alpha\in \h^*, a\in \h, f\in \Fu_{\underline{W}}(W,\underline{\Halg}^{\wedge_0}).
\end{split}
\end{equation}
Let us explain the notation used in the formula. By $x_\alpha,\underline{x}_\alpha$ we mean the images of
$\alpha\in \h^*$ under the embeddings $\h^*\hookrightarrow \Halg,\underline{\Halg}$.
By $\lambda_s$ we denote a unique non-unit eigenvalue of $s$ on $\h^*$.

The homomorphism $\vartheta^b$ is an isomorphism. Indeed, modulo $\param$ the homomorphism
$\vartheta^b$ coincides with the isomorphism induced by $\theta_0$ constructed in Subsection \ref{SUBSECTION_completion_I}. Since $\Halg^{\wedge_b}$
is complete in the $\param$-adic topology, $\vartheta^b$ is surjective, and since $\CC(\underline{\Halg}^{\wedge_0})$
is flat over $\K[[\param^*]]$, $\vartheta^b$ is injective.

We remark that we have a $(\K^\times)^2$-action on $\Halg$ given by $(t_1,t_2).x=t_1^2 x,
(t_1,t_2).y=t_2^2 y, (t_1,t_2).w=w,  (t_1,t_2).\cb_i=t_1^2t_2^2\cb_i$ for $x\in \h^*,y\in \h, w\in W$.
There is  an analogous action on $\underline{\Halg}$.
The isomorphism $\vartheta^b$ is equivariant with respect to the second copy of $\K^\times$.
This is checked directly from the formulas above. Also tracking the construction of the isomorphism
$\theta:\Halg^{\wedge_\Leaf}|_{\Leaf}\rightarrow \CC(\underline{\Halg}^{\wedge_{\underline{\Leaf}}}|^{tw}_{\underline{\Leaf}})^\Xi$, we see that it can be made $(\K^\times)^2$-equivariant. So the isomorphism
$\theta^b:\Halg^{\wedge_b}\rightarrow \CC(\underline{\Halg}^{\wedge_0})$ induced by $\theta$ is also equivariant with respect to the second
copy of $\K^\times$.

We do not know whether it is possible to find $\theta$ with $\theta^b=\vartheta^b$.
However, the following claim holds.

\begin{Lem}\label{Lem:isom_comp}
There is an element $f\in \param\otimes \K[\h/\underline{W}]^{\wedge_0}$ such that $\theta^b=\exp(\frac{1}{\tb}\operatorname{ad}(f))\vartheta^b$.
\end{Lem}
\begin{proof}
Consider the automorphism $\rho:=\theta^b\circ (\vartheta^{b})^{-1}$ of the algebra
$\CC(\underline{\Halg}^{\wedge_0})$. This automorphism is the identity modulo $\param$
and is $\K^\times\times W$-equivariant. So $\rho=\exp(d)$, where $d$ is a $\param$-linear $\K^\times\times W$-equivariant derivation of $\CC(\underline{\Halg}^{\wedge_0})$.
 As in the proof of Proposition \ref{Prop:2.3}, we see that $d=\frac{1}{\tb}\ad(f)$ for some $f\in \param\z^{\tb}(\CC(\underline{\Halg}^{\wedge_0}))$. We may assume that $\frac{1}{\tb}f$ is
$\K^\times\times W$-invariant. But the latter just means that $f\in \param\otimes \K[\h/\underline{W}]^{\wedge_0}$.
\end{proof}

In particular, we see that $\theta^b(x)=\vartheta^b(x)$ for $x\in \h^*$ (in fact, this follows directly
from the $\K^\times$-equivariance condition).

\subsection{Associated varieties and $\bullet^\ddag$}\label{SUBSECTION_assoc_var}
In this subsection we will use the results of the previous one to relate the associated varieties
of the ideals $\Ialg\subset \underline{\tilde{\Halg}}^+$ and $\Ialg^{\ddag}\subset \tilde{\Halg}$.

In the sequel we will need a remarkable {\it Euler} element in $\tilde{\Halg}$.
It is given by the following formula:
\begin{equation}\label{eq:Euler}
\mathbf{h}=\sum_{i=1}^n x_iy_i+\frac{\dim \h}{2}-\sum_{s\in S}\frac{2\cb(s)}{1-\lambda_s}s.
\end{equation}
Here $y_1,\ldots,y_n$ is a basis in $\h$, and $x_1,\ldots,x_n$ is a dual basis in $\h^*$.
A crucial property of $\mathbf{h}$ is that $[\mathbf{h},x_i]=x_i, [\mathbf{h},y_i]=-y_i$
for all $i$. We can define the Euler element $\underline{\mathbf{h}}^+\in \underline{\tilde{\Halg}}^+$
in a similar way.

\begin{Prop}\label{Prop:up_dag_assoc_var}
Let $\Ialg$ be a Poisson ideal in $\underline{\tilde{\Halg}}^+$. Suppose that $\VA(\underline{\tilde{\Halg}}^+/\Ialg)$ coincides with the closure of the symplectic leaf $\underline{\Leaf}_0\subset V_+^*/\underline{W}$ corresponding to a subgroup $W_0\subset \underline{W}$. Then $\VA(\tilde{\Halg}/\Ialg^\ddag)$ coincides with the closure of the leaf
$\Leaf_0\subset V^*/W$ corresponding to $W_0(\subset W)$.
\end{Prop}
\begin{proof}
Set $\h_+:=\h\cap V_+$. First, we are going to show that the associated variety of $\Ialg\cap \K[\h_+]$ coincides
with $\underline{W} \h_+^{W_0}$.

To this end let us show that the associated graded ideal $\gr\Ialg$ of $\Ialg$ with respect to the usual grading
coincides with the associated graded $\gr^y\Ialg$ for the grading induced by the $\K^\times$-action considered
in the previous subsection (we will call it the $y$-grading because the degree of $y\in \h_+$ is 2,
while the degree of $x\in \h_+^*$ is 0).

The $\K^\times$-action on $\underline{\tilde{\Halg}}^+$ given by $t.x=t x, t.y=t^{-1}y, t.w=w,t.\hbar=\hbar,
t.\cb_i=\cb_i, x\in \h^*_+, y\in\h_+$ is inner, the corresponding grading is the  inner grading  given by eigenvalues
of $\frac{1}{\tb}\ad(\underline{\mathbf{h}}_+)$.  Being Poisson (and
so $\frac{1}{\tb}\ad(\mathbf{h}_+)$-stable),
the ideal $\Ialg$ is graded with respect to the inner grading. The standard grading and the
$y$-grading differ by the inner grading. This implies the claim in the previous paragraph.

Now $\Ialg\cap \K[\h_+]^{\underline{W}}$ is the 0th component in $\gr^y (\underline{\tilde{\Zalg}}_+\cap\Ialg)$.
To establish the statement in the beginning of the proof it is enough to check that the associated variety of the $\K[\h_+]^{\underline{W}}\,$-module $\K[\h_+]/\K[\h_+]\cap \Ialg$ coincides with $\pi_{\underline{W}}(\h_+^{W_0})$. The latter is nothing else
but the projection of $\VA(\underline{\tilde{\Halg}}^+/\Ialg)=\pi_{\underline{W}}(V_+^{W_0})$ to $\h_+/\underline{W}$ under the natural morphism $V_+/\underline{W}\twoheadrightarrow \h_+/\underline{W}$. So it coincides with the variety of zeros of  $\Ialg\cap \K[\h_+]^{\underline{W}}$.
This implies the claim in the beginning of the proof.

Now let us show that the variety of zeros of $\Ialg^{\ddag}\cap \K[\h]^W$ is contained in
$\pi_W(\h^{W_0})$. Let $\W_\tb$ stand for the Weyl algebra of $V^{\underline{W}}$
and let $I$ be the ideal of $\pi_{\underline{W}}(\h^{W_0})\subset \K[\h]^{\underline{W}}$.
By the above, $(\W_{\tb}\otimes_{\K[\tb]}\Ialg)\cap \K[\h]^{\underline{W}}\supset I^n$ for some $n$.
By the construction of $\Ialg^{\ddag}$ using $\theta^b$ given in Subsection \ref{SUBSECTION_Functors_b}, we see that
$\Ialg^\ddag\supset (\theta^{b})^{-1}\left[(\W_{\tb}\otimes_{\K[\tb]}\Ialg)\cap \K[\h]^{\underline{W}}\right]$.
Now we can use Lemma \ref{Lem:isom_comp} and the explicit form of $\vartheta^b$ to see that
$\Ialg^{\ddag}\supset I^n\cap \K[\h]^{W}$. This implies the claim in the beginning of the paragraph.

So we see that $\VA(\tilde{\Halg}/\Ialg^{\ddag})\subset \pi_W(V^{W_0})$. Let us prove
the equality. The dimension of $\VA(\underline{\tilde{\Halg}}^+/(\Ialg^{\ddag})_{\dagger})$
equals $\dim \VA(\tilde{\Halg}/\Ialg^{\ddag})-\dim V^{\underline{W}}$ by assertion 4 of Proposition \ref{Prop:3.10.1}.
On the other hand, $\dim V^{W_0}-\dim V^{\underline{W}}=\dim V_+^{W_0}=\dim \VA(\underline{\tilde{\Halg}}^+/\Ialg)$.
It remains to notice that $(\Ialg^{\ddag})_{\dagger}\subset \Ialg$
and so $\dim \VA(\underline{\tilde{\Halg}}^+/(\Ialg^{\ddag})_{\dagger})\geqslant \dim \VA(\underline{\tilde{\Halg}}^+/\Ialg)$.
\end{proof}

\subsection{Equivalence of definitions of Harish-Chandra bimodules}\label{SUBSECTION_RCA_HC}
Let us recall the notion of a Harish-Chandra $\Hrm_{1,c}$-$\Hrm_{1,c'}$-bimodule from \cite{BEG_HC}.
Consider the subalgebras $(S\h)^W$, $(S\h^*)^W\subset \Hrm_{1,c},\Hrm_{1,c'}$. In \cite{BEG_HC} an $\Hrm_{1,c}$-$\Hrm_{1,c'}$-bimodule $\M$ is called Harish-Chandra if the operators $\ad(a),\ad(b)$ are
locally nilpotent on $\M$ for any $a\in (S\h^*)^W$ and $b\in (S\h^*)^W$.
The goal of this subsection is to show that this definition  is equivalent to
Definition \ref{defi_HC2} (we remark that no analog
of the Berest-Etingof-Ginzburg definition is known for general SRA).

The easier implication is that a HC bimodule in the sense of Definition \ref{defi_HC2}
is also HC in the sense of \cite{BEG_HC}. This is based on the following proposition.

\begin{Prop}\label{Prop:HC_comparison}
Let $\Malg\in \HC(\tilde{\Halg})$. Then the operators $\ad(a),\ad(b)$ are locally nilpotent for any
$a\in (S\h^*)^W,b\in (S\h)^W$.
\end{Prop}
\begin{proof}
Recall the Euler element $\mathbf{h}\in \tilde{\Halg}$ defined by (\ref{eq:Euler}). Hence $\frac{1}{\tb}[\mathbf{h},\cdot]$ is a grading preserving linear map $\Malg\rightarrow \Malg$. So we have the decomposition of the graded component $\Malg_i$ into the direct sum $\bigoplus_\alpha \Malg_{i,\alpha}$ of generalized eigenspaces of $\frac{1}{\tb}\ad(\mathbf{h})$.
For $\beta\in \K$ set $\Malg_{(\beta)}:=\bigoplus_i \Malg_{i,\beta+i}$. Then $x\Malg_{(\beta)}\subset \Malg_{(\beta)}$
for any $x\in \h^*$, while $y\Malg_{(\beta)}=\Malg_{(\beta-2)}$ for $y\in \h$. Since $\Malg$ is a finitely generated
graded $\Halg$-module, we see that there are finitely many elements $\beta_i\in \K$ with the property
that $\Malg_{(\beta)}=0$ whenever $\beta_i-\beta\not\in \mathbb{Z}_{\geqslant 0}$.

On the other hand,
for $a\in (S\h^*)^W$ we have $\ad(a)\Malg_{(\beta)}\subset\Malg_{(\beta)}$ and hence $\frac{1}{\tb}\ad(a)\Malg_{(\beta)}
\subset \Malg_{(\beta+2)}$. It follows that the operator $\frac{1}{\tb}\ad(a)$ is locally nilpotent.
The proof for $b\in (S\h)^W$ is similar.
\end{proof}

The following claim follows directly from Proposition \ref{Prop:HC_comparison} and the definition
of $_{\, c\!}\HC(\Hrm)_{c'}$.

\begin{Cor}\label{Cor:HC_comparison}
Any $\M\in _{\, c\!}\HC(\Hrm)_{c'}$ is a Harish-Chandra bimodule in the sense of \cite{BEG_HC}.
\end{Cor}


\begin{Prop}\label{Prop:HC_coinc}
Let a $\Hrm_{1,c}$-$\Hrm_{1,c'}$-bimodule $\M$ be HC in the sense of \cite{BEG_HC}. Then there exists a filtration
$\F_i\M$ satisfying the conditions of Definition \ref{defi_HC2}.
\end{Prop}
\begin{proof}
{\it Step 1.}
Before constructing a filtration on $\M$ let us introduce some notation.

Abusing the notation, by $\mathbf{h}$ we denote the images of $\mathbf{h}$ in $\Hrm_{1,c},\Hrm_{1,c'}$.

Pick free homogeneous generators $a_1,\ldots,a_n$ of $(S\h^*)^W$ and $b_1,\ldots, b_n$ of $(S\h)^W$.
Let $d_1,\ldots,d_n$ be their degrees. Further, let $\underline{c}_1,\ldots,\underline{c}_m$ be bi-homogeneous
generators of the $(S\h^*)^W\otimes (S\h)^W$-module $\K[\h\oplus\h^*]^W$. Let their degrees (of polynomials
on $\h\oplus\h^*$) be $d_1',\ldots,d_m'$. Lift $\underline{c}_1,\ldots,\underline{c}_m$ to $W$-invariant $\ad(\mathbf{h})$-eigenvectors in $\Hrm_{1,c},\Hrm_{1,c'}$. Abusing the notation, we denote both
liftings by $c_1,\ldots,c_m$. We remark that if we have another collection of liftings
$\tilde{c}_1,\ldots,\tilde{c}_m$, then $c_i-\tilde{c}_i\in \F_{d_i'-2}\Hrm_{\bullet}$.
Finally, let us choose  bi-homogeneous generators $\underline{e}_1,\ldots, \underline{e}_k$
of $\K[\h\oplus \h^*]\# W$ over $\K[\h\oplus\h^*]^W$ and lift them to  $\ad(\mathbf{h})$-eigenvectors
$e_1,\ldots,e_k$ of $\Hrm_{1,c},\Hrm_{1,c'}$. Let $d_1'',\ldots, d_k''$ be the degrees
of $\underline{e}_1,\ldots, \underline{e}_k$ viewed as polynomials.

On the next steps we  construct a filtration $\F_i\M$ on $\M$ such that
$\gr\M$ is a finitely generated left $SV\#\Gamma$-module and
\begin{align}\label{eq:filtr1}
&[z, v]\in \F_{i+j-2}\M,\\\label{eq:filtr2}
&vh\subset  \F_{i+j}\M,\\\label{eq:filtr3}
&h v\subset \F_{i+j}\M.
\end{align}
for any $v\in \F_j \M, h\in \F_i \Hrm_{1,\bullet}$ and any $z$ that
is an ordered monomial in $c_k,b_j,a_i$ of total degree $i$ (with the degrees
of $a_i,b_j,c_k$ as above).
This filtration has the properties required in
Definition \ref{defi_HC2}.

{\it Step 2.}
We will define a filtration on $\M$   using the inductive procedure
below.

Let $V_j$ be a finite dimensional subspace in $\M$, and $ A_i,i\in I,$ be linear operators
on $\M$.  Set $V_{j+1}:=\sum_{i\in I} A_i V_j$.
Suppose that $V_j$ is equipped with an increasing filtration $\F_m V_j$.
To each $A_i$ we assign an integral degree $n_i$.
Then we can define a filtration on $V_{j+1}$ by defining $\F_n V_{j+1}$ to be
the linear span of elements $ A_{i}v_0$, where $v_0\in \F_m V_j$
and $m+n_i\leqslant n$.

{\it Step 2.1.} Let $V_0$ be a  subspace
that generates $\M$ as a left $\Hrm_{1,c}$-module. Since $\M$ is a finitely generated
left $\Hrm_{1,c}$-module we can choose $V_0$ to be finite dimensional. For the operators
$A_i$ we take all monomials in $\ad(a_i), i=1,\ldots,n$.
For the degree of $\ad(a_i)$ we take $d_i-2$, and the degree of a monomial
is defined in an obvious way.
Since  $\ad(a_i)$ are pairwise commuting locally nilpotent operators on $\M$, we
see that $V_1$ is finite dimensional.

{\it Step 2.2.} To define $V_2$ we use $V_1$ and the operators that are monomials in $\ad(b_i)$,
whose degrees are defined analogously to Step 3.1. The space $V_2$ is finite dimensional
for the same reason as $V_1$.

{\it Step 2.3.} Now  consider the operators $1,\ad(c_i), i=1,\ldots,m$.
To $\ad(c_i)$ we assign degree $d_i'-2$, and $1$ is supposed to be
of degree 0. We use the filtered
subspace $V_2\subset \M$ and the operators $1,\ad(c_i)$ to define $V_3$,
which is obviously finite dimensional.

{\it Step 2.4.} We define $V_4$ using $V_3$ and the operators of the right
multiplication by $1,e_i$. The operator corresponding to $e_i$ has degree $d_i''$.

{\it Step 2.5.} Consider the basis of the form
$w x_1^{i_1}\ldots x_n^{i_n} y_1^{j_1}\ldots y_n^{j_n}$
in $\Hrm_{1,c}$, where $x_1,\ldots,x_n$ is a basis in $\h^*$, $y_1,\ldots, y_n$ is a basis in $\h$.
Define $V_5$ from $V_4$ using the operators of left multiplication by the basis elements.
The degree of an operator coming from the basis element above is $\sum_{l=1}^n (i_l+j_l)$.
Of course, $V_5=\M$.

Let $\F_\bullet^i$ be the filtration on $V_i, i=1,2,3,4,5$ defined on the corresponding step
above.  By our definition, $\F^{j}_i V_j\subset \F^{j+1}_i V_j$ for all $i,j$.

On the next step we are going to prove that the filtration $\F_\bullet:=\F^5_\bullet$ on $\M$
satisfies all conditions indicated in the end of Step 1.

{\it Step 3.} Consider a collection $C=(\mathbf{i}^1,w,\mathbf{i}^2,\alpha,\beta, \mathbf{j}^1,\mathbf{j}^2)$,
where $\mathbf{i}^1,\mathbf{i}^2,\mathbf{j}^1,\mathbf{j}^2$ are $n$-tuples of non-negative integers, $w\in W$,
$\alpha$ is either an element of $\{1,2,\ldots,m\}$ or the symbol $\emptyset$, and $\beta$ is either
an element of $\{1,2,\ldots,k\}$, or $\emptyset$. To $v_0\in V_0$ and $C$ as above we assign
the element $v(C,v_0)\in \M$ by $$v(C,v_0):=w x^{\mathbf{i}^1} y^{\mathbf{i}^2}[\ad(c)_\alpha\ad(b)^{\mathbf{j}^1}\ad(a)^{\mathbf{j}^1}v_0] e_\beta,$$
where, for instance $\ad(a)^{\mathbf{j}^2}$ means $\prod_{i=1}^n \ad(a_i)^{\mathbf{j}^2_i}$,
and $\ad(c)_\alpha$ is $1$ when $\alpha=\emptyset$, and $\ad(c_\alpha)$ otherwise.

Further, define the ``degree'' $$d(C):=\sum_{r=1}^n (i_r^1+i_r^2+ (j^1_r+j^2_r)(d_r-2))+ d'_\alpha-2+d''_\beta,$$
where $d'_\emptyset:=2, d''_\emptyset:=0$ so that $v(C,v_0)\in \F_{d(C)}\M$
(but it may happen that $v(C,v_0)\in \F_{d(C)-1}\M$). Finally, define the ``length''
$$l(C)=\sum_{r=1}^n (i^1_r+i^2_r+j^1_r+j^2_r)+l_\alpha+l_\beta+l_w,$$ where $l_\alpha=0, l_\beta=0, l_w=0$
when $\alpha=\emptyset, \beta=\emptyset, w=1$ and $l_\alpha=1,l_\beta=1,l_w=1$ otherwise.

By our choice of $V_0$, the elements $v(C,v_0)$ span $\M$ as a vector space.
Moreover, by our definition of the filtration on $\M$, the images of
$v(C,v_0)$ in $\F_{d(C)}\M/\F_{d(C)-1}\M$ span $\gr \M$. From  here it is easy to
see that $\gr\M$ is a finitely generated left $\K[\h\oplus\h^*]\# W$-module.
It remains to check (\ref{eq:filtr1})-(\ref{eq:filtr3}). It is enough to
do this for those $v(C,v_0)$ that do not lie in $\F_{d(C)-1}\M$.

First of all, let us notice that (\ref{eq:filtr3})
follows directly from the construction of the filtration.
We are going to prove (\ref{eq:filtr1}),(\ref{eq:filtr2})
by induction on $l(C)$.

We start  with $l(C)=0$ (so that $v(C,v_0)=v_0$).
Assume that (\ref{eq:filtr1}),(\ref{eq:filtr2})
are proved for all $z,h$ of degree less than $N$ and let us prove the inclusions
for degree $N$ (the base is $N=0$, here there is nothing to prove).
Each $h\in \Hrm_{1,\bullet}$ can be uniquely written as
the sum of ordered monomials in $c_\alpha,b_j,a_i,e_\beta$
(with $c_\alpha, e_\beta$ occurring at most once). For $h\in \F_i\Hrm_{1,\bullet}$
the degrees of monomials do not exceed $i$. Also if $h$ is an eigenvector
for $\ad(\mathbf{h})$ and its image in $\gr \Hrm_{1,\bullet}$ is in the center,
then we see that all monomials of degree exactly $i$ will be in $c_\alpha,b_j,a_i$
(or we can also write the variables in the  opposite order: $a_i,b_j,c_\alpha$),
and all other monomials will have degree $\leqslant i-2$.


To prove $[z,v]\in \F_{N-2}\M$ we may assume that $z$ is an ordered monomial in $a_i,b_j,c_\alpha$ (in this order).
If $z$ is a single variable, then we are done by the definition of the filtration.
If not, write $z$ as a product $z=z_1z_2$, where $z_1$ is a variable. Then rewrite $[z,v_0]=z_1[z_2,v_0]+[z_1,v_0]z_2=z_1[z_2,v_0]+z_2[z_1,v_0]-[z_2,[z_1,v_0]]$.
Then the first  two summands lie in $\F_{N-2}\M$ thanks to the inductive assumption and (\ref{eq:filtr3}).
To show that the third summand lies in $\F_{N-2}\M$ we apply the same procedure for  $z_2=z_2'z_2''$ and $[z_1,v_0]$ instead of $v_0$. We get $[z_2,[z_1,v_0]]=z_2'[z_2'',[z_1,v_0]]+z_2''[z_2',[z_1,v_0]]+[z_2'',[z_2',[z_1,v_0]]]$.
So we need to show that $\ad(z_2')\ad(z_1)v_0, \ad(z_2'')\ad(z_1)v_0, \ad(z_2'')\ad(z_2')\ad(z_1)v_0$
lie in the appropriate filtration components. For the first expression
this follows directly from the definition of the filtration.  We remark that the variables $z_1,z_2'$ still precede
all variables in $z_2''$ in the order prescribed in the beginning of the
paragraph. So we apply the same procedure as before to $z_2''$ and so on.
Finally, we arrive at the sum of  expressions of the form $z^0 \ad(z^k)\ldots \ad(z^1)v_0$, where
$z^k,\ldots,z^1$ are ordered correctly. Tracking the construction and using the definition
 of the filtration, we see that all these expressions
are in $\F_{N-2}\M$.

The inclusion (\ref{eq:filtr2})  for $v\in V_0$ is proved in a similar way.

Now suppose that (\ref{eq:filtr1}),(\ref{eq:filtr2}) hold for all $z,h$ and all
$v(C,v_0)$ with $l(C)<l$. Let us prove (\ref{eq:filtr2}) for $l(C)=l$,
the proof of (\ref{eq:filtr1}) is similar. If $\beta$ in $C$ is not $\emptyset$,
then $v(C,v_0)h=v(C',v_0)e_\beta h$ with $l(C')<l(C)$ and we are done by
the inductive assumption. Suppose now that $\beta=\emptyset$. Then  $v(C,v_0)=[z,v(C',v_0)]$
for appropriate $z$ and $l(C')<l(C)$. Then $v(C,v_0)h=z(v(C',v_0)h)-v(C',v_0)zh$.
By the inductive assumption, $v(C',v_0)h, v(C',v_0)zh$ lie in the appropriate filtration
component. Now the right hand side itself lies in the appropriate filtration component
by (\ref{eq:filtr3}).
\end{proof}

\subsection{Alternative definitions of $\bullet_{\dagger,b},\bullet^{\dagger,b}$}\label{SUBSECTION_functors_more}
Let $\M$ be a Harish-Chandra $\Hrm_{1,c}$-$\Hrm_{1,c'}$-bimodule. Since the adjoint action of $(S\h^*)^W=\K[\h/W]$
on $\M$ is locally nilpotent, the tensor product $\M^{\wedge_b}:=\K[\h/W]^{\wedge_b}\otimes_{\K[\h/W]}\M$
has a natural structure of a bimodule over $\Hrm_{1,c}^{\wedge_b}:=\K[\h/W]^{\wedge_b}\otimes_{\K[\h/W]}\Hrm_{1,c}$
and $\Hrm_{1,c'}^{\wedge_b}$. So $\theta^b_{*}(\M')$ is a bimodule over
$\CC(\underline{\Hrm}_{1,c}^{\wedge_0})$-$\CC(\underline{\Hrm}_{1,c'}^{\wedge_b})$. Consider the subspace
$\M_{\heartsuit,b}\subset e(\underline{W})\theta^b_{*}(\M^{\wedge_b})e(\underline{W})$ consisting of
all elements $v$ that commute with $\h^{\underline{W}},(\h^*)^{\underline{W}}$
and such that $\ad(a)^N v=\ad(b)^N v=0$ for all $a\in S(\h_+^*)^{\underline{W}},b\in S(\h_+)^{\underline{W}},N\gg 0$.

\begin{Prop}\label{Prop:fun_iso}
There is a functorial isomorphism $\M_{\dagger,b}\xrightarrow{\sim}\M_{\heartsuit,b}$.
\end{Prop}
\begin{proof}
First of all, let us define a derivation $D$ of $\M$ that is compatible with $\ad(\mathbf{h})$,
is diagonalizable, and has integral eigenvalues.

For this fix a lifting $\iota:\K/\mathbb{Z}\rightarrow \K$ of the natural projection
$\K\twoheadrightarrow \K/\mathbb{Z}$. The action of $\ad(\mathbf{h})$ (where recall we write
$\mathbf{h}$ for the images of $\mathbf{h}\in \Halg$ in $\Hrm_{1,c},\Hrm_{1,c'}$) on $\M$ is locally finite,
this follows from Definition \ref{defi_HC2}, because $\frac{1}{\tb}\ad(\mathbf{h})$
is a degree preserving map for any Harish-Chandra $\tilde{\Halg}$-bimodule. Define $D$
by making it act by $\lambda-\iota(\lambda)$ on the generalized eigenspace of $\ad(\mathbf{h})$
with eigenvalue $\lambda$. Let $\M(j)$ denote the $j$th eigenspace of $D$ in $\M$.

Now choose a filtration $\F_i\M$ on $\M$ as in Definition \ref{defi_HC2}. Consider the filtration
$\F^{D}_i\M$ defined by $\F^{D}_i\M:=\bigoplus_j \F_{i-j}\M(j)$. The filtration $\F^{D}_\bullet$ is compatible
with the doubled $y$-filtration on $\Hrm_{1,c},\Hrm_{1,c'}$. Moreover, the associated graded
$\gr^D \M$ is still a finitely generated $S(\h\oplus\h^*)\# W$-module. Also let us remark that the
Rees bimodules $R_\hbar^D(\M),R_{\hbar}(\M)$ are naturally identified (the gradings differ
by twisting with $D$). We write $\M_\hbar$ for both these bimodules. Unless  specified otherwise,
we consider the grading on $\M_\hbar$ coming from $R^D_\hbar(\M_\hbar)$.

The bimodule $\M^{\wedge_b}$ has the $y$-filtration inherited from $\M$. The corresponding Rees
bimodule $R_\hbar(\M^{\wedge_b})$ coincides with the subspace of $\K^\times$-finite vectors
in $\M_\hbar^{\wedge_b}$, where the latter completion is defined as in Subsection \ref{SUBSECTION_Functors_b}.

Let $\Ncal_\hbar'$ denote the centralizer of $(\h\oplus\h^*)^{\underline{W}}$ in $e(\underline{W})\theta^b_{*}(\M_\hbar^{\wedge_b})e(\underline{W})$. Consider
two subspaces $\Ncal_\hbar^{1},\Ncal_{\hbar}^{2}\subset \Ncal'_\hbar$ defined as follows. The subspace
$\Ncal_\hbar^{1}$ is defined as the subspace of all vectors that are locally finite for
the Euler derivation of $\Ncal_\hbar'$. The bimodule $\M_{\dagger,b}$ is just the quotient
of $\Ncal_\hbar'$ by $\hbar-1$. Further, the subspace $\Ncal_\hbar^2$ consists of all
vectors that are locally finite for the $\K^\times$-action and locally nilpotent for
the operators $\ad(a),\ad(b)$ for all $a\in S(\h_+^*)^{\underline{W}},
b\in S(\h_+)^{\underline{W}}$. It is easy to see that  $\M_{\heartsuit,b}$ is the quotient  of $\Ncal^2_\hbar$
by $\hbar-1$. So it remains to show that $\Ncal_\hbar^1=\Ncal_\hbar^2$.

Since $\Ncal_\hbar^1=R_\hbar(\M_\dagger)$, and the action of $\ad(\mathbf{h})$
on $\M_{\dagger,b}$ is locally finite, we see that $\K^\times$ acts
locally finitely on $\Ncal_\hbar^1$. Since, in addition, the actions of $\ad(a),\ad(b)$ on $\Ncal_\hbar^1$
are locally nilpotent, it follows that $\Ncal_\hbar^1\subset \Ncal_\hbar^2$.

Let us check that $\Ncal_\hbar^2\subset \Ncal_\hbar^1$. Again, it is enough to show that
$\frac{1}{\tb}\ad(\mathbf{h})$ acts locally finitely on $\Ncal_\hbar^2$. It is enough to
show that $\frac{1}{\tb}\ad(\mathbf{h})$  acts locally finitely on the $\K^\times$-eigenspaces
in $\Ncal^2_\hbar$. But each such eigenspace embeds into $\M_{\heartsuit,b}$ under the $\hbar=1$ quotient
map. So it is enough to show that $\ad(\mathbf{h})$ is locally finite on $\M_{\heartsuit,b}$
or on any its finitely generated sub-bimodule. But any such sub-bimodule is Harish-Chandra
in the sense of Definition \ref{defi_HC2}  by  Proposition \ref{Prop:HC_coinc}. And hence $\ad(\mathbf{h})$
acts locally finitely.
\end{proof}

We remark that although we used a lifting $\K/\mathbb{Z}\hookrightarrow \K$ to construct a functorial
isomorphism $\M_{\dagger,b}\xrightarrow{\sim}\M_{\heartsuit,b}$, this isomorphism is independent
of the choice of $\iota$. Also we remark that the functor $\bullet_{\heartsuit,b}$ does not change
(up to an isomorphism) if we identify $\Hrm_{1,\bullet}^{\wedge_b}$ with $\CC(\underline{\Hrm}_{1,\bullet}^{\wedge_0})$ using the isomorphism $\vartheta^b$ instead of $\theta^b$. This follows from Lemma \ref{Lem:isom_comp}.

We can construct a functor $\bullet^{\heartsuit,b}:_{c}\!\HC(\underline{\Hrm}^+)_{c'}\rightarrow _{c}\!\widehat{\HC}(\Hrm)_{c'}$ in a similar way. Namely, let $\Ncal\in _{c}\!\HC(\underline{\Hrm}^+)_{c'}$.
Set $\Ncal'=(\theta^b_{*})^{-1}\left(\CC(W,\underline{W}, \mathcal{D}(\h^{\underline{W}\wedge_0})\otimes \Ncal^{\wedge_0})\right)$. For $\Ncal^{\heartsuit,b}$ take the subspace of all elements, where $\ad(a),\ad(b),a\in S(\h)^W, b\in S(\h^*)^W$ act locally nilpotently. As in the proof of Proposition \ref{Prop:3.15.1}, we see that
$\bullet^{\heartsuit,b}$ is right adjoint to $\bullet_{\heartsuit,b}$. Therefore $\bullet^{\heartsuit,b}=\bullet^{\dagger,b}$.

\subsection{Induction and restriction functors and the annihilators}\label{SUBSECTION_ResInd}
The main application of the isomorphism of completions theorem in \cite{BE} is the construction of
functors between the categories $\Str$ of appropriate Cherednik algebras.

Namely, consider the algebra $\Hrm_{1,c}$. By definition, the corresponding category $\Str$ is the full
subcategory of $\Hrm_{1,c}$-modules consisting of all modules $M$ satisfying the following two conditions:
\begin{itemize}
\item $M$ is finitely generated over $\K[\h]$ (or, equivalently, over $\K[\h/W]$).
\item $\h\subset \Hrm_{1,c}$ acts on $M$ by locally nilpotent endomorphisms.
\end{itemize}
Let $\underline{\Str}^+,\underline{\Str}$ denote the similar categories  for $\underline{\Hrm}_{1,c}^+,\underline{\Hrm}_{1,c}$.
Pick a point $b\in \h$ with $W_b=\underline{W}$. Following \cite{BE}, let us
construct certain functors $\operatorname{Res}_b:\Str\rightarrow \underline{\Str}^+,
\operatorname{Ind}_b:\underline{\Str}^+\rightarrow \Str$.

Recall the (partial) completion $\Hrm_{1,c}^{\wedge_b}:=\K[\h/W]^{\wedge_b}\otimes_{\K[\h/W]}
\Hrm_{1,c}$.
Consider the category $\Str^{\wedge_b}$ of all $\Hrm_{1,c}^{\wedge_b}$-modules that are finitely generated over
$\K[\h/W]^{\wedge_b}$. We have the completion functor $\Str\rightarrow \Str^{\wedge_b}, M\mapsto \K[\h/W]^{\wedge_b}\otimes_{\K[\h/W]}M$. There is an equivalence between $\Str^{\wedge_b}$
and $\underline{\Str}^+$ established in \cite{BE}.

Namely, recall the Bezrukavnikov-Etingof isomorphism
$\vartheta^b: \Hrm_{1,c}^{\wedge_b}\rightarrow \CC(\underline{\Hrm}^{\wedge_0}_{1,c})$, where
$\underline{\Hrm}_{1,c}^{\wedge_0}:=\K[\h/\underline{W}]^{\wedge_0}\otimes_{\K[\h/\underline{W}]}
\underline{\Hrm}_{1,c}$.
The module categories of $\CC(\underline{\Hrm}_{1,c}^{\wedge_0})$
and of $\underline{\Hrm}_{1,c}^{\wedge_0}$ are equivalent, an equivalence
$\rho:\CC(\underline{\Hrm}_{1,c}^{\wedge_0})$-$\operatorname{Mod}\rightarrow
\underline{\Hrm}_{1,c}^{\wedge_0}$-$\operatorname{Mod}$ is given by $M\mapsto e(\underline{W})M$.
So we have an equivalence $\rho\circ \vartheta^b_*:\Str^{\wedge_b}\rightarrow \underline{\Str}^{\wedge_0}$.
Now let us recall an equivalence $\underline{\Str}^{\wedge_0}\rightarrow \underline{\Str}^{+\wedge_0}$
from \cite{BE} sending a module $M\in \underline{\Str}^{\wedge_0}$ to the joint kernel of the operators
from $\h^{\underline{W}}$. Next, we have an equivalence $\underline{\Str}^{+\wedge_0}\rightarrow
\underline{\Str}^+$ sending a module $M\in \underline{\Str}^{+\wedge_0}$ to its subspace of all
vectors, where $\h_+$ acts locally nilpotently. Composing the equivalences we have just constructed
we get a required equivalence $\Str^{\wedge_b}\rightarrow \underline{\Str}^+$. Now the functor $\operatorname{Res}_b$
is obtained by taking the composition of the completion functor $\Str\rightarrow \Str^{\wedge_b}$
with the equivalence $\Str^{\wedge_b}\rightarrow \underline{\Str}^+$.

Let us recall the construction of the functor $\operatorname{Ind}_b:\underline{\Str}^+\rightarrow \Str$.
Consider the equivalence $\underline{\Str}^+\rightarrow \Str^{\wedge_b}$ that is inverse to the
equivalence $\Str^{\wedge_b}\rightarrow \underline{\Str}^+$. Then, as was checked in \cite{BE}, for
any module $M\in \Str^{\wedge_b}$ its subspace $M_{l.n.}$ consisting of all vectors, where $\h\subset \Hrm_{1,c}$
acts locally nilpotently, is a finitely generated $\Hrm_{1,c}$-module. So we have a functor
$\Str^{\wedge_b}\rightarrow \Str, M\mapsto M_{l.n.}$. Composing this functor with the equivalence
$\underline{\Str}^+\rightarrow \Str^{\wedge_b}$ we get the functor $\operatorname{Ind}_b$.

In a subsequent paper we will need the following result on annihilators.

\begin{Prop}\label{Prop:annihilators}
For $M\in \Str, N\in \underline{\Str}_+$ we have the inclusions $$\Ann_{\Hrm_{1,c}}(M)_\dagger\subset
\Ann_{\underline{\Hrm}^+_{1,c}}(\operatorname{Res}_b(M)),
\Ann_{\underline{\Hrm}^+_{1,c}}(N)^\dagger\subset \Ann_{\Hrm_{1,c}}(\operatorname{Ind}_b(N)).$$
\end{Prop}
\begin{proof}
The first inclusion follows from the isomorphism $\bullet_{\dagger,b}\cong \bullet_{\heartsuit,b}$
and the construction of $\bullet_{\heartsuit,b}$ (with $\vartheta^b$ instead of $\theta^b$, see the remarks
after the proof Proposition \ref{Prop:fun_iso}). The second inclusion follows similarly from the isomorphism
$\bullet^{\dagger,b}=\bullet^{\heartsuit,b}$.
%
%
\end{proof}

\subsection{Bifunctor $L(\bullet,\bullet)$}\label{SUBSECTION_loc_fin}
Let $M,N$ be modules over the algebras $\Hrm_{1,c'},\Hrm_{1,c}$, respectively.
The space $\Hom_{\K}(M,N)$ is an $\Hrm_{1,c}$-$\Hrm_{1,c'}$-bimodule. Consider the subspace
$L(M,N)\subset \Hom_{\K}(M,N)$ of all vectors that are locally nilpotent
with respect to the operators $\ad(a),\ad(b)$
for all $ a\in (S\h^*)^W, b\in (S\h)^W$. An analogous construction is very
important in the representation theory of semisimple Lie algebras.

From the definition it follows that the bimodule $L(M,N)$ is the direct limit of its
Harish-Chandra sub-bimodules. The following proposition, which is the main result of this subsection, implies that $L(M,N)$
is Harish-Chandra itself.

\begin{Prop}\label{Prop:L_fin_gen}
For all $M,N$ in the category $\mathcal{O}$ the bimodule $L(M,N)$ is finitely generated.
\end{Prop}

The proof is based on the following two claims.

\begin{Lem}\label{Cor:support}
Let $M,N$ be simples in $\mathcal{O}$ with different supports. Then $L(M,N)=0$.
\end{Lem}

We remark that the support of a simple in $\Str$ (considered as a $\K[\h/W]$-module) is irreducible,
see \cite{Ginzburg_irr}, Theorem 6.8.

\begin{Prop}\label{Prop:strong_adjoint}
Let $M$ be a module in the category $\mathcal{O}$ for $\Hrm_{1,c'}$, and $\underline{N}$ be a
finite dimensional module
in the category $\mathcal{O}$ for $\underline{\Hrm}^+_{1,c}$, where $b$ is chosen in the
open stratum in the support of $M$ so that $\Res_b(M)$ is finite dimensional. Then there is a (bi)functorial
isomorphism $$L(M,\Ind_b(\underline{N}))\cong \Hom_\K(\Res_b(M),\underline{N})^{\dagger,b}.$$
\end{Prop}


\begin{proof}[Proof of Lemma \ref{Cor:support}]
Suppose $L(M,N)\neq 0$.

Let $I$ be the annihilator of $M$ in $(S\h^*)^W$. Pick $\varphi\in L(M,N)$.
Then the image of $\varphi$ consists of vectors that are locally nilpotent for all $a\in I$.
It follows that $\sum_{\varphi}\im\varphi$ is a submodule in $N$ that is annihilated
by $I^n$ for $n\gg 0$. Since $N$ is simple, this means
that $N$ itself is annihilated by $I^n$. So we see that
$\operatorname{Supp}(N)\subset \operatorname{Supp}(M)$.

Now let $J$ be the annihilator of $N$ in $(S\h^*)^W$. Then for any $\varphi\in L(M,N)$ the inclusion
$J^n M\subset \ker\varphi$ holds for $n\gg 0$. Pick a nonzero finitely generated sub-bimodule $X\subset L(M,N)$.
Then $X$ is a Harish-Chandra bimodule, and so is finitely generated as a left $\Hrm_{1,c}$-module.
Let $\varphi_1,\ldots,\varphi_k$ be generators. Then $\bigcap_{i=1}^k \ker\varphi_i\subset\ker\varphi$
for any $\varphi\in X$. So $\bigcap_{i=1}^k \ker\varphi_i$ is an $\Hrm_{1,c'}$-submodule in
$M$. Since $M$ is simple, we have $\bigcap_{i=1}^k \ker\varphi_i=\{0\}$. But now $J$ acts locally
nilpotently on $M$, which implies $\operatorname{Supp}(M)\subset \operatorname{Supp}(N)$.
\end{proof}

\begin{Rem}\label{Rem:L_diff_supp}
One can strengthen the claim of Lemma \ref{Cor:support} (with the same proof) as follows.
Let $M$ be a simple module and let $N_0$ be the maximal submodule of $N$ supported
on the support of $M$. Then the natural homomorphism $L(M,N_0)\rightarrow L(M,N)$
is an isomorphism. Analogously, let $N$ be simple, and $M_0$ be the maximal quotient
of $M$ supported on the support of $N$. Then $L(M,N)\cong L(M_0,N)$.
\end{Rem}

\begin{proof}[Proof of Proposition \ref{Prop:strong_adjoint}]
Recall the category $\Str^{\wedge_b}$ considered in Subsection
\ref{SUBSECTION_ResInd}, the completion functor
$\bullet^{\wedge_b}: \Str\rightarrow \Str^{\wedge_b}$ and the functor
$\bullet_{l.n.}:\Str^{\wedge_b}\rightarrow \Str$ of taking locally
nilpotent vectors.

{\it Step 1.} Let $M\in \Str_{c'}, M'\in \Str^{\wedge_b}_c$.
We claim that the bimodules $L(M,M'_{l.n.})$ and $L(M^{\wedge_b},M')$
are naturally identified.

Pick $\varphi\in L(M,M'_{l.n.})$. We can view $\varphi$ as an element of
$L(M,M')$. Since the action of $\ad(a), a\in (S\h^*)^W$ on $L(M,M')$ is locally
nilpotent, we see that $\varphi$ is continuous in the $b$-adic topology.
So $\varphi$ uniquely extends by continuity to $M^{\wedge_b}$. The extension
$\psi_\varphi$ obviously lies in $L(M^{\wedge_b},M')$.

Conversely,  take $\psi\in L(M^{\wedge_b},M')$ and consider the composition
of $\psi$ with the natural map $M\rightarrow M^{\wedge_b}$. The composition $\varphi_\psi$
lies in $L(M,M')$. Since $\ad(b), b\in (S\h)^W$, acts locally nilpotently on
$L(M,M')$, we see that the image of $\varphi_\psi$ lies in $M'_{l.n.}$. So we can view
$\varphi_\psi$ as an element of $L(M,M'_{l.n.})$. The maps $\varphi\mapsto \varphi_\psi, \psi\mapsto
\psi_\varphi$ as mutually inverse.

{\it Step 2.} Now for $M'$ take the image of $N$ under the equivalence $\underline{\Str}^+\xrightarrow{\sim}
\Str^{\wedge_b}$. We claim that the bimodules $L(M^{\wedge_b},M')$ and $\Hom(\Res_b(M),N)^{\dagger,b}$
are naturally identified. This will complete the proof.

First of all, let us embed $\Hom(\Res_b(M),N)^{\dagger,b}$ into $\Hom_{\K}(M^{\wedge_b},M')$. For this recall,
that, by the discussion after Proposition \ref{Prop:fun_iso}, $\Hom(\Res_b(M),N)^{\dagger,b}$ is the same as
$$(\theta^b_{*})^{-1}[\Dcal(\h^{\underline{W}})^{\wedge_0}\otimes \CC(\Hom(\Res_b(M),N))]_{l.n.}$$
On the other hand,
\begin{align*}
\theta^b_{*}(M^{\wedge_b})=&\K[[\h^{\underline{W}}]]\otimes \Fu_{\underline{W}}(W,\Res_b(M)),\\
\theta^b_{*}(M')=&\K[[\h^{\underline{W}}]]\otimes \Fu_{\underline{W}}(W,N).
\end{align*}
So $\Hom(M^{\wedge_b},M')=\Hom(\K[[\h^{\underline{W}}]],\K[[\h^{\underline{W}}]])\otimes \CC(\Hom(\Res_b(M),N))$
and the natural action of $\Dcal(\h^{\underline{W}})^{\wedge_0}$ on $\K[[\h^{\underline{W}}]]$ defines
an embedding $\Hom(\Res_b(M),N)^{\dagger,b}\hookrightarrow \Hom(M^{\wedge_b},M')$. We need to check
that the Harish-Chandra (=locally nilpotent) part of $\Hom(M^{\wedge_b},M')$ lies in $\Hom(\Res_b(M),N)^{\dagger,b}$.
For this we notice that $\ad(a)$ acts locally nilpotently on $L(M^{\wedge_b},M')$ for any
$a\in \K[\h/W]^{\wedge_b}$. In particular, the adjoint action of $\K[\h^{\underline{W}}]\subset \theta^b_{*}(\K[\h/W]^{\wedge_b})$ on $\theta^b_{*}(L(M^{\wedge_b},M'))$ is locally nilpotent.
But it is well-known that the locally nilpotent part of
$\Hom(\K[[\h^{\underline{W}}]],\K[[\h^{\underline{W}}]])$ is nothing else but
$\Dcal(\h^{\underline{W}})^{\wedge_0}$. This completes the proof.
\end{proof}

\begin{proof}[Proof of Proposition \ref{Prop:L_fin_gen}]
Let us remark that the bifunctor $L(\bullet,\bullet)$ is left exact.
So it is enough to prove that $L(M,N)$ is finitely generated for simple
$M,N$. Thanks to Lemma \ref{Cor:support}, here we only need to consider
the case when $M,N$ have the same support. Pick $b$ in the open stratum
of the support. Since $\Ind_b$ is a right adjoint
for $\Res_b$, we have a natural homomorphism $N\rightarrow \Ind_b\circ\Res_b(N)$.
Since $N$ is simple, this homomorphism is an inclusion, and so $L(M,N)\hookrightarrow
L(M, \Ind_b\circ \Res_b(N))$. According to Proposition \ref{Prop:strong_adjoint},
the latter bimodule is $L(\Res_b(M),\Res_b(N))^{\dagger,b}$. It is finitely generated
by Remark \ref{Rem:fin_gen}. Being a sub-bimodule in a finitely generated
bimodule, $L(M,N)$ is finitely generated as well.
\end{proof}

\subsection{Ideals in the rational Cherednik algebra of type $A$}\label{SUBSECTION_type_A}
In this section we will describe the structure of the set of two-sided ideals
in the rational Cherednik algebra  $\Hrm_{1,c}$ of type $A$. ``Type A'' means that
$\Gamma=S_n$ is a symmetric group in $n$ letters, $V=\h\oplus \h^*$, where $\h\cong \K^{n-1}$
is the reflection $\Gamma$-module. In this case $c$ is a single number. It is well known,
see, for example \cite{BEG}, that $\Hrm_{1,c}$ has a finite dimensional module
if and only if $c=\frac{r}{n}$, where $n\in \mathbb{Z}_{>1}$  and $r$ is an integer mutually prime with
$n$. In the latter case there is a unique (up to an isomorphism) indecomposable finite dimensional module.

The following theorem describes the structure of the set $\Id(\Hrm_{1,c})$.

\begin{Thm}\label{Thm:appl2}
\begin{enumerate}
\item The algebra $\Hrm_{1,c}$ is not simple if and only if $c=\frac{q}{m}$, where $q,m$
are mutually prime integers, and $1<m\leqslant n$. Below we suppose that $c$ has this form.
\item Set $s=[n/m]$.  The set of proper two-sided ideals of $\Hrm_{1,c}$ consists of $s$ elements, say $\J_i,i=1,\ldots,s$, and $\J_1\supsetneq \J_2\supsetneq\ldots\supsetneq \J_s$.
\item $\VA(\J_j)=\pi(V^{*\underline{\Gamma}_j})$, where $\underline{\Gamma}_j=S_m^{\times j}\hookrightarrow S_n$.
\end{enumerate}
\end{Thm}

Part (1) of this theorem as well as the fact that the associated varieties of two-sided ideals in
 $\Hrm_{1,c}$ have the form described in (3) seem to be known before, although we could not find a precise reference
(for example, these facts are easily proved using the claim on finite dimensional representations
and techniques from \cite{BE}). The statement of part (2) was communicated to me by Etingof.

\begin{proof}
{\it Step 1.} Let $\J$ be a proper two-sided ideal of $\Hrm_{1,c}$. Choose a symplectic leaf $\Leaf$
in $V^*/\Gamma$ that is open in $\VA(\J)$. Let $\underline{\Gamma},\widetilde{\Xi},\underline{\Hrm}_{1,c}$ be such as in Subsection \ref{SUBSECTION_1c}. The subgroup $\underline{\Gamma}$ has the form $\prod_{i=1}^j S_{n_i}$, where $n=n_1+\ldots+n_j$. Then $\J_\dagger$ is a $\widetilde{\Xi}$-stable ideal of finite codimension in $\underline{\Hrm}_{1,c}^+$. The algebra $\underline{\Hrm}_{1,c}^+$ is the tensor product of the algebras $\Hrm_{1,c}(\K^{n_i-1},S_{n_i})$ (we suppose that for $n_i=1$ the latter algebra is $\K$).
Each of these algebras has a finite dimensional representation. From the discussion preceding Theorem \ref{Thm:appl2}  we see that $n_i$ is either some $m>1$ or $1$, and $c=\frac{q}{m}$ with $\operatorname{GCD}(q,m)=1$. In particular, $\underline{\Gamma}=S_m^{\times j}$, and the group $\Xi$ is naturally identified with
$S_j\times S_{n-mj}$. There is a section $\Xi\hookrightarrow \widetilde{\Xi}$ of the projection $\widetilde{\Xi}\rightarrow \Xi$, the corresponding subgroup $S_j\subset \widetilde{\Xi}$  permutes the tensor factors of $\underline{\Hrm}_{1,c}$. Set $\A:=\Hrm_{1,c}(\K^{m-1},S_m)$.

{\it Step 2.} For $j=1,\ldots,s(=[n/m])$  let $\Leaf_j$ be the open symplectic
leaf in $\pi(V^{*\underline{\Gamma}_j})$, where $\underline{\Gamma}_j$ is as in part (3).

We claim that there is a unique maximal element $\J_j$ in $\Id_{\Leaf_j}(\Hrm_{1,c})$ and that $\VA(\J_j/\J)\subset \partial\Leaf_j(=\cup_{i>j}\Leaf_i)$ for any $\J\in \Id_{\Leaf_j}(\Hrm_{1,c})$.

Let us prove this claim. If $\J\in \Id_{\Leaf_j}(\Hrm_{1,c})$,  then $\J_\dagger$ is an ideal of finite
codimension in the corresponding slice algebra $\underline{\Hrm}_{1,c}^+=\A^{\otimes j}$.
But the latter algebra has a unique maximal ideal of finite codimension. Indeed, if $I$ is such an ideal,
then the kernels of all homomorphisms $\A\rightarrow \A^{\otimes j}/I$ coincide with a unique ideal of
finite codimension in $\A$ (the uniqueness of such an ideal follows from the fact that $\A$
has a unique indecomposable finite dimensional module). Now the claim follows from the observation that the tensor product of the matrix
algebras is again a matrix algebra and so is simple.

So the ideal $(\J_\dagger)^\ddag\in \Id_{\Leaf_j}(\Hrm_{1,c})$ does not depend on $\J$ and contains $\J$.
By Theorem \ref{Thm:4}, $(\J_\dagger)^\ddag/\J$ is supported on $\partial\Leaf_j$. It remains to put
$\J_j:=(\J_\dagger)^\ddag$.

{\it Step 3.} Fix some $j$. On this step we are going to describe the set of proper $S_j$-stable ideals in $\A^{\otimes j}$. We can apply the results of Step 2 to $\A$ to see that $\A$ has only one proper two-sided ideal,
say $\n$. This ideal has finite codimension.
We claim that any $S_j$-stable  ideal in $\A^{\otimes j}$ is of the form
$\J_p(j)=\sum_{i_1<\ldots<i_p}\n_{i_1}\ldots \n_{i_p}, p=1,\ldots,j,$ where $\n_i:=\A\otimes\ldots\otimes \n\otimes \ldots\otimes \A$, with $\n$ on the $i$-th place.

First,
we claim that any {\it prime}=(primitive) ideal in $\A^{\otimes j}$ has the form $\J_I:=\sum_{i\in I}\n_i$ for some
$I\subset \{1,\ldots,j\}$. First of all, being the tensor product of a matrix algebra and of several
copies of $\A$, the algebra $\A^{\otimes j}/\J_I$ is prime. Now let $\J$ be  a prime ideal
 in $\A^{\otimes j}$ and $\Leaf'$ be the open symplectic
leaf in $\VA(\A^{\otimes j}/\J)\subset V^+/\underline{\Gamma}=(\K^{2m-2}/S_m)^{\times j}$. The algebra
$\A^{\otimes j}$ is an SRA so we can use an argument of Step 1. We see that $\overline{\Leaf}'$ is the product
of $0$'s or $\K^{2m-2}/S_m$'s and moreover $\J=\I^{\dagger}$, where $\I$ is a unique ideal of finite codimension
in the slice algebra (which is again the tensor product of several copies of $\A$). In particular, we have a
unique prime ideal with a given associated variety and this prime ideal must be  some $\J_I$.
This completes the proof of the claim in the beginning of the paragraph.


Now let $\J$ be an arbitrary $S_j$-stable ideal in $\underline{\Hrm}_{1,c}$. Let $\J_I, I\in F,$ be all
minimal prime ideals of $\J$, where $F$ is some (finite) index set. Then $F$ is $S_j$-stable. We claim that
$\bigcap_{\sigma\in S_j}\J_{\sigma.I}=\J_p(j)$, where $p=j-\#\I+1$.

The proof is by the induction on $j$.  Set $$\J^1:=\bigcap_{\sigma\in S_j, j\not\in \sigma.I}\J_{\sigma.I}, \J^2:=\bigcap_{\sigma\in S_j, j\in \sigma.I}\J_{\sigma.I}.$$ Then, using the inductive
assumption, we get $\J^1=\J_{p-1}(j-1)\otimes \A$ and $\J^2=\J_{p}(j-1)\otimes \A+\n_j$. The first equality is trivial,
let us explain why the last one holds. It is clear that $\J^2\supset \n_j$. Using the inductive
assumption, we see that $\J^2/\n_j=\J_{p}(j-1)\otimes (\A/\n)$. The equality for $\J^2$ follows.

 Since $\J_{p-1}(j-1)\otimes\A\cap \n_j=\J_{p-1}(j-1)\otimes \n$,  we see that
$\J^1\cap \J^2=\J_p(j-1)\otimes \A+ ([\J_{p-1}(j-1)\otimes \A]\cap \n_j)=\J_p(j)$.

It follows from the previous paragraph that the radical of any $S_j$-stable ideal in $\underline{\Hrm}_{1,c}$
coincides with some $\J_p(j)$ (we remark that $\J_1(j)\supset \J_2(j)\supset\ldots\supset \J_j(j)$).
Note however, that $\J_p(j)^2=\J_p(j)$, because $\n^2=\n$. So we have proved that any proper $S_j$-stable ideal in
$\underline{\Hrm}_{1,c}$ coincides with some $\J_p(j)$. Moreover, we remark that $\J_p(j)\J_q(j)=\J_{\max(p,q)}(j)$.

{\it Step 4.}
Consider the (smallest) symplectic leaf
$\Leaf_1$  and the algebra $\underline{\Hrm}_{1,c}$ corresponding to this leaf. By  assertion (4) of Proposition \ref{Prop:3.10.1},
$\J_\dagger/\J_j(s)=0$ provided $\VA(\J)=\overline{\Leaf}_j$. Also note that $\J\subset\J', \J_\dagger=\J'_\dagger$
implies $\J=\J'$.  Indeed, $\VA(\J'/\J)\subset \partial\Leaf_1$ and so $\J'/\J=0$.

Let us check that $\J_{i_1}\J_{i_2}=\J_{\max(i_1,i_2)}$.  Consider the ideals $\J_{i_1\dagger},\J_{i_2\dagger}$. The ideals $\J_{i_1\dagger},\J_{i_2\dagger}$
are proper $S_s$-stable ideals in $\A^{\otimes s}$ and so are of the form $\J_{i_1}(s),\J_{i_2}(s)$.
Also we remark that $(\J_{i_1}\J_{i_2})_\dagger=\J_{i_1\dagger}\J_{i_2\dagger}$. It follows that
$(\J_{i_1}\J_{i_2})_\dagger=\J_{\max(i_1,i_2)\dagger}$ and so, by the previous paragraph, $\J_{i_1}\J_{i_2}=\J_{\max(i_1,i_2)}$.  In particular, it follows that
$\J_1\supsetneq \J_2\supsetneq\ldots\supsetneq \J_s$.

It remains to check that any ideal $\J$ of $\Hrm_{1,c}$
coincides with some $\J_i$. Let $\VA(\J)=\overline{\Leaf}_j$. Then $\J\subset\J_j$ by Step 2.
However, $\J_\dagger=\J_j(s)=\J_{j\dagger}$ and hence $\J=\J_j$.
\end{proof}


\begin{thebibliography}{99}
\bibitem[B]{Bellamy} G. Bellamy. {\it Cuspidal representations of rational Cherednik algebras
at $t=0$}. Preprint, arXiv:0911.0069. To appear in Math. Z.
\bibitem[BEG1]{BEG} Yu. Berest, P. Etingof, V. Ginzburg. {\it Finite dimensional representations
of rational Cherednik algebras}.  Int. Math. Res. Not. 19(2003),  1053-1088.
\bibitem[BEG2]{BEG_HC} Yu. Berest, P. Etingof, V. Ginzburg. {\it Cherednik algebras
and differential operators on quasi-invariants}.  Duke Math. J. 118(2003),  279-337.
\bibitem[BE]{BE} R. Bezrukavnikov, P. Etingof. {\it Parabolic induction and restriction
functors for rational Cherednik algebras}.  Selecta Math.,  14(2009), 397-425.
\bibitem[BeKa]{BeKa} R. Bezrukavnikov, D. Kaledin. {\it Fedosov quantization in positive characteristic}.
J. Amer. Math. Soc. 21(2008), 409-438.
\bibitem[BoKr]{BoKr} W. Borho, H. Kraft. {\it \"{U}ber die
Gelfand-Kirillov-Dimension}. Math. Ann. 220(1976), 1-24.
\bibitem[BrGo]{BG} K.A. Brown, I. Gordon. {\it Poisson orders, representation theory and
symplectic reflection algebras}. J. Reine Angew. Math. 559(2003), 193-216.
\bibitem[D]{Dixmier} J. Dixmier. {\it Enveloping algebras}. North-Holland Mathematical Library, Vol. 14.
\bibitem[Ei]{Eisenbud} D. Eisenbud {\it Commutative algebra with a view
towards algebraic geometry}. GTM 150, Springer Verlag, 1995.
\bibitem[Et]{Etingof} P. Etingof. {\it Calogero-Moser systems and representation theory}.
Z\"{u}rich Lectures in Advanced Mathematics. EMS, Z\"{u}rich, 2007.
\bibitem[EG]{EG} P. Etingof, V. Ginzburg. {\it Symplectic reflection algebras, Calogero-Moser space,
and deformed Harish-Chandra homomorphism}. Invent. Math. 147 (2002), N2, 243-348.
\bibitem[ES]{ES} P. Etingof, T. Schedler. {\it Poisson traces and D-modules on Poisson varieties}.
GAFA 20(2010), n.4,  958-987.
\bibitem[G1]{Ginzburg_char} V. Ginzburg. {\it Characteristic varieties and vanishing cycles}.
Inv. Math. 84(1986), 327-402.
\bibitem[G2]{Ginzburg_irr} V. Ginzburg. {\it On primitive ideals}. Selecta Math., new series, 9(2003), 379-407.
\bibitem[HHT]{HHT} R. Hotta, K. Takeuchi, T. Tanisaki. {\it D-modules, perverse sheaves and representation
theory}. Progress in Mathematics, v. 236, Birkh\"{a}user, 2008.
\bibitem[Jo]{Joseph} A. Joseph. {\it On the associated variety of a primitive ideal}. J. Algebra,
93(1985), 509-523.
\bibitem[K]{Kremnitzer} K. Kremnizer. {\it Proof of the De Concini-Kac-Procesi conjecture}.
Preprint, arXiv:math/0611236.
\bibitem[Lo1]{Wquant} I.V. Losev. {\it Quantized symplectic actions and
$W$-algebras}. J. Amer. Math. Soc. 23(2010), 35-59.
\bibitem[Lo2]{HC} I. Losev. {\it Finite dimensional representations of W-algebras}.  
Duke Math J. 159(2011), n.1, 99-143.
\bibitem[Lo3]{ES_appendix} I. Losev. Appendix to \cite{ES}.
\bibitem[M]{Martino} M. Martino. {\it The Associated variety of a Poisson Prime Ideal}. J. London Math. Soc. (2),
72(2005), 110-120.
\bibitem[MCR]{MCR} J.C. McConnel, J.C. Robson. {\it Noncommutative noetherian rings}. With the cooperation of L. W. Small. 2nd ed, Graduate texts in Math. 30, AMS 2001.
\bibitem[P]{Premet} A. Premet. {\it Irreducible representations of Lie algebras of reductive groups and the Kac-
Weisfeiler conjecture.} Invent. Math. 121 (1995), 79-117.
\end{thebibliography}
\end{document}